\documentclass[a4paper,10pt]{amsart} 
\usepackage{amsmath, amssymb, esint, comment, centernot}
\usepackage{amscd}
\usepackage{shadow, epsfig, color}
\usepackage{graphicx}
\usepackage{url, ulem}
\usepackage{mathrsfs}
\usepackage{microtype}
\usepackage{fancybox}
\usepackage{xcolor,cancel}
\usepackage{ulem}
\usepackage{diagbox}
\usepackage{makecell}
\usepackage{tikz, venndiagram}
\usetikzlibrary{shapes,backgrounds}
\usepackage{blindtext}
\usepackage{setspace, enumerate}

\usepackage[T1]{fontenc}
\usepackage{lmodern}

\definecolor{Burgundy}{RGB}{130,20,0}
\definecolor{OliveGreen}{RGB}{53,110,45}

\usepackage[ 
colorlinks=true,
linkcolor=black,
filecolor=black,
urlcolor=black,
citecolor=black
]{hyperref}

\newcommand{\N}{\mathbb N}

\newcommand{\R}{\mathbb R}

\newcommand{\e}{\varepsilon}

\renewcommand{\qed}{\ \hfill \fbox{} \bigskip}

\newcommand{\E}{\mathscr E}
\newcommand{\EE}{\mathcal E}

\newcommand{\V}{\mathcal V}
\newcommand{\T}{\mathcal T}
\newcommand{\DD}{\mathscr D}
\newcommand{\D}{\mathcal D}
\newcommand{\DDD}{\mathsf D}

\newcommand{\p}{\overline{p}}

\renewcommand{\tilde}{\widetilde}

\newcommand{\Cyl}{{\sf Cyl}_b^\infty(C^\infty_*(\R^n))}
\renewcommand{\Cyl}{{\rm CylF}(\U)}
\newcommand{\CylF}{{\rm CylF}}
\newcommand{\CylV}{{\rm CylV}}
\newcommand{\ECylF}{{\rm ECylF}}

\newcommand{\U}{\Upsilon}

\renewcommand{\p}{\pi_{\sf m}}

\newcommand{\dg}{{\rm diag}}
\newcommand{\s}{\mathsf s}

\renewcommand{\p}{\pi}
\renewcommand{\d}{{\mathsf d}}
\renewcommand{\S}{\mathsf S}

\renewcommand{\SS}{\mathbf S}

\renewcommand{\c}{\mathsf c}

\renewcommand{\.}{\ldots}
\numberwithin{equation}{section}

\newtheorem{thm}{Theorem}[section]
\newtheorem*{thm*}{Theorem}
\newtheorem{defn}[thm]{Definition}
\newtheorem{lem}[thm]{Lemma}
\newtheorem{rem}[thm]{Remark}
\newtheorem{cor}[thm]{Corollary}

\newtheorem{prop}[thm]{Proposition}

\DeclareRobustCommand{\Sec}{\ifmmode\mathsection\else\textsection\fi}

\makeatletter
\@namedef{subjclassname@2020}{%
  \textup{2020} Mathematics Subject Classification}
\makeatother

\renewcommand{\tilde}{\widetilde}

\renewcommand{\paragraph}[1]{\medskip{\noindent{\it #1}. }}

\makeatletter   
\def\@fnsymbol#1{\ensuremath{\ifcase#1\or *\or \mathsection \or  \mathparagraph \or \dagger\or
    \ddagger \or \|\or **\or \dagger\dagger
   \or \ddagger\ddagger \else\@ctrerr\fi}}
 \makeatother 

\makeatletter
\newcommand\thankssymb[1]{\textsuperscript{\@fnsymbol{#1}}}
\makeatother

\makeatletter
\def\l@subsection{\@tocline{2}{0pt}{2.5pc}{5pc}{}}
\makeatother

\title{BV Functions and Sets of Finite Perimeter on Configuration Spaces}
\author[E.~ Bru\'e]{Elia Bru\'e\thankssymb{1}}
\author[K.~Suzuki]{Kohei Suzuki\thankssymb{2}\thankssymb{3}}
\thanks{\thankssymb{1}Department of Decision Sciences, Bocconi University, Italy.~Email:\phantom{\thankssymb{1}} elia.brue@unibocconi.it}
\thanks{\thankssymb{2} Department of Mathematical Sciences, Durham University, United Kingdom. Email:\phantom{\thankssymb{2}} kohei.suzuki@durham.ac.uk}
\thanks{\thankssymb{3}Theoretical Sciences Visiting Program, Okinawa Institute of Science and Technology Graduate University, Onna, 904-0495, Japan}


\begin{document}

\maketitle
\begin{center}
\today
\end{center}
\begin{abstract}
\noindent  In this paper, we aim to develop the foundations of a theory of BV functions in the configuration space over the Euclidean space $\R^n$ equipped with the Poisson measure $\p$. 
We first construct the $m$-codimensional Poisson measure ---formally written as ``$(\infty-m)$-dimensional Poisson measure''--- on the configuration space. We then show that our construction is consistent with potential theory induced by the infinitely many independent Brownian motions by establishing relations between the $m$-codimensional Poisson measure and Bessel capacities. Secondly, we introduce three different definitions of BV functions based on the variational, relaxation, and semigroup approaches, and prove the equivalence of them.  In the process, we prove the $p$-Bakry--\'Emery inequality on the configuration space for any $1<p<\infty$. Thirdly, we construct perimeter measures and introduce an appropriate notion of measure-theoretic boundary, called the reduced boundary. We then prove that the perimeter measure can be expressed by the $1$-codimensional Poisson measure restricted to the reduced boundary, which is a generalisation of De Giorgi's identity to the configuration space.  Finally, we construct the total variation measures for functions of bounded variation, and prove the Gau\ss--Green formula. 
\end{abstract}

\tableofcontents

\section{Introduction}
%\subsection{Overview}
%\footnote{Do we mention the case of the configuration space over manifolds?\\EB: why? It seems unnecessary to me}
The purpose of this paper is to establish the foundations for functions of bounded variations (BV functions) in the space of all locally finite point measures (without multiplicity) in the Euclidean space $\R^n$, denoted by $\U(\R^n)$ and called the {\it configuration space}.
%More precisely, we study three different notions of {\it functions with bounded variation} showing their equivalence and a {\it Gau\ss--Green formula}. We study fine properties of {\it sets with finite perimeter}, showing that the so-called {\it perimeter measure} is represented in terms of a newly defined {\it $1$-codimensional measure} restricted on the {\it essential boundary}.
%
%\medskip
%
%
%The configuration space over $\R^n$ is defined as
%$$
%\U(\R^n):=\biggl\{\gamma = \sum_{i=1}^{N} \delta_{x_i}: \gamma(K) <\infty, \ \text{$\forall K \Subset \R^n$}, \ \gamma(\{x\}) \in \{0,1\},\ \forall x \in \R^n, \ N \in \N \cup\{\infty\} \biggr\} \, .
%$$
The space $\U(\R^n)$ is endowed with the vague topology $\tau_v$, the $L^2$-transportation (extended) distance ${\sf d}_\U$, which stems from the optimal transport problem, and the Poisson measure $\p$ whose intensity measure is the Lebesgue measure $\mathbf L^n$ on $\R^n$.
The resulting {\it topological (extended) metric measure structure} $(\U(\R^n), \tau_v, {\sf d}_\U, \p)$ plays a fundamental role to describe dynamical systems of infinite particles stemming from statistical physics, random point processes, random graphs and integral geometry, representation theory of diffeomorphism groups on manifolds, and many others. Instead of giving enormous numbers of related references here, we refer the reader to \cite[Section1.6]{DS21} for an overview of the aforementioned subjects. 
%see e.g., \cite{AKR98, C19, DS21, EH15, GGV75, KLR02, Osa96} (see also references therein)}. 
%Hence it has been extensively investigated in the last years
%by a number of research groups in the fields of statistical physics, probability, and representation theory.
%In spite of its importance for the applications to statistical physics and stochastic analysis, the theory of BV functions and sets of finite perimeters on $\U(\R^n)$ has -- so far -- never been investigated in any literature that we are aware of, due to its nature of infinite-dimenisonality and non-linearity. %which are explained in more detail below. 
%This paper is the first paper constructing theory of BV functions and sets of finite perimeters in  {\it infinite-dimensional} and {\it non-linear} spaces.  

\medskip

The studies of BV functions and sets of finite perimeter beyond the standard Euclidean space have seen a thriving development in the last years, see \cite{AlBaChRoShTe1,AlBaChRoShTe2,AlBaChRoShTe3,Am01,Am02,ABS,AmbDiM14,AmbGheMa,BPS,MirandaJr} and references therein.
However, all of these results do not cover the configuration space $\U(\R^n)$. 
%due to its intrinsic infinite dimensionality and the fact that the canonical (extended) set distance induced by ${\sf d}_\U$ can be infinite between sets of $\p$-positive measure. 
The space~$(\U(\R^n), \tau_v, {\sf d}_\U, \p)$ is known to possesses several pathological properties (see details in \cite{DS21}): 
\begin{itemize}
\item the extended distance ${\sf d}_\U$ is not continuous with respect the topology $\tau_v$; 
\item ${\sf d}_\U$-metric balls are negligible with respect to the Poisson measure $\p$; 
\item ${\sf d}_{\U}$--Lipschitz functions are not necessarily $\p$-measurable;
\item the Riesz--Markov--Kakutani representation theorem does not hold.
\end{itemize}
For these reasons, the study of the configuration space $(\U(\R^n), \tau_v, {\sf d}_\U, \p)$ does not fall into the standard framework of metric measure geometry. Furthermore, the lack of the Riesz--Markov-Kakutani's representation theorem causes further complexity to construct total variation measures supporting the Gau\ss--Green formula by means of standard functional-analytic technique.

%{\color{purple}
%There is another stream of geometric measure theory based merely on abstract Dirichlet form structures \purple{(references)}. 
%However, It seems to be not suitable for the study of finite codimensional measures and fine properties of sets with finite perimeter. Indeed the latter require several geometric information of the underlying space beyond the scope of the abstract theory.  
%}

\medskip

In the setting of infinite-dimensional spaces, the study of geometric measure theory has been pioneered by Feyel--de la Pradelle \cite{F92}, Fukushima \cite{Fuk01}, Fukushima--Hino \cite{FH01} and Hino \cite{H10} in the Wiener space. In \cite{F92}, they constructed the finite-codimensional Gau\ss--Hausdorff measure in the Wiener space and investigated its relation to capacities. In \cite{Fuk01} and \cite{FH01}, they developed the theory of functions of bounded variation and constructed perimeter measures, and prove the Gau\ss--Green formula. Based on these results,  Hino introduced in \cite{H10} a notion of reduced boundary and investigated relations between the one-codimensional Hausdorff--Gau\ss\ measure and the perimeter measures. Further fine properties were investigated by Ambrosio--Figalli \cite{AmbFig}, Ambrosio--Figalli--Runa \cite{AmbFigRuna}, Ambrosio--Miranda--Pallara \cite{AmbMiPa1,AmbMiPa2}, Ambrosio--Maniglia--Miranda--Pallara \cite{AMMP10}. The notion of functions of bounded variation has been studied also in a Gelfand triple by R\"ockner--Zhu--Zhu \cite{RZZ10, RZZ12, RZZ15}. All of the aforementioned results rely heavily on the linear structure of the Wiener space or the Hilbert space, which is used to perform finite-dimensional approximations. However, the configuration space does not have a linear structure and there is no chance to apply similar techniques.

\subsection{Non-linear dimension reduction and overview of the main results}

To overcome the difficulties explained above, we develop a {\it non-linear dimensional reduction} tailored to the configuration space $\U(\R^n)$.
A key observation is that $\U(B_r)$, the configuration space over the Euclidean closed metric ball $B_r$ centred at the origin $o$ with radius $r>0$, is essentially finite dimensional. More precisely, due to the compactness of $B_r$, $\U(B_r)$ can be written as the disjoint union $\sqcup_{k \in \N} \U^k(B_r)$ of the $k$-particle configuration spaces $\U^k(B_r)$, each of which is isomorphic to the quotient space of the $k$-product space $B_r^{\times k}$ by the $k$-th symmetric group. 
In light of this observation, the main task is to lift geometric measure theory on $\U(B_r)$ to the infinite-dimensional space $\U(\R^n)$ by finite-dimensional approximations.
%\footnote{Do we want to mention that $\U^k(B_r)$ is RCD?\\EB: I think it is not necessary} %the method of the non-linear dimension reduction.

\smallskip
In this paper, we first construct the $m$-codimensional Poisson measure on the configuration space (Theorem \ref{thm: MI} and Definition \ref{defn: MCH}), and
study its relation to $(1,p)$-Bessel capacities (Theorem \ref{thm: CH}). Secondly, we introduce three different definitions of functions of bounded variation based on the variational, relaxation and the semigroup approaches, and prove their equivalence (Theorem \ref{prop: IBV}). In the process of showing the equivalence of these three definitions, we prove the $p$-Bakry--\'Emery inequality (Theorem \ref{rem: BE}) for the heat semigroup on $\U(\R^n)$ for $1<p<\infty$, which was previously known only for $p=2$ in  Erbar--Huesmann \cite{EH15}.  Thirdly, we construct perimeter measures and introduce the notion of the reduced boundary in Section \ref{sec:reduced boundary}. We then prove that the perimeter measure can be expressed by the $1$-codimensional Poisson measure restricted to the reduced boundary (Theorem \ref{thm: fpr}). Fourthly, we construct the total variation measures for functions of bounded variation and prove the Gau\ss--Green formula (Theorem \ref{thm: BV}).

\smallskip
We now explain each result in details.

\subsection{$m$-codimensional Poisson measure}
The first main result of this paper is the construction of the $m$-codimensional Poisson measure on $\U(\R^n)$. Since $\U(\R^n)$ is infinite-dimensional, it is formally written as 
\begin{align*}
\text{``$(\infty-m)$-dimensional Poisson measure''}.
\end{align*}
In the case of finite-dimensional spaces, usually the construction of finite-codimensional measures builds upon covering arguments, which heavily rely on the volume doubling property of the ambient measure. However, this property does not hold for the Poisson measure $\pi$ on $\U(\R^n)$. 

%{\color{red}Furthermore, as was explained, any ${\sf d}_\U$-metric ball is negligible for the Poisson measure $\p$, due to which any covering argument based on ${\sf d}_\U$-metric balls does not convey meaningful geometric information. }

We construct the $m$-codimensional Poisson measure on $\U(\R^n)$ by passing to the limit of finite dimensional approximations obtained by using the $m$-codimensional Poisson measure on $\U(B_r)$.
The key step in the construction is to prove the {\it monotonicity} of these finite dimensional approximations  with respect to the radius $r$, allowing us to find a unique limit measure.
More in details, based on the decomposition $\U(B_r) = \sqcup_{k \in \N} \U^k(B_r)$, we build $\rho^m_{\U(B_r)}$, the {\it spherical Hausdorff measure of codimension $m$} in $\U(B_r)$, by summing the $m$-codimensional spherical Hausdorff measure  $\rho^{m, k}_{\U(B_r)}$ on the $k$-particle configuration space $\U^k(B_r)$, which is obtained by the quotient measure of the $m$-codimensional spherical Hausdorff measure on the $k$-product space $B_r^{\times k}$ with a suitable renormalisation corresponding to the Poisson measure.
The {\it localised $m$-codimensional Poisson measure $\rho^m_r$} of a set $A\subset \U(\R^n)$ is then obtained by averaging the $\rho^m_{\U(B_r)}$-measure of sections of $A$ with the Poisson measure $\pi_{B_r^c}$ on $\U(B_r^c)$, i.e.
\begin{equation*}
	\rho^m_r(A) 
	:= \int_{\U(B_r^c)} \rho^m_{\U(B_r)}(\{\gamma\in \U(B_r)\, : \gamma + \eta \in A\}) d \p_{B_r^c}(\eta) \, .
\end{equation*}
We prove that $\rho^m_r$ is well-defined on Borel sets (indeed, we prove it for all Suslin sets), and that it is monotone increasing with respect to $r$ (Theorem \ref{thm: MI} and Definition \ref{defn: MCH}). In particular, we can define {\it the $m$-codimensional Poisson measure} as
\begin{equation*}
	\rho^m := \lim_{r\to \infty} \rho^m_r \, .
\end{equation*}
We refer the readers to Section \ref{s:FCP} for the detailed construction of $\rho^m$. 

\subsection{Bessel capacity}
In Section \ref{sec: cap}, we compare the $m$-codimensional Poisson measure $\rho^m$ and ${\rm Cap}_{\alpha, p}$, the {\it Bessel capacity} induced by the Dirichlet form associated with infinite independent Brownian motions constructed in Albeverio--Kondratiev--R\"ockner~\cite{AKR98}. We prove that zero capacity sets are $\rho^m$ negligible provided $\alpha p > m$ (Theorem \ref{thm: CH}). This result, that is well-known in the case of finite-dimensional spaces, proves that our $m$-codimensional Poisson measure~$\rho^m$ behaves coherently with the potential-analytic structure of~$\U(\R^n)$. 
To prove it,  we introduce the $(\alpha, p)$-Bessel capacity ${\rm Cap}^{\U(B_r)}_{\alpha, p}$ on $\U(B_r)$ and the localised $(\alpha, p)$-Bessel capacity
${\rm Cap}^{r}_{\alpha, p}$ on $\U(\R^n)$ based on the localisation argument of the $L^p$-heat semigroup $\{T_t\}$ on $\U(\R^n)$. 
We prove that ${\rm Cap}_{\alpha, p}$ is approximated by ${\rm Cap}^{r}_{\alpha, p}$ as $r \to \infty$, hence we can obtain the proof by lifting the corresponding result for $\rho^m_{\U(B_r)}$ and ${\rm Cap}^{\U(B_r)}_{\alpha, p}$ in $\U(B_r)$ (see Proposition \ref{lem: MC3}). We refer the readers to Section \ref{sec: cap} for the detailed arguments. 

As an application, we prove in Corollary \ref{cor: measure} that, if ${\rm Cap}_{1,2}(A)=0$ then $|\DDD F|(A)=0$ for every $F\in {\rm BV}(\U(\R^n)) \cap L^2(\U(\R^n), \p)$, where $|\DDD F|$ is the total variation measure (Definition \ref{defn: TV}) and ${\rm BV}(\U(\R^n))$ is the space of functions of bounded variation (Definition \ref{defn: BVU}).
The latter result will be fundamental for applications  to stochastic analysis of infinite-particle diffusions, which will be the subject of a forthcoming paper.

\subsection{Functions of bounded variations and Caccioppoli sets}

In the second part of this paper we develop the theory of functions of bounded variation and sets of finite perimeter in $\U(\R^n)$.
%\paragraph{\bf BV functions}
In Section \ref{sec:BV} we propose three different notions of functions with bounded variation. The first one follows the classical {\it variational approach}, the second one is built upon the  {\it relaxation approach}, while the third one relies on the regularisation properties of the {\it heat semigroup}.
It turns out that they are all equivalent, as shown in Section \ref{sec: EDBV},
and the resulting class is denoted by ${\rm BV}(\U(\R^n))$. For $F\in {\rm BV}(\U(\R^n))$ we define a total variation measure $|\DDD F|$ and prove a {\it Gau\ss--Green formula} (see Theorem below). We remark that in our infinite-dimesional setting, Riesz--Markov--Kakutani's representation theorem is not available due to the lack of local compactness. In particular, the construction of the total variation measure is not straightforward. We follow an unusual path to show its existence: we first develop the theory of {\it sets with finite perimeter} relying on the non-linear dimension reduction. We then employ the {\it coarea formula} to build the total variation measure of a function of bounded variation as a superposition of perimeter measures.

\medskip

Sets of finite perimeter are those Borel sets $E$ such that $\chi_E \in {\rm BV}(\U(\R^n))$, where $\chi_E$ denotes the indicator function of $E$. In Section \ref{sec:reduced boundary}, we study their structure by means of the non-linear reduction approach, a part of which uses a strategy inspired by Hino \cite{H10} for the study of Wiener spaces.
The key result in this regard is Proposition \ref{lem: BV1} saying that if $E$ has finite perimeter then the projection $E_{\eta,r}:=\{ \gamma\in \U(B_r)\, : \, \gamma + \eta \in E \}$ has finite {\it localised total variation} in $B_r$, for $\pi_{B_r^c}$-a.e. $\eta\in B_r^c$. Hence, we can reduce the problem to the study of sections that are sets with finite perimeter in $\U(B_r)$. As already remarked, the latter is essentially a finite dimensional space, so we can appeal to classical tools of geometric measure theory to attack the problem. 

The {\it reduced boundary} $\partial^* E$ of a set of finite perimeter $E\subset \U(\R^n)$ is then defined in terms of the reduced boundary of the sections $E_{\eta,r}$, through a limit procedure. The resulting object allows us to represent the perimeter measure as
\begin{equation*}
	\| E \| = \rho^1|_{\partial^* E} \, ,
\end{equation*}
which is a generalisation of the identity proven in the Euclidean setting by E. De Giorgi \cite{DeGiorgi54,DeGiorgi55}.

Our approach  to the BV theory deviates from the standard one. We define the total variation measure $|\DDD F|$ of a function $F\in {\rm BV}(\U(\R^n))$ by imposing the validity of the coarea formula. 
More precisely, we show that $dt$-a.e. level set $\{F>t\}$ is of finite perimeter and we set
\begin{equation*}
	|\DDD F|:=\int_{-\infty}^{\infty} \ \bigl\| {\{ F > t\}} \bigr\|  d t \, ,
\end{equation*}
taking advantage of the perimeter measure $\|\{F>t\}\|$ that has been already defined using finite dimensional approximations.
The reason for this non-standard treatment is that we are not able to build directly $|\DDD F|$ through a finite dimensional approximation, since the latter does not have a simple expression in terms of $1$-codimensional Poisson measure $\rho^1$ restricted to a suitable subset. 
Our approach is, however, consistent with the standard one, as shown in Corollary \ref{cor: PMTV} and in Theorem \ref{thm: BV}.

We summarise the main results in Section \ref{sec:reduced boundary} and Section \ref{sec: CF} concerning functions of bounded variations and a sets of finite perimeter.  We denote by $\CylV(\U(\R^n))$ the space of cylinder vector fields on $\U(\R^n)$ and by $(T\U, \langle \cdot, \cdot \rangle_{T\U})$ the tangent bundle to $\U(\R^n)$ with the pointwise inner product $\langle \cdot, \cdot \rangle_{T\U}$ (see Section \ref{subsection:Sobolev}).

\begin{thm*}[Theorems \ref{thm: fpr}, \ref{thm: BV}]\label{thm:main intro}
	For $F \in L^2(\U(\R^n),\pi) \cap  {\rm BV}(\U(\R^n))$, there exists a unique positive finite measure $|\DDD F|$ on $\U(\R^n)$ and a  $\pi$-a.e. unique $T\U$-valued measurable function $\sigma$ on $\U(\R^n)$ so that $|\sigma|_{T\U}=1$  $|\DDD F|$-a.e., and 
	\begin{align*} 
		\int_{\U(\R^n)} (\nabla^* V) F d\p = \int_{\U(\R^n)} \langle V, \sigma\rangle_{T\U} d| \DDD F|\, , \quad \forall V \in \CylV(\U) \, .
	\end{align*} 
	If, furthermore, $F=\chi_{E}$, then %denoting $|\DDD \chi_E| =\| E \|$. 
	$$|\DDD \chi_E| = \rho^1|_{\partial^* E}.$$
\end{thm*}

\subsection{Potential applications}
%\sout{Firstly, by our results, the configuration space turns out to be the first example of an {\it infinite-dimensional} and {\it non-linear} space in which a reasonable theory of functions of bounded variation and sets of finite perimeters can be developed. From the technical viewpoint, it is of independent interest since all the known BV theories of infinite-dimensional spaces heavily rely upon the linear structure of given spaces. We expect that a part of the technical tools developed in this paper can be applied also to other non-linear infinite-dimensional spaces.}

Our theory of functions of bounded variation has several potential applications to related fields such as singular boundary problems of infinite interacting diffusions.  In the case of the Euclidean space $\R^n$ --- the the case of {\it one particle} Brownian motion --- there is a connection between the theory of BV functions and stochastic analysis: the (modified) reflected Brownian motion on an open set $A \subset \R^n$ is {\it semi-martingale} if and only if $A$ is Caccioppoli. Furthermore, the modified reflected  Brownian motion satisfies the {\it generalised Skorokhod equation} and the {\it generalised It\^o's formula}, where the reflection at the boundary is phrased by the local time in terms of the reduced boundary (see, \cite[Theorem 7.1, 7.2]{Fuk99}).   %the  --- as indicated in \cite{Fuk01, FH01} in the case of the Wiener space, 
As an infinite dimensional counterpart,  one can expect that the main results in this paper would be useful to construct infinite particle diffusions with singular boundary conditions (cf.\ \cite[Theorem 4.4.]{FH01} in the case of the Wiener space). 
%the extension of the aforementioned results in stochastic analysis  to $\U(\R^n)$, which corresponds to stochastic analysis in the infinite particle case (cf.\ \cite[Theorem 4.4.]{FH01} for the case of the Wiener space). 
%the singular boundary theory of infinite-particle systems of Brownian motions on $\R^n$ and to develop {\it generalised It\^o's formula} that describes the infinite-dimensional It\^o's formula involving the reflection term at boundaries (). 
%All these subjects will be addressed in forthcoming work. 

\subsection{Structure of the paper}
In Section \ref{sec: Pre}, we collect preliminary results regarding the the configuration space, Suslin sets and measurability of sections.
In Section \ref{s:FCP}, we construct the $m$-codimensional measure. Relations with the Bessel capacity are studied in Section \ref{sec: cap}.
Section \ref{sec:BV} is devoted to the study of functions of bounded variation. We introduce three different notion and prove the equivalence. In Section \ref{sec:reduced boundary}, we introduce and study sets of finite perimeter. We build the notion of reduced boundary and the perimeter measure, and we show an integration by parts formula.
In Section \ref{sec: CF}, we introduce the total variation measure of functions with bounded variations by employing the coarea formula, and prove  a Gau\ss--Green type integration-by-parts formula.

%The letter $k$ always denotes the number of the product 

\subsection*{Acknowledgements}
The two authors are indebted to Lorenzo Dello Schiavo and the anonymous referee for their very carefull reading that improved the original manuscript.   
The first named author was supported by the Giorgio and Elena Petronio Fellowship at the Institute for Advanced Study at the time of the writing.
The second named author gratefully acknowledges funding by: the JSPS Overseas Research Fellowships, Grant Nr.290142; World Premier International Research Center Initiative (WPI), MEXT, Japan; JSPS Grant-in-Aid for Scientific Research on Innovative Areas “Discrete Geometric Analysis for Materials Design”, Grant Number 17H06465; and the Alexander von Humboldt Stiftung, Humboldt-Forschungsstipendium.

\subsection*{Declarations of interest}
The authors declare that they have no competing
interest.

%The goal of this note is to obtain the explicit representation of the perimeter measure in the Gau\ss--Green formula for Caccioppli sets in the multiple configuration space $\U=\U(M)$ with the Poisson measure $\pi_{m}$ over a complete connected Riemannian manifold $M$ with the Hausdorff measure  $m=\mathcal H$. The GG formula is going to be proved in 
%another manuscript BV\_DirichletForm.pdf in the following form: If $A \subset \U$ is Caccioppli (i.e., $\1_A \in BV(\U)$), then there exists a smooth measure $\nu_A$ and $TX$-valued measurable function $\sigma$ so that $|\sigma|=1$ $\nu$-a.e., and 
%\begin{align} \label{eq1}
%\int_A {\rm div}G d\pi_{m} = -\int_{\partial A} \langle G, \sigma \rangle d\nu_A, \quad \forall G \in {\rm TestV}(\U).
%\end{align}
%Here ${\rm Test}F(\U):=\{F(f_1^*\gamma,\ldots,f_k^*\gamma): f_1,\ldots,f_k \in C_c^\infty(M), \ F\in C^\infty_b(\R^k, \R)\}$ whereby $f^* \gamma=\int_M f d\gamma=\sum_{x_i \in \gamma} f(x_i),$ and ${\rm Test}V(\U):=\{\sum_{i=1}^k F_i V_i: F_i \in {\rm Test}F(\U), V_i=v_i^*\ (v_i \in V_c(M))\}.$ Note that $v^*(\gamma):=\sum_{x \in \gamma}v(x) \in T_\gamma\U$ for $v \in V_c(M)$. 
%
%The measure $\nu_A$ is called the perimeter measure of $A$, and denoted by $\nu_A(d\gamma)=d\|A\|$. 
%In this note, our aim is to obtain the concrete geometric expression $\nu_A(B)=\rho(B \cap \partial_*A)$ by introducing suitable notions of ``one-codimensional Poisson measure $\rho$" and the reduced boundary $\partial_*A$ of $A$.

\section{Preliminaries} \label{sec: Pre}
\subsection{Notational convention}
In this paper, the bold fonts $\SS$, $\mathbf{L}, \ldots$ are mainly used for objects in product spaces or vector-valued objects, while the serif fonts $\S, \DDD, \ldots$ are used for objects in the quotient space of product spaces with respect to the $k$-symmetric group $\mathfrak S_k$ or for objects in the configuration space $\U(\R^n)$. 
\smallskip
\\
The lower-case fonts $f, g, h, v, w,\ldots$ are mainly used for functions on the base space $\R^n$, while the upper-case fonts $F, G, H, V, W,\ldots$ are used for functions on the configuration space $\U(\R^n)$.
\smallskip
\\
We denote by $\chi_E$ the indicator function of $E$, i.e., $\chi_E=1$ on $E$ and $\chi_E=0$ on $E^c$. 
Let $\Omega \subset \R^n$ be a closed domain. We denote by $C_*^{\infty}(\Omega)$ the space of smooth functions with compact support in $\Omega\setminus \partial \Omega$ (i.e., functions vanish at the boundary $\partial \Omega$), while $C_c^{\infty}(\Omega)$ denotes the space of compactly supported smooth functions on $\Omega$ (functions do not necessarily vanish at the boundary $\partial \Omega$). Note that $C_*^{\infty}(\Omega) \subset C_c^{\infty}(\Omega)$ in general, but these two function spaces coincide, i.e. $C_c^\infty(\R^n) = C_*^\infty(\R^n)$, when we take $\Omega=\R^n$.

\subsection{Configuration spaces}
Let $\R^n$ be the $n$-dimensional  Euclidean space. Let $B_r:=B_r(0) \subset \R^n$ be the closed ball with radius $r>0$ centred  at the origin $0$. Let $\delta_{x}$ denote the point measure at $x \in \R^n$, i.e. $\delta_x(A)=1$ if and only if $x \in A$.  
We denote by $\U(\R^n)$ the {\it configuration space over $\R^n$ without multiplicity}, i.e.  the set of all locally finite point measures $\gamma$ on $\R^n$ so that $\gamma(\{x\}) \in \{0,1\}$ for every $x \in \R^n$. Elements in $\U(\R^n)$ can be written as $\gamma = \sum_{i=1}^N \delta_{x_i}$ with $N \in \N \cup \{\infty\}$ and $\{x_i\}_{i \in \N} \subset \R^n$.
%$$
%\U(\R^n):=\biggl\{\gamma = \sum_{i=1}^{N} \delta_{x_i}: \gamma(K) <\infty, \ \text{$\forall K \Subset \R^n$}, \ \gamma(\{x\}) \in \{0,1\},\ \forall x \in \R^n, \ N \in \N \cup\{\infty\} \biggr\} \, .
%$$
%$$
%\U(\R^n):=\Bigl\{\gamma = \sum_{i \in \N} \delta_{x_i}: \gamma(K) <\infty\ \text{for every compact $K \subset \R^n$}, \ \gamma(\{x\}) \in \{0,1\}\ \forall x \in \R^n \Bigr\} \, .
%$$
Let $\U(A)$ denote the configuration space over a Polish subspace $A \subset \R^n$ defined analogously to $\U(\R^n)$, and $\U^k(A)$ denote the space of $k$-configurations on a subset $A$, i.e. $\U^{k}(A)=\{\gamma \in \U(A): \gamma(A)=k\}$.
We equip $\U(\R^n)$ with the vague topology $\tau_v$, i.e., $\gamma_n \in \U(\R^n)$ converges to $\gamma \in \U(\R^n)$ in $\tau_v$ if and only if $\gamma_n(f) \to \gamma(f)$ for any $f \in C_c(\R^n)$. For a subset $A \subset \R^n$, we equip $\U(A)$ with the relative topology as a subset in $\U(\R^n)$. Let $\mathscr B(\U(A), \tau_v)$ denote the Borel $\sigma$-algebra associated with the vague topology $\tau_v$.
For a set $A \subset \R^n$, let ${\rm pr}_A: \U(\R^n) \to \U(A)$ be the projection defined by the restriction of configurations on $A$, i.e. ${\rm pr}_A(\gamma)=\gamma|_{A}$.

Given $A \subset \R^n$, an open or closed domain, we denote by $\pi_A$ the {\it Poisson measure} on $\U(A)$ whose intensity measure is the Lebesgue measure restricted to $A$, namely, $\pi_A$ is the unique Borel probability measure so that, for all $f \in C_c(A)$, the following holds
\begin{align} \label{eq:Int}
\int_{\U(A)} e^{ f^\ast}d\p_{A} = \exp\biggl\{ \int_{A} (e^{f}-1)d \mathbf {L}^n(x)\biggr\} \, .
\end{align}
Here $\mathbf {L}^n$ denotes the $n$-dimensional Lebesgue measure. See \cite{GV64} for a reference for the expression~\eqref{eq:Int}. We write $\p=\p_{\R^n}$.
Note that $\pi_A$ coincides with the push-forward measure $\pi_A=({\rm pr}_A)_\#\pi$.
Let 
$$
\dg_k:=\{(x)_{1 \le i \le m} \in (\R^n)^{\times k}: \exists i, j\  \text{s.t.}\ x_i = x_j\} \, ,
$$ 
denote the set of all sub-diagonals in $ (\R^n)^{\times k}$, and let $\mathfrak S_k$ denote the $k$-symmetric group. For any  set $A \subset \R^n$, we identify 
$$\U^{k}(A) \cong (A^{\times k}\setminus \dg_k) /\mathfrak S_k, \quad k \in \N.$$ 
Let $\s_{k}:A^{\times k} \setminus \dg_k \to \U^{k}(A)$ be the canonical projection with respect to the action of $\mathfrak S_k$, i.e. $\s_{k}: (x_i)_{1\le i \le k} \mapsto \sum_{i=1}^k\delta_{x_i}$. %Let $\L_{k}$ be the set of all measurable inverse maps for $\s_k$, i.e., 
%$$\L_{k}:=\{\l_k: \U^{k}(A) \to A^{\times k} \setminus \dg_k: \ \text{$\mathscr B(\U^k(A), \tau_v)/\mathscr B(A^{\times k} \setminus \dg_k)$-measurable, $\s_k\circ \l_k(\gamma)=\gamma$}\}.$$
%We call $\l_k \in \L_k$ {\it a labelling map}. 
We say that a function $f: \sqcup_{k=1}^\infty (\R^n)^{\times k} \to \R$ is {\it symmetric} iff $f(\mathbf x_{\sigma_k}) = f(\mathbf x_k)$ with $\mathbf x_{\sigma_k}:=(x_{\sigma_k(1)}, \ldots, x_{\sigma_k(k)})$ for every permutation $\sigma_k \in \mathfrak S_k$ and every $k \in \N$.
% \{0, \ldots, k-1\} \to  \{0, \ldots, k-1\}$.

For $\mathbf x _k, \mathbf y_k \in A^{\times k}$ with $\mathbf s_k(\mathbf x_k) = \gamma \in \U^k(A)$ and $\mathbf s_k(\mathbf y_k) = \eta \in \U^k(A)$, define {\it the $L^2$-transportation distance $\d_{\U^k}(\gamma, \eta)$} on $\U^k(A)$ by the quotient metric w.r.t.\ $\mathfrak S_k$:
\begin{align} \label{d:L2T}
\d_{\U^k}(\gamma, \eta)=\inf_{\sigma_k \in \mathfrak S_k} | \mathbf x_{\sigma_k} - \mathbf y_k |_{\R^{nk}} \, .
\end{align}
Here $|\mathbf x_k - \mathbf y_k |_{\R^{nk}}$ denotes the standard Euclidean distance in $\R^{nk}$.
\begin{rem}[Polishness/lack of completeness] \normalfont \  \label{r:PI}
\begin{enumerate}[(a)] 
\item The space~$\U(\R^n)$ equipped with the vague topology is a Polish space. The subpace~$\U^k(A) \subset \U(\R^n)$ is a Polish subspace for every $k \in \N$ if $A$ is a Polish subspace in $\R^n$. 
This fact will play a role later in Section~\ref{s:FCP} to discuss Suslin sets.  
\item The metric space~$(\U^k(A), \d_{\U^k})$ is not complete even if $A$ is closed, 
due to the lack of multiple configurations in $\U^k(A)$. This fact is, however, irrelevant to the rest of arguments. 
\end{enumerate}
\end{rem}

\subsection{Spherical Hausdorff measure} 
Let $(X, d)$ be a metric space and $n$ be the Hausdorff dimension of $X$. 
For $m \le n$, the {\it $m$-dimensional spherical Hausdorff measure}~$\SS^m_X$  on~$X$ is defined as the restriction of the following outer measure $\SS^m_X$ on $\SS^m_X$-measurable sets (i.e., the Carath\'eodory measurable sets):
\begin{align} \label{SH}
\SS^m_X(A):=\lim_{\e \to 0} \SS_{X, \e}^m(A):=\lim_{\e \to 0} \inf \Bigl\{ \sum_{i\in\N} {\sf diam}(B_i)^m:\text{ $B_i$ open ball with ${\sf diam}(B_i)<\e$, $ A \subset \sum_{i\in \N} B_i$}\Bigr\}.
\end{align}
Here ${\sf diam}(B_i)=\sup\{d(x,y): x, y \in B_i\}$ denotes the diameter of $B_i$. We call $\SS^m_{X,\e}$ the {\it $m$-dimensional $\e$-Hausdorff measure}. If $X=\R^n$, we simply write $\SS^m$ and $\SS^m_\e$ instead of $\SS^m_{\R^n}$ and $\SS^m_{\R^n,\e}$ respectively.  
\begin{rem}[Comparison with the standard Hausdorff measure] \normalfont \label{rem: CSH}
In the case of $m<n$, the spherical Hausdorff measure $\SS^m_X$ does not coincide with the standard Hausdorff measure in general, the latter is smaller since it is defined allowing {\it all non-empty coverings} instead of {\it open balls}. In the case of $m=n$ and $X=\R^n$, however, $\SS^m$ coincides with the standard $n$-dimensional Hausdorff measure and also with the $n$-dimensional Lebesgue measure (\cite[2.10.35]{Fed69}). 
Note that $\SS^m$ is a Borel measure, but not $\sigma$-finite for $m<n$. 
\end{rem}
 For a bounded set $A \subset \R^n$, let $\SS^n|_{A}$ be the spherical Hausdorff measure restricted to $A$. The spherical Hausdorff measure $(\SS^n|_{A})^{\otimes k}$ on  $A^{\times k}$ can be push-forward to the $k$-configuration space $\U^k(A)$ by the projection map $\s_{k}$, i.e.
$$
\S_{A}^{k} :=\frac{1}{k!} (\s_{k})_\# (\SS^n|_{A})^{\otimes k} \, .
$$ 
It is immediate by construction to see that $\S_{A}^{k}$ is the spherical Hausdorff measure on $\U^k(A)$ induced by the $L^2$-transportation distance $\d_{\U^k}$ up to constant multiplication.
We introduce the $m$-codimensional spherical Hausdorff measure and the $m$-codimensional $\e$-spherical Hausdorff measure on $\U^k(A)$ as follows
\begin{align} \label{m-H}
\S_{A}^{m, k}=\frac{1}{k!} (\s_{k})_\#(\SS^{nk-m}|_{A^{\times k}}), \quad \S_{A, \e}^{m, k}=\frac{1}{k!} (\s_{k})_\#(\SS_\e^{nk-m}|_{A^{\times k}}).
\end{align}
One can immediately see that $\S_{A, \e}^{m, k}$ is (up to constant multiplication) the $m$-codimensional $\e$-spherical Hausdorff measure on $\U^k(A)$ associated with the $L^2$-transportation distance $\d_{\U^k}$.

 \subsection{Regularity of the spherical Hausdorff measures} 
 In this section, we prove the upper semi-continuity of the $\e$-spherical Hausdorff measure on sections of compact sets, which will be of use in Section \ref{s:FCP}. 
 
\begin{prop} \label{prop: usc}
Let $(X,d_X)$, $(Y, d_Y)$ be metric spaces, and $K\subset X\times Y$ be a compact set. Then, the map $Y\ni y\mapsto \SS_{X, \e}^m(K^y)$ is upper semi-continuous. Here,  $K^y:=\{x\in X\, :\, (x,y)\in K \}$.
\end{prop}
\begin{proof}
    	Let us fix $y\in Y$ and a sequence $y_n \to y$. The family of compact sets $(K^{y_n}\times \{y_n\})_{n\in\N} \subset K$ is precompact with respect to the Hausdorff topology in $K$ endowed with the product metric (e.g., \cite[Theorem 7.3.8]{BBI01}). In particular, we can take a (non-relabeled) subsequence so that $K^{y_{n}}\times \{y_{n}\}\to \bar K\times \{ y\} \subset K$, as $n \to \infty$ in the Hausdorff topology, and $\bar K \subset K^y$ by the definition of $K^y$.

Let us fix $\delta>0$ and a family of open balls $B_1,\ldots,B_\ell \subset X$ with radius smaller than $\e(1-\delta)>0$ so that 
    	\begin{equation*}
    		\bar K \subset \bigcup_{i=1}^\ell B_i,
    	\end{equation*}
    	and
    	\begin{equation}\label{z1}
    		\SS^m_{X, \e}(\bar K) \ge c(m)\sum_{i=1}^\ell r_i^m - \delta.
    	\end{equation}
	Here $c(m)$ denotes the constant depending on $m$ such that $\mathbf{L}^m(B_i) = c(m)r_i^m$. Note that we can always take $\ell=\ell(\delta)$ to be finite for any $\delta>0$ by the compactness of $\bar{K}$. Let $\underline r=\underline r(\delta):=\min\{r_i: 1 \le i \le l(\delta)\}>0$ be the minimum radius among $\{B_i\}_{1 \le i \le l}$. 
       
    We claim that there exists $\bar{k}=\bar{k}(\delta) \in \N$ so that $K^{y_{n_k}} \subset \cup_{i=1}^\ell B(x_i, \frac{1}{1-\delta}r_i)$ for any $k \ge \bar{k}$. Here $x_i$ and $r_i$ are the centre and the radius of $B_i$. 
       
         Indeed, by the Hausdorff convergence of $K^{y_{n_k}}$ to $\bar{K}$,  there exists $\bar k:=\bar k(\delta)\in \N$ such that, for any $k>\bar k$,  it holds that $K^{y_{n_k}} \subset B_{\underline r \delta} (\bar K)$. Here,  $B_{\underline r \delta }(\bar K)$ denotes the $\underline  r \delta$-neighbourhood of $\bar K$ in $X$, i.e., $B_{\underline r \delta}(\bar K):=\{x \in X: d(x, \bar K) < \underline  r\delta\}$.  Hence,  for any $z\in K^{y_{n_k}}$, we can always find $x\in \bar K$ such that $d(x,z)<\underline  r \delta$. Since $x\in B_i$ for some $i=1,\ldots,\ell$, we conclude $z\in B(x_i, \frac{1}{1-\delta}r_i)$ by noting that $\frac{1}{1-\delta}r_i -r_i = \frac{\delta}{1-\delta} r_i \ge \delta \underline r$. 
        
       By using the claim in the previous paragraph, the monotonicity $\SS^m_{X, a} \ge \SS^m_{X, b}$ whenever $a \le b$, and \eqref{z1}, we obtain that 
        \begin{equation}
        \SS^m_{X, \e}(K^{y_{n_k}}) \le c(m)(1+\delta)^m \sum_{i=1}^\ell r_i^m
        	\le (1+\delta)^m \SS^m_{X, \e(1-\delta)}(\bar K) + \delta(1+\delta)^mc(m),
        \end{equation}
        for any $k \ge \bar{k}(\delta)$. By taking $\delta \to 0$ after taking $k \to \infty$, we conclude that 
        \begin{equation}
        	\limsup_{k\to \infty} \SS^m_{X, \e}(K^{y_{n_k}}) \le \SS^m_{X, \e}(\bar K) \le \SS^m_{X, \e}(K^y),
        \end{equation}
        which is the sought conclusion.
        \end{proof}

\subsection{Differential structure on configuration spaces}\label{subsection:Sobolev}
In this section, $\Omega \subset \R^n$ will denote either a closed domain with smooth boundary or the whole Euclidean space $\R^n$.
Below we review the natural differential structure of $\U(\Omega)$, obtained by lifting the Euclidean one on $\Omega$.
We follow closely the presentation in \cite{AKR98}.

\paragraph{Cylinder functions, vector fields and divergence}
% We introduce a class of test functions on the configuration space $\U(\R^n)$.

\begin{defn}[Cylinder functions] \normalfont
    We define the class of {\it cylinder functions} as
	\begin{equation} \label{d:Cyl}
		\CylF(\U( \Omega)) := \{\Phi(f^\ast_1,\.,f^\ast_k) \, : \,  \Phi \in C_b^\infty(\R^k), \, f_i \in C_c^\infty(\Omega), \, k\in \N  \} \, ,
	\end{equation}
    where $f^\ast(\gamma) := \int_{\Omega} f d\gamma$ for every $\gamma\in \U(\Omega)$. We call $f_i$ {\it inner function} and $\Phi$ {\it outer function}.    
\end{defn}

The tangent space $T_\gamma \U(\Omega)$ at $\gamma \in \U(\Omega)$ is identified with the Hilbert space of measurable $\gamma$-square-integrable vector fields $V: \Omega \to T(\R^n)$ equipped with the scalar product: for $\gamma$-measurable $V, W:\Omega \to T(\R^n)$, 
 \begin{align*} 
 	\langle V, W \rangle_{T\U} &= \int_{\Omega} \langle V(x), W(x) \rangle_{T\R^n} d\gamma(x) \, ,
	\\
	|V|^2_{T\U}&=\int_{\Omega} \langle V(x), V(x) \rangle_{T\R^n} d\gamma(x) \,.
 \end{align*}
We define the tangent bundle of $\U(\Omega)$ by $T\U(\Omega):=\sqcup_{\gamma \in \U(\Omega)} T_\gamma\U(\Omega)$.

 \begin{defn}[Cylinder vector fields] \normalfont
 	We define two classes of {\it cylinder vector fields} as
 	\begin{equation*}
 		\CylV(\U(\Omega)):=\Bigl\{V(\gamma, x)=\sum_{i=1}^kF_i(\gamma) v_i(x): F_i \in \CylF(\U(\Omega)),\ v_i \in C_c^{\infty}(\Omega; \R^n), k \in \N \Bigr\} \, ,
 	\end{equation*}
 	
 	\begin{equation*}
 		\CylV_*(\U(\Omega)):=\Bigl\{V(\gamma, x)=\sum_{i=1}^kF_i(\gamma) v_i(x): F_i \in \CylF(\U(\Omega)),\ v_i \in C_*^{\infty}(\Omega; \R^n), k \in \N \Bigr\} \, .
 	\end{equation*}    
 \end{defn}
 Notice that
 $\CylV_*(\U(\Omega)) \subset \CylV(\U(\Omega))$, and $\CylV_*(\U(\Omega)) = \CylV(\U(\Omega))$ when $\Omega = \R^n$. Using the tensorial notation, we can write
\begin{equation}
\begin{split}
	&\CylV(\U(\Omega)) = \CylF(\U(\Omega)) \otimes_{\R} C^\infty_c(\Omega;\R^n)
	\\
	&\CylV_*(\U(\Omega)) = \CylF(\U(\Omega)) \otimes_{\R} C^\infty_*(\Omega;\R^n)\, .
\end{split}
\end{equation}

 Let $p\in [1,\infty)$. For $V \in \CylV(\U(\Omega))$, we define 
  \begin{align}\label{d:VFN}
\|V\|_{L^p(T\U(\Omega))}^p:=\|V\|_{L^p(\U(\Omega) \to T\U(\Omega), \pi_{\Omega})}^p:= \int_{\U(\Omega)} | V(\gamma) |_{T_\gamma\U}^p d \pi_{\Omega}(\gamma) \, ,
 \end{align}
and introduce the associated Banach space by
 $$L^p(T\U(\Omega), \pi_\Omega):= \text{the completion of $\CylV(\U(\Omega))$ with respect to~$\|\cdot\|_{L^p(T\U(\Omega))}$} \, .$$
% \begin{equation*}
% \|V\|_{L^p(T\U(\Omega))}^p:= \int_{\U(\Omega)} | V(\gamma) |_{T_\gamma\U}^p d \pi_{\Omega}(\gamma) \, .
% \end{equation*}
%The formula~$\|V\|_{L^p(T\U(\Omega))}$ by the following proposition.
 See \cite[the fifth displayed formula on p.23]{ADL01} in the case of $p=2$.
 \begin{rem}\normalfont
When $p=2$, the closure $L^p(T\U(\Omega), \pi_\Omega)$ coincides with the $L^2$-section of vector fields $L^2(\U(\Omega) \to T\U(\Omega), \pi_\Omega)$ defined as the direct integral of the Hilbert spaces $(T_\gamma\U(\Omega), \langle\cdot, \cdot \rangle_{T\U})$ with respect to $\pi_{\Omega}$. See, for instance, the proof of \cite[p.165, 3rd bullet point]{Ebe99}. 
 %However, we do not use the direct integral representation in the rest of arguments. 
 \end{rem}
 
% {\color{blue}
 %	To check
 %}
 
 \begin{prop}\label{prop:LPV}
 	Let $1 \le p < \infty$.  Then,
 	\begin{equation}
 		\| V \|_{L^p(T\U(\Omega))} < \infty \, ,
 		\quad \text{$V\in \CylV(\U(\Omega))$} \, .
 	\end{equation}
% 	\begin{equation}\label{z50}
% 		\CylV_*(\U(\Omega)) \subset \CylV(\U(\Omega)) \subset L^p(T\U(\Omega), \pi) \, ,
% 	\end{equation}
 	Moreover, $\CylV_*(\U(\Omega))$ is dense in $L^p(T\U(\Omega), \pi_\Omega)$.
 \end{prop}
    
 \begin{proof}
 	Let $V(\gamma, x)=\sum_{i=1}^kF_i(\gamma) v_i(x)$. Then, we have that 
 	\begin{align*}
 		\int_{\U(\Omega)} |V|^p_{T_\gamma \U(\Omega)} d \p_{\Omega}(\gamma) 
 		&\le \max_{1 \le i \le k}\|F\|_{L^\infty} \sum_{i, j=1}^k \int_{\U(\Omega)} \biggl( \int_{\Omega}|v_i||v_j|d\gamma \biggr)^{p/2} d \p_{\Omega}(\gamma).
 	\end{align*}
 	By the exponential integrability implied by \eqref{eq:Int}, we obtain that the function $\gamma \mapsto G(\gamma):=\int_{\Omega}|v_i||v_j|d\gamma$ is $L^p(\U(\Omega), \p_{\Omega})$ for any $1 \le p <\infty$, which concludes the first assertion.

 	The density of  $\CylV_*(\U(\Omega))$ in $L^p(T\U(\Omega), \pi_\Omega)$ follows from the density of $C^\infty_*(\Omega; \R^n)$ in $L^p(\Omega; \R^n)$.
 	More precisely, we check that for any $V\in \CylV(\U(\R^n))$ and $\e>0$ there exists $W\in \CylV_*(\U(\Omega))$ such that $\int_{\U(\Omega)} | V - W|_{T_\gamma\U}^p d \p_\Omega \le \e$. To this aim we write 
 	$V= \sum_{i=1}^kF_i v_i$ and pick $w_i \in C_*^\infty(\Omega)$ such that $\sum_{i=1}^k \| v_i - w_i\|_{L^p} < \e$ and set $W:= \sum_{i=1}^kF_i w_i$. It is straightforward to see that $W$ satisfies the needed estimate. By noting that $L^p(T\U(\Omega), \pi_\Omega)$ is defined as the completion of $\CylV(\U(\Omega))$ with respect to the norm $\|V\|_{L^p(T\U(\Omega))}$, the proof is complete.
 \end{proof}

\begin{defn}[Directional derivatives.\ {\cite[Def.\ 3.1]{AKR98}}]\normalfont
Let $F = \Phi(f_1^*, \ldots, f_k^*) \in \CylF(\U(\Omega))$ and $v \in C_*^\infty(\Omega, \R^n)$. We denote by $\phi$ the flow associated to $v$, i.e.
$$
\frac{d}{dt} \phi_t(x) = v(\phi_t(x)), \quad \phi_*(x) = x \in \Omega \, .
$$
The {\it directional derivative} $\nabla_v F(\gamma) \in T_\gamma\U(\Omega)$ is defined as
$$
\nabla_v F(\gamma):=\frac{d}{dt} F(\phi_t(\gamma))\Bigl|_{t=0}\, ,
$$
where $\phi_t(\gamma) := \sum_{x\in \gamma} \delta_{\phi_t(x)}$.
\end{defn}

\begin{defn}[Gradient of cylinder functions.\ {\cite[Def.\ 3.3]{AKR98}}] \normalfont \label{defn: grad}
	The gradient $\nabla_{\U(\Omega)}F$ of $F\in \CylF(\U(\Omega))$ at $\gamma\in \U(\Omega)$ is defined as the unique vector field $\nabla_{\U(\Omega)}F$ so that 
	$$\nabla_v F(\gamma) = \langle \nabla_{\U(\Omega)} F, v \rangle_{T_\gamma\U(\Omega)}, \quad \gamma \in \U(\Omega),\ v \in C_*^\infty(\Omega, \R^n).$$
By the expression \eqref{d:Cyl}, the gradient $\nabla_{\U(\Omega)}F$ can be written as
	\begin{equation} \label{eq:GF}
		\nabla_{\U(\Omega)} F(\gamma) = \sum_{i=1}^k \partial_i \Phi(f_1^\ast,\ldots, f_k^\ast)(\gamma) \nabla_{\R^n} f_i \in T_\gamma \U(\Omega) \, ,
	\end{equation}
	where $\nabla_{\R^n}$ is the gradient operator in $\R^n$. 
When $\Omega=\R^n$,  we simply write $\nabla:= \nabla_{\U(\R^n)}$ in the rest of the paper when no confusion occurs.
\end{defn} 
Notice that $\nabla_{\U(\Omega)} F \in \CylV(\U(\Omega))$ for any $F\in \CylF(\U(\Omega))$ by \eqref{eq:GF}.  In particular, for any $F \in \CylF(\Omega)$, it holds that $\nabla_{\U(\Omega)} F\in L^p(T\U(\Omega), \pi_\Omega)$ for any $1\le p<\infty$ by Proposition \ref{prop:LPV}.

%The following proposition follows immediately by the definition of $\CylF(\Omega)$ and the property \eqref{eq:Int} of the Poisson measure $\p$. We give a proof for the reader's convenience.

 \begin{rem}[Ampleness of $L^\infty$-vector fields]\label{rem: bounded test} \normalfont
 By Proposition \ref{prop:LPV}, $\CylV_*(\U(\Omega))\subset L^p(T\U(\Omega), \pi_\Omega)$ for any $p\in [1,\infty)$, while the inclusion is false for $p=\infty$. See \cite[Example 4.35]{DS21} for a counterexample. However, $\CylV_*(\U(\Omega))$ can be approximated by the subspace of bounded cylinder vector fields with respect to the pointwise convergence and the  convergence in the $L^p(\U(\Omega)\to T\U(\Omega))$-norm for $1\le p<\infty$. Indeed, given $\e>0$ and $V= \sum_{i=1}^k F_i(\gamma) v_i(x) \in \CylV_*(\U(\Omega))$ it holds
 	\begin{equation*}
 		|V|_{T_\gamma\U}^2 = \sum_{i,j=1}^k F_i(\gamma)F_j(\gamma) \int_{\U(\Omega)} v_i(x)\cdot v_j(x) d\gamma(x) \, ,
 		\quad\quad
 		\frac{1}{1 + \e|V|_{T_\gamma\U}^2} \in \CylF(\U(\Omega)) \, ,
 	\end{equation*}
 	hence
 	\begin{equation*}
 		V_\e : = \frac{1}{1+\e |V|_{T_\gamma\U(\Omega)}^2} V \in \CylV(\U(\Omega)) \, .
 	\end{equation*}
    Finally, notice that for $\gamma\in \U(\Omega)$ it holds
    \begin{equation*}
    	| V - V_\e|_{T_\gamma\U(\Omega)} = \e \frac{|V|_{T_\gamma\U(\Omega)}^3}{1+\e |V|_{T_\gamma\U(\Omega)}^2} 
    	\le \e |V|_{T_\gamma\U(\Omega)}^3
    	\to 0 \, ,
    	\quad
    	\text{as $\e\to 0$}\, .
    \end{equation*}
     Moreover, for every $1\le p<\infty$ we have
     \begin{equation*}
     	\| V - V_\e \|_{L^p(T\U(\Omega))} \le \e \| V \|_{L^{3p}(T\U(\Omega))}^3 \to 0
     	\, ,
     	\quad
     	\text{as $\e\to 0$}\, .
     \end{equation*}	     		
 \end{rem}

We now define the adjoint operator of the gradient $\nabla_{\U(\Omega)}$. 
\begin{defn}[Divergence. {\cite[Def.\ 3.5]{AKR98}}] \normalfont 
  	Let $ 1< p <\infty$.
  	We say that $V\in L^p(T\U(\Omega), \pi_\Omega)$ is in the domain $\mathcal D(\nabla^*_{\U(\Omega)})$ of the divergence if there exists a unique function $\nabla^*_{\U(\Omega)}V \in L^p(\U(\Omega), \pi_\Omega)$ such that
	\begin{equation}\label{eq: IbP1}
	\int_{\U(\Omega)} \langle V, \nabla_{\U(\Omega)} F \rangle_{T_\gamma\U} d \pi_{{\Omega}}(\gamma) = - \int_{\U(\Omega)} F (\nabla_{\U(\Omega)}^* V) d \pi_{{\Omega}} \, ,
	\quad F\in \CylF(\U(\Omega)).
\end{equation}
When $\Omega=\R^n$,  we simply write $\nabla^*:= \nabla_{\U(\R^n)}^*$ in the rest of the paper when no confusion occurs.
\end{defn}

\begin{prop} \label{prop: IVPD}
	The following inclusion holds:
	$$\CylV_*(\U(\Omega)) \subset \mathcal D(\nabla^*_{\U(\Omega)}) \, .$$
	 Furthermore, for $V(\gamma, x)=\sum_{i=1}^m F_i(\gamma) v_i(x) \in \CylV_*(\U(\Omega))$, 
	\begin{align} \label{eq: DIV}
		\nabla_{\U(\Omega)}^* V(\gamma) = \sum_{i=1}^m \nabla_{v_i} F_i(\gamma) + \sum_{i=1}^m F_i(\gamma)(\nabla^*_{\R^n} v_i)^\ast \gamma \, ,
	\end{align}
where $\nabla^*_{\R^n}$ is the divergence operator in $\R^n$.
	In particular, $\nabla^*_{\U(\Omega)} V \in L^p(\U(\Omega), \p_\Omega)$ for every $p\in [1, \infty)$.
\end{prop}

\begin{proof}
	Let $r>0$ be such that $\text{supp}(v_i) \subset\Omega_r:= \{x\in \Omega\, : \, d(x, \partial \Omega) > r\}$. For any $\e< r/2$ we define $\phi_\e \in C^\infty_*(\Omega)$ satisfying $\phi=1$ on $\Omega_{\e}$.
	For any $i=1, \ldots, m$ we write $F_i = \Phi_i(f_{1,i}^*, \ldots, f_{k_i, i}^*)$ and set $F_i^\e := \Phi_i((\phi_\e f_{1,i})^*, \ldots, (\phi_\e f_{k_i, i})^*)$. Observe that $V_\e:= \sum_{i=1}^m  F_i^\e(\gamma) v_i \in \CylV_*(\U(\Omega))$ and also $V_\e \in \CylV(\U(\R^n))$ by construction. Furthermore, we note  that $F\in \CylF(\U(\Omega))$ can be extended to $\tilde{F} \in \CylF(\U(\R^n))$ with $F = \tilde{F}$ on $\U(\Omega)$ by extending each inner function $f_i \in C^\infty_c(\Omega)$ to $\tilde{f}_i \in C^\infty_c(\R^n)$ with $f_i = \tilde{f}_i$ on $\Omega$ (e.g., by Whitney's extension theorem). Thus, $\nabla^*_{\U(\Omega)}$ and $\nabla^*_{\U(\R^n)}$ defined in \eqref{eq: IbP1} are consistent, so that $\nabla_{\U(\Omega)}^* V_\e(\gamma) = \nabla_{\U(\R^n)}^* V_\e(\gamma)$. By~\cite[Prop.\ 3.1]{AKR98}, therefore, we have 
	%\footnote{The first equality is probably not true since the dual is taken for $\CylF(\Omega)$ in \eqref{eq: DIV}, whose inner functions $\{f_i\}$ are not necessarily vanishing at $\partial \Omega$. Thus, we do not have $F \circ {\rm pr}_\Omega(\gamma)=F(\gamma_\Omega)$ and we cannot apply \cite{AKR98} (integration by parts formula in the case of $\Omega=\R^n$) to this argument. We should use $\CylF_*(\Omega)$ for the duality in \eqref{eq: DIV}. Due to this change, the proof of the closability of $(\mathcal E_{\U(\Omega)}, \CylF(\Omega))$ must be modified, which is, however,  not a serious problem since the closability is known to be true by the other argument. As far as I checked, there are no other place that we need to modify due to this change, but we should check it again in the next round. }
	\begin{align*} 
		{\nabla_{\U(\Omega)}^* V_\e(\gamma) = \nabla_{\U(\R^n)}^* V_\e(\gamma)} 
		& = \sum_{i=1}^m \nabla_{v_i} F_i^\e(\gamma) + \sum_{i=1}^m F_i^\e(\gamma)(\nabla^*_{\R^n} v_i)^\ast \gamma
		\\&
		= \sum_{i=1}^m \nabla_{v_i} F_i(\gamma) + \sum_{i=1}^m F_i^\e(\gamma)(\nabla^*_{\R^n} v_i)^\ast \gamma \, .
	\end{align*}
	Here we used the fact that $F_i(\gamma) = F^\e_i(\gamma)$ for any $\gamma$ concentrated on the support of $v_i$. The sought conclusion \eqref{eq: DIV} follows from the observation that $F_i^\e \to F_i$ in $L^p(\U(\Omega), \p_\Omega)$ and $V_\e \to V$ in $L^p(T\U(\Omega), \pi_{\Omega})$ combined with \eqref{eq: IbP1}. The last assertion is then a direct consequence from Proposition \ref{prop:LPV} and \eqref{eq: DIV}.  
%	Let us note that, arguing as in Proposition \ref{prop:LPV}, if $V\in \CylV_*(\U(\Omega))$ then $\nabla^*_{\U(\Omega)} V \in L^p(\U(\Omega), \p)$ for any $p\in [1, \infty)$.
\end{proof}

\paragraph{Sobolev spaces} 
We now introduce the $(1,p)$-Sobolev space. The operator
\begin{equation}
	\nabla_{\U(\Omega)} : \CylF(\U(\Omega)) \subset L^p(\U(\Omega), \p_\Omega)  \to \CylV(\U(\Omega)) \,
\end{equation}
is densely defined and closable. The latter fact is a direct consequence of the integration-by-parts formula \eqref{eq: DIV}. Indeed, we observe that, if $F_n\in \CylF(\U(\Omega))$, $F_n \to 0$ in $L^p(\U(\Omega), \p_\Omega)$, and $\nabla_{\U(\Omega)} F_n \to W$ in $L^p(T\U(\Omega), \pi_\Omega)$, then for any $V\in \CylV_*(\U(\Omega))$, it holds
\begin{equation*}
	\int_{\U(\Omega)} \langle V, W \rangle_{T_\gamma \U} d \pi_\Omega(\gamma)
	= 
	\lim_{n \to \infty} \int_{\U(\Omega)} \langle V, \nabla_{\U(\Omega)} F_n \rangle_{T_\gamma \U} d \pi_\Omega(\gamma)
	=
	-\lim_{n \to \infty} \int_{\U(\Omega)} (\nabla_{\U(\Omega)}^*V) F_n d \pi_\Omega(\gamma)
	= 0 \, ,
\end{equation*}
yielding $W=0$ as a consequence of the density of $\CylV_*(\U(\Omega))$ in $L^p(T\U(\Omega), \pi_\Omega)$ by Proposition~\ref{prop:LPV}. The above argument justifies the following definition.

\begin{defn}[$H^{1,p}$-Sobolev spaces] \normalfont \label{defn: WS}
	Let $1<p<\infty$. We define $H^{1,p}(\U(\Omega), \p_\Omega)$ as the closure of $\CylF(\U(\Omega))$ in $L^p(\U(\Omega), \p_\Omega)$ with respect to the following 
	$(1,p)$-Sobolev norm: 
	\begin{equation*}
		\|F\|^p_{H^{1,p}(\U(\Omega))} := \|F\|^p_{L^p(\U(\Omega))} + \|\nabla_{\U(\Omega)} F\|^p_{L^p(T\U(\Omega))} \, .
	\end{equation*}
We set $\|F\|_{H^{1,p}}:=\|F\|_{H^{1,p}(\U(\R^n))}$. When $p=2$, we write the corresponding Dirichlet form (i.e., a closed form satisfying the unit contraction property \cite[Def.\ 4.5]{MR92})  by 
$$
\EE_{\U(\Omega)}(F,G):=\int_{\U(\Omega)} \bigl\langle \nabla_{\U(\Omega)}F, \nabla_{\U(\Omega)}G \bigr\rangle_{T_\gamma\U(\Omega)} d\p_\Omega(\gamma), \quad F, G \in H^{1,2}(\U(\Omega), \p_\Omega) \, . 
$$
We set $\EE:=\EE_{\U(\R^n)}$.
\end{defn}

%\begin{rem}[$H^{1,p}$ and $W^{1,p}$]\normalfont
%As the definition of the Sobolev space in Definition \ref{defn: WS} is the smallest extension of $\CylF(\U(\Omega))$, one might be tempted to write $H^{1,p}$ instead of $W^{1,p}$. However, the $L^p$-uniqueness result (see \cite[Cor.\ 4.19]{DS21}) combined with the coincidence of the Bessel Sobolev space and $H^{1,p}$ (\purple{reference}) implies the smallest extension coincides with the largest one, because of which we adopt the notation $W^{1,p}$.\footnote{Normally $W^{1,2}$ is used for the maximal domain while $H^{1,2}$ is used for the smallest domain. If we use $W^{1,2}$ in this paper, then this remark seems necessary. }
%\end{rem}

\begin{rem}[The case of $p=1$]\normalfont
As is indicated by \eqref{eq: DIV}, it is not true in general that $\nabla_{\U(\Omega)}^*V \in L^\infty(\U(\Omega), \p_\Omega)$ since arbitrarily many finite particles can be concentrated on the supports of inner functions of $F \in \CylF(\U(\Omega))$ and vector fields $v_i$. See \cite[Example 4.35]{DS21} for more detail. Due to this fact, the standard integration by part argument for the closability of the operator $\nabla_{\U(\Omega)}: \CylF(\U(\Omega)) \to  \CylV(\U(\Omega)) \subset L^p(T\U(\Omega), \pi_\Omega)$ does not work in the case of $p=1$. For this reason,  we restricted the definition of the $H^{1,p}$-Sobolev spaces to the case $1<p<\infty$ in Definition \ref{defn: WS}.
\end{rem}

Once the closed form $\EE_{\U(\Omega)}$ on $L^2(\U(\Omega), \p_\Omega)$ is constructed, one can define the infinitesimal generator on $L^2(\U(\Omega), \p_\Omega)$ as the unique non-positive definite self-adjoint operator 
%by general {\color{red}theory of functional analysis}.{\color{blue}EB: Here we can simply erase the red part, we do not need a reference}%, which is called {\it the $L^2(\U(\Omega), \p)$-Laplace operator}.
\begin{defn}[Laplace operator\ {\cite[Theorem 4.1]{AKR98}}]\normalfont \label{defn: Laplace}
{\it The $L^2(\U(\Omega), \p_\Omega)$-Laplace operator} $\Delta_{\U(\Omega)}$ with domain $\mathcal D(\Delta_{\U(\Omega)})$ is defined as the unique non-positive definite self-adjoint operator $\Delta_{\U(\Omega)}$ so that
$$\mathcal{E}_{\U(\Omega)}(F, G) = -\int_{\U(\Omega)} (\Delta_{\U(\Omega)}F)G d\p_\Omega, \quad F \in \mathcal D(\Delta_{\U(\Omega)}),  \ G \in \mathcal D(\EE_{\U(\Omega)})\, .$$
\end{defn}
In the case of $\Omega=\R^n$, employing \eqref{eq:GF} and \eqref{eq: DIV}, one can compute that 
\begin{align*}
\Delta_{\U(\R^n)} F:=\nabla^*_{\U(\R^n)} \nabla_{\U(\R^n)} F, \quad F \in \CylF(\U(\R^n))
\, .
\end{align*}
When $\Omega=\R^n$, we shortly write $\Delta=\Delta_{\U(\R^n)}$ in the rest of the paper when no confusion occurs.
%\purple{By \cite[Cor.\ 4.19]{DS21}, the operator $\Delta_{\U(\R^n)}$ on $\CylF(\U(\R^n))$ is closable w.r.t.\ the graph norm defined by 
%$$\|\Delta_{\U(\R^n)}\cdot\|_{L^p(\U(\R^n))} + \|\cdot\|_{L^p(\U(\R^n))},$$
%for any $1 \le p<\infty$. Its closure $(D(\Delta^{(p)}_{\U(\R^n)}), \Delta^{(p)}_{\U(\R^n)})$ defines the {\it $L^p$-Laplace operator}.  We set $\Delta^{(p)} := \Delta^{(p)}_{\U(\R^n)}$. If no confusion occurs, we simply write $\Delta$ without specifying the exponent $p$.}

\medskip

Let $\{T^{\U(\Omega)}_t\}$  and $\{G^{\U(\Omega)}_{\alpha}\}$ be the strongly continuous Markovian $L^2$-semigroup and resolvent, respectively, corresponding to the energy $\EE_{\U(\Omega)}$. 
We set $G_{\alpha}:=G^{\U(\R^n)}_{\alpha}$ and $T_t:=T^{\U(\R^n)}_t$. 
By the Riesz--Thorin Interpolation Theorem, $T^{\U(\Omega)}_t$ and $\{G^{\U(\Omega)}_\alpha\}$ can be uniquely extended to $L^p$ strongly continuous Markovian semigroup and resolvent, respectively, for every $1 \le p <\infty$  (see e.g.\ \cite[Section 2, p.\ 70]{S97}).

\subsection{Product semigroups and exponential cylinder functions}
In this section, we relate the finite-product semigroup on $\Omega^{\times k}$ and the semigroup on $\U^k(\Omega)$ when $\Omega \subset \R^n$ is a bounded closed domain with smooth boundary. To this aim we introduce a class of test functions, which is suitable to compute the semigroups.

  \begin{defn}[Exponential cylinder functions. {\cite[(4.12)]{AKR98}}] \label{defn: Ecyl} \normalfont 
  Let $\Omega \subset \R^n$ be a bounded closed domain with smooth boundary, or $\Omega=\R^n$. The class ${\rm ECyl}(\U(\Omega))$ of {\it exponential cylinder functions} is defined as the vector space spanned by
  $$
  \left\lbrace
  \exp\bigl\{\log (1+f)^\ast \bigr\} \,  : \, 
  f \in \mathcal D(\Delta_{\Omega})  \, , \Delta_{\Omega} f \in L^1(\Omega),\ -\delta \le f \le 0 \, \,  \text{for some $\delta\in (0,1)$} 
  \right\rbrace \, .
  $$  
Here $(\Delta_{\Omega}, \mathcal D(\Delta_{\Omega}))$ denotes the $L^2$-Neumann Laplacian on $\Omega$ when $\Omega \subsetneq \R^n$.
\end{defn}
 The space ${\rm ECyl}(\U(\Omega))$ is dense in $L^p(\U(\Omega), \pi_{\Omega})$ for any $1 \le p <\infty$ (see \cite[p.\ 479]{AKR98}). Noting that $\Delta_{\Omega}$ is essentially self-adjoint on the core $C_c^\infty(\Omega) \cap \{\frac{\partial f}{\partial n}=0\ \text{in} \ \partial \Omega\}$, where $\frac{\partial}{\partial n}$ is the normal derivative on $\partial \Omega$,  and the corresponding $L^2$-semigroup $\{T_t^{\Omega}\}$ is conservative,  we can apply the same argument in the proof of \cite[Prop.\ 4.1]{AKR98} to obtain the following: $T_t^{\U(\Omega)} {\rm ECyl}(\U(\Omega)) \subset {\rm ECyl}(\U(\Omega))$ and 
 \begin{align} \label{e: SE}
 T_t^{\U(\Omega)} \exp\bigl\{\log (1+f)^\ast \bigr\} = \exp\Bigl\{\log \bigl(1+(T_t^{\Omega}f) \bigr)^\ast \Bigr\} \, .
 \end{align}
% where $\{T_t^{\Omega}\}$ is the $L^2$-semigroup corresponding to $\Delta_{\Omega}$.

Let $T_t^{\Omega, \otimes k}$ be the $k$-tensor semigroup of $T_t^\Omega$, i.e. the unique semigroup in $L^p(\Omega^{\times k})$ satisfying
\begin{equation}\label{z51}
	T_t^{\Omega, \otimes k} f(x_1,\ldots, x_k):=T_t^{\Omega}f_1(x_1) \cdot \cdot \cdot T_t^{\Omega}f_k(x_k) \, ,
	\quad \text{for every $k\in \N$}\, ,
\end{equation}
whenever $f(x_1, \ldots , x_k) = f_1(x_1) \cdot \cdot \cdot f_k(x_k)$ with $f_i\in L^\infty(\Omega)$ for  $i=1, \ldots, k$.

 \begin{prop} \label{prop: IDP}
 Let $\Omega \subset \R^n$ be a bounded closed domain with smooth boundary and $1 \le p <\infty$. 
 For $F\in L^p(\U^k(\Omega), \p_{\Omega}|_{\U^k(\Omega)})$, it holds
 \begin{align} \label{eq: TPT}
 T_t^{\Omega, \otimes k} (F \circ \mathsf s_k) = (T_t^{\U^k(\Omega)}F) \circ \mathsf s_k, \quad \text{$\mathbf S^{kn}_{\Omega^{\times k}}$-a.e.}
 \end{align} 
 \end{prop}
 \proof
 Since ${\rm ECyl}(\U^k(\Omega))$ is dense in $L^p(\U^k(\Omega))$ for any $1 \le p <\infty$, it suffices to show \eqref{eq: TPT} only for $F \in {\rm ECyl}(\U^k(\Omega))$. Furthermore, we can reduce the argument to the case $F=\exp\bigl\{\log (1+f)^\ast\}$ by using the linearity of semigroups. From \eqref{z51} and  \eqref{e: SE} we get
 \begin{align*}
 T_t^{\Omega, \otimes k} (F \circ \mathsf s_k)(x_1, \ldots, x_k) & = T_t^{\Omega, \otimes k} (\exp\bigl\{\log (1+f)^\ast\} \circ \mathsf s_k)(x_1, \ldots, x_k)
 \\
 &= T_t^{\Omega, \otimes k} \Bigl(\prod_{i=1}^k (1+f)(\cdot_i)  \Bigr) (x_1, \ldots, x_k)
 \\
 &= \prod_{i=1}^k \Bigl(1+T_t^{\Omega}f(x_i) \Bigr)
 \\
 &=\exp\Bigl\{\log \bigl(1+(T_t^{\Omega}f)^\ast \bigr) \Bigr\} \circ \mathsf s_k(x_1,\ldots, x_k)
 \\
 &= T_t^{\U(\Omega)} \exp\bigl\{\log (1+f)^\ast \bigr\} \circ \mathsf s_k(x_1,\ldots, x_k) \, .\qedhere
 \end{align*}

\subsection{Suslin sets}
Let $X$ be a set.  We denote by $\N^\N$ the space of all infinite sequences $\{n_i\}_{i \in \N}$ of natural numbers. For $\phi \in \N^\N$, we write $\phi|_l \in \N^l$ for the restriction of $\phi$ to the first $l$ elements, i.e., $\phi|_l:=(\phi_i: 1 \le i \le l)$. Let $\mathcal S:=\cup_{l \in \N}\N^l$, and for $\sigma \in \mathcal S$, we denote the length of the sequence $\sigma$ by $\#\sigma:=\#\{\sigma_i\}$. Let $\E \subset 2^{X}$ be a family of subsets in $X$. We write $\mathcal S(\E)$ for the family of sets expressible in the following form:
$$\bigcup_{\phi \in \N^\N}\bigcap_{l \ge 1} E_{\phi|_l},$$
for some family $\{E_{\sigma}\}_{\sigma \in \mathcal S}$ in $\E$. A family $\{E_{\sigma}\}_{\sigma \in \mathcal S}$ is called {\it Suslin scheme}; the corresponding set $\cup_{\phi \in \N^\N}\cap_{l \ge 1} E_{\phi|_l}$ is its {\it kernel}; the operation
$$\{E_{\sigma}\}_{\sigma \in \mathcal S} \mapsto \bigcup_{\phi \in \N^\N}\bigcap_{l \ge 1} E_{\phi|_l},$$
is called {\it Suslin's operation}.  We denote by $\mathcal S(\E)$ the family of sets generated from sets in $\E$ by Suslin's operation, whose elements are called an $\E$-Suslin set (or simply Suslin set). It is known that $\mathcal S(\E)$ is closed under Suslin's operation (\cite{Sou17}, and e.g., \cite[421D Theorem]{Fre03}).  If $E_\sigma$ is compact for all $\sigma \in \mathcal S$, we call $\{E_{\sigma}\}_{\sigma \in \mathcal S}$ a {\it compact} Suslin scheme. 
We say that $\{E_\sigma\}_{\sigma \in \mathcal S}$ is regular if $E_\sigma \subset E_{\tau}$ whenever $\# \tau \le \# \sigma$ and $\sigma_i \le \tau_i$ for any $i<\#\sigma$ (\cite[421X (n) $\&$ 422H Theorem (b)]{Fre03}). 
%If $X$ is a topological space and 

In the following remark, we list basic properties of Suslin sets in a Polish space and relations to Choquet capacities and Borel measures. In the rest of this section, we assume that 
\begin{align} \label{asmp: SS} 
&\text{$(X, \tau)$ is a Polish space,\ \  $\c$ is a Choquet capacity on $X$,\ \ $\mu$ is a bounded Borel measure,} \ \notag
\\
&\text{$\E:=\mathcal C(X):=\{C: \text{closed set in $X$}$\}}\ .
\end{align}
%and $A \in \mathcal S(\E)$ is simply called a {\it Suslin set}. 
We refer the readers to, e.g., \cite[432I Definition]{Fre03} for the definition of Choquet capacity.

\begin{rem} \label{rem: SS} \normalfont Under the assumption \eqref{asmp: SS}, the following hold:
\begin{itemize}
\item[(i)] Every Borel set is a Suslin set, i.e., $\mathcal B(\tau) \subset \mathcal S(\E)$ (e.g., \cite[423B(a) and 423F(a)]{Fre03});
\item[(ii)] Every Suslin set is $\mu$-measurable, i.e., $\mathcal S(\E) \subset \overline{\mathcal B}^{\mu}(\tau)$ (e.g., \cite[431B Corollary]{Fre03});
\item[(iii)] Let $A$ be a Suslin set in $X$. Then, $A$ is the kernel of a {\it compact regular} Suslin scheme $\{E_{\sigma}\}_{\sigma \in \mathcal S}$. Furthermore, it holds that 
\begin{align} \label{eq: APX1}
\c(A) = \sup_{\psi \in \N^\N}\c(A_{\psi}), \quad A_{\psi} = \bigcup_{\phi \le \psi}\bigcap_{l \ge 1} E_{\phi|_l},
\end{align}
whereby $\phi \le \psi$ means that $\phi(l) \le \psi(l)$ for all $l \in \N$ (e.g., \cite[423B Theorem $\&$ the proof of 432J Theorem]{Fre03}). By the regularity of $\{E_{\sigma}\}_{\sigma \in \mathcal S}$, \eqref{eq: APX1} can be reduced to the following form:
\begin{equation} \label{eq: APX}
\c(A) = \sup_{\psi \in \N^\N}\c(A_{\psi}), \quad A_{\psi} = \bigcap_{l \ge 1} E_{\psi|_l}, \quad \psi \in \N^\N;
\end{equation}
\item[(iv)] A subset $A \subset X$ is Suslin iff $A$ is analytic iff $A$ is $K$-analytic (\cite[423E Theorem (b)]{Fre03}. See \cite[422F, 423A Definitions]{Fre03} for the definitions of {\it $K$-analyticity} and {\it analyticity} respectively). As every $K$-analytic set is capacitable (e.g.,\cite[432J]{Fre03}),  in particular, we have that $\c(A)$ is well-defined for every Suslin set $A$ as 
\begin{align}\label{e:CB}
\c(A)=\sup\{\c(K): K \subset A \ \text{compact}\} \ .
\end{align}
%\item 
\end{itemize}
\end{rem}

\section{Finite-codimensional Poisson measures} \label{s:FCP}
In this section, we construct finite-codimensional Poisson measures on $\U(\R^n)$. As a first step we prove measurability results for sections of \textit{Suslin} subsets of the configuration space.

\subsection{Measurability of sections of Suslin sets}
Let $B \subset \R^n$. For $A \subset \U(\R^n)$ and $\eta \in \U(B)$, the section $A_{\eta, B} \subset \U(B^c)$ of $A$ at $\eta$ is defined as 
%\footnote{EB; That is a very non intuitive definition, it would be more natural to pick $\eta\in B^c$ and write $A_{\eta, B}=\{\gamma \in \Upsilon(B): \gamma + \eta \in A\}$}
\begin{align} \label{eq: DIS}
A_{\eta, B}=\{\gamma \in \Upsilon(B^c): \gamma + \eta \in A\}.
\end{align}
The subset of  $A_{\eta, B}$ consisting of $k$-particle space $\U^{k}(B^c)$ is denoted by $A_{\eta, B}^k:=A_{\eta, B} \cap \U^{k}(B^c).$
To shorten the notation we often write $A_{\eta,r}$ in place of $A_{\eta,B_r^c}$, where $\overline{B_r}$ is the closed ball centred at the origin.

\begin{lem} \label{lem: sec}
Let $B \subset \R^n$ be a Borel set. If  $A$ is Suslin in $\U(\R^n)$ then $A_{\eta, B}^k$ is Suslin in $\U^k(B^c)$ for every $\eta \in \U(B)$, $k \in \N$ and $r>0$.
\end{lem}
\proof
We can express $A_{\eta, B}={\rm pr}_{B^c}\bigl({\rm pr}_{B}^{-1}(\eta) \cap A \bigr)$.
The set ${\rm pr}_{B}^{-1}(\eta) \cap A$ is Suslin in $\U(\R^n)$ whenever $A$ is Suslin. Set $\U_{\eta, B}(\R^n)={\rm pr}_B^{-1}(\eta) \cap \U(\R^n)$, which is Suslin.
The map ${\rm pr}_{B^c}: \U_{\eta, B}(\R^n) \to \U(B^c)$ is continuous.
Thus, $A_{\eta, B}$ is the continuous image ${\rm pr}_{B^c}\bigl({\rm pr}_{B}^{-1}(\eta) \cap A \bigr)$ of the Suslin set ${\rm pr}_{B}^{-1}(\eta) \cap A$ in the Suslin Hausdorff space $ \U_{\eta, B}(\R^n)$. Hence, $A_{\eta, B}$ is Suslin (\cite[423B Proposition (b) \& 423E Theorem (b)]{Fre03}). 
Since $A_{\eta,B}^k= A_{\eta,B} \cap \U^k(B^c)$ and $\U^k(B^c)$ is Borel in $\U(B^c)$, we conclude that  $A^k_{\eta, B}$ is Suslin. 
\qed

\begin{comment}
We can express $A_{\eta, B}={\rm pr}_r\bigl({\rm pr}_{B}^{-1}(\eta) \cap A\bigr)$, recalling that $p_r$ and $q_r$ are the projections $\U(\R^n) \to \U(B_r)$ and $\U(\R^n) \to \U(B_r^c)$ respectively.  Thus, the set $q_r^{-1}(\eta) \cap A$ is Borel (resp.\ Suslin) in $\U(\R^n)$ whenever $A$ is Borel (resp.\ Suslin).
%since $A$ is Borel in $\U(\R^n)$ and $q_r$ is a Bore map 
%lthe pre-image of one point $\{\eta\}$ by the Borel measurable map $q_r$. 
Recall $\U_{\eta, r}(\R^n)=q_r^{-1}(\eta) \cap \U(\R^n)$.
%and equip $\U_{\eta, r}(\R^n)$ with the subset topology from the vague topology in $\U(\R^n)$. 
Since $B_r$ is closed in $\R^n$, the set $\U_{\eta, r}(\R^n)$ is closed with respect to the vague topology in $\U(\R^n)$. Since every closed set in a Polish space is Polish (\cite[6.1.11.\ Example]{Bog07}), the set $\U_{\eta, r}(\R^n)$ is Polish. Furthermore, the map $p_r: \U_{\eta, r}(\R^n) \to \U(B_r)$ is injective. Since any continuous injective image of  Borel (resp.\ Suslin) subsets in a Polish space is Borel (resp.\ Suslin) by \cite[423I Theorem]{Fre03} (resp.\ \cite[422G Theorem \& 423E Theorem (b)]{Fre03}),  we conclude that $A_{\eta, r}$ is Borel (resp.\ Suslin) for any $\eta \in \U(B_r^c)$ and $r>0$. By noting that $A_{\eta,r}^k= A_{\eta,r} \cap \U^k(B_r)$ and that $\U^k(B_r)$ is a closed set in $\U(B_r)$, we conclude that  $A^k_{\eta, r}$ is Borel (resp.\ Suslin). 
\qed
\end{comment}

\begin{lem} \label{lem: secS}
Let $B \subset \R^n$ be an open set. Let $A \subset \U(\R^n)$ be the kernel of a compact Suslin's scheme $\{E_{\sigma}\}_{\sigma \in \mathcal S}$, i.e., $A=\cup_{\phi \in \N^\N}\cap_{l \ge 1} E_{\phi|_l}$ with $E_\sigma$ compact for any $\sigma \in \mathcal S$. Then, $A_{\eta, B}$ is the kernel of the compact Suslin scheme $\{(E_{\sigma})_{\eta, r}\}_{\sigma \in \mathcal S}$. 
%\purple{(Write for general set instead of balls).}
\end{lem}
\proof 
By expressing $(E_{\sigma})_{\eta, B} = {\rm pr}_{B^c}\bigl(\U_{\eta, B}(\R^n) \cap E_{\sigma}\bigr)$, where $\U_{\eta, B}(\R^n)={\rm pr}_B^{-1}(\eta) \cap \U(\R^n)$, we see that $(E_{\sigma})_{\eta, B}$ is compact since $\U_{\eta, B}(\R^n)$ is closed, $E_\sigma$ is compact by the hypothesis, ${\rm pr}_{B^c}$ is continuous on $\U_{\eta, B}(\R^n)$ and every continuous image of a compact set is compact. To see that $A_{\eta, B}$ is the kernel of $\{(E_{\sigma})_{\eta, r}\}_{\sigma \in \mathcal S}$, 
\begin{align*}
A_{\eta, B} 
= p_r\bigl(\U_{\eta, B}(\R^n) \cap A \bigr) 
&= p_r\Biggl(\U_{\eta, B}(\R^n) \cap \bigcup_{\phi \in \N^\N}\bigcap_{l \ge 1} E_{\phi|_l} \Biggr)
= p_r\Biggl( \bigcup_{\phi \in \N^\N}\bigcap_{l \ge 1} \U_{\eta, B}(\R^n) \cap E_{\phi|_l} \Biggr)
\\
&= \bigcup_{\phi \in \N^\N}\bigcap_{l \ge 1} p_r \bigl( \U_{\eta, B}(\R^n) \cap E_{\phi|_l}\bigr)
= \bigcup_{\phi \in \N^\N}\bigcap_{l \ge 1} (E_{\sigma})_{\eta, B} \, . \qedhere
\end{align*}

\subsection{Localised finite-codimensional Poisson measures}

 In this section, we construct {\it a localised} version of the $m$-codimensional Poisson measure $\rho_r^m$, which will be used to construct the $m$-condimensional Poisson measure by taking the limit for $r \to \infty$. We also show that Suslin sets are contained in the domain of the finite-codimensional Poisson measure.  
 
Let $A\subset \U(\R^n)$ be a Suslin subset.
By Lemma \ref{lem: sec}, the set $A^k_{\eta,r} = A^k_{\eta, B^c_r}$ is Suslin.  Since $\S^{m, k}_{B_r}$ is a Choquet capacity, the expression $\S^{m, k}_{B_r}(A^k_{\eta,r})$ is well-defined and satisfies~\eqref{e:CB}, which in particular implies that $A^k_{\eta,r}$ is a $\S^{m, k}_{B_r}$-measurable set.
%all analytic sets are contained in the completion of the domain of $\S^{m, k}_{B_r}$ (see e.g.\ \cite[432J]{Fre03}),  hence $\S^{m, k}_{B_r}(A^k_{\eta,r})$ is well-defined.
We define the domain $\DD^m$ of the $m$-codimensional measures by
\begin{equation} \label{eq: DM}  
\DD^m:=\bigcap_{r>0} \mathscr D^m_r, 	
\end{equation}
where the localised domain $\DD^m_r$ is defined by
\begin{equation*}
	\DD^m_r:=\{A \subset \U(\R^n) \, :\,  \text{the map $\U(B_r^c) \ni \eta \mapsto \S^{m, k}_{B_r}(A^k_{\eta,r})$ is $\p_{B_r^c}$-measurable for every $k$}\} \, .
\end{equation*}
We first introduce the $m$-codimensional Poisson measure on the configuration space $\U(B_r)$ over the ball $B_r$.
\begin{defn} \normalfont \label{defn: LFP1}
The {\it $m$-codimensional Poisson measure $\rho^m_{\U(B_r)}$ on $\U(B_r)$} is defined as  
\begin{equation}
	\rho^m_{\U(B_r)}(A)
	:= e^{-\SS^{n}(B_r)}\sum_{k=1}^\infty \S^{m, k}_{B_r}(A^k) \quad \text{for every Suslin set $A$ in $\U(B_r)$}\, ,
\end{equation}
where $A^k=A \cap \U^k(B_r)$.
\end{defn}

\begin{rem} \normalfont
Notice that $\rho^0_{\U(B_r)} = \pi_{B_r}$, in other words the $0$-codimension Poisson measure $\rho^0_{\U(B_r)}$ on $\U(B_r)$ is the Poisson measure $\p_{B_r}$ on $\U(B_r)$. It can be shown by noting that the $m$-dimensional spherical Hausdorff measure $\SS^m$ and the $n$-dimensional Lebesgue measure $\mathbf L^n$ coincide when $m=n$ (see Remark \ref{rem: CSH}).
\end{rem}
We introduce the localised $m$-codimensional Poisson measure on $\U(\R^n)$ by averaging the $m$-codimensional Poisson measure $\rho^m_{\U(B_r)}$ by means of $\p_{B_r^c}$.

\begin{defn} \normalfont \label{defn: LFP}
The {\it localised $m$-codimensional Poisson measure $\rho^m_r$ on $\U(\R^n)$} is defined by 
\begin{align} \label{def: sm}
\rho^m_r(A) 
%& := {\color{red} e^{-\SS^{n}(B_r)}}  \sum_{k=1}^\infty\int_{\U(B_r^c)}\S^{m, k}_{r}(A^k_{\eta,r})\ d\p_{B_r^c}(\eta)
%\\& 
 = \int_{\U(B_r^c)} \rho^m_{\U(B_r)}(A_{\eta,r}) d \p_{B_r^c}(\eta), \quad  A \in \DD^m.
\end{align}
\end{defn}
Before investigating the main properties of $\rho^m_r$, we check that sufficiently many sets are contained in $\DD^m$, i.e. we show that all Suslin sets are contained in the domain $\DD^m$ for $m \le n$.

\begin{prop} \label{prop: meas}
Any Suslin set in $\U(\R^n)$ is contained in $\DD^m$ for $m \le n$. 
\end{prop}

\begin{proof}
Let $A \subset \U(\R^n)$ be a Suslin set.  
Let $\{E_{\sigma}\}_{\sigma \in \mathcal S}$ be a Suslin scheme whose kernel is $A$. Noting that $\U(B_r^c)$ is Polish, by applying (i) of Remark \ref{rem: SS} with $X=\U(B_r^c)$ and $\mu=\p_{B_r^c}$, any Suslin set is $\p_{B_r^c}$-measurable. Hence, it suffices to show that every super-level set $\{\eta: \S^{m, k}_{B_r}(A^k_{\eta,r}) > a\}$ is Suslin for any $a \in \R$, $r>0, k \in \N$ and $m \le n$. 
Note that $A^k_{\eta, r}$ is Suslin by Lemma \ref{lem: sec}, whence the expression $\{\eta: \S^{m, k}_{B_r}(A^k_{\eta,r}) > a\}$ is well-defined as was discussed in the paragraph before \eqref{eq: DM}.

Since $\U(\R^n)$ is Polish, by using (iii) in Remark \ref{rem: SS}, we may assume that  $\{E_{\sigma}\}_{\sigma \in \mathcal S}$ is a compact regular Suslin scheme. By Lemma \ref{lem: secS} and $\U(B_r)=\sqcup_{k \in \N} \U^k(B_r)$, we see that $A^k_{\eta, r} \subset \U^k(B_r)$ is the kernel of the compact regular Suslin scheme $\{(E_{\sigma})^k_{\eta, r}\}_{\sigma \in \mathcal S}$, whereby $(E_{\sigma})^k_{\eta, r}:=(E_{\sigma})_{\eta, B_r^c} \cap \U^k(B_r)$. 
Since $\S^{m, k}_{B_r}$ is an outer measure on $\U^k(B_r)$ by construction, $\S^{m, k}_{B_r}$ is a Choquet capacity on $\U^k(B_r)$.  Hence, by applying \eqref{eq: APX} in (iii) of Remark \ref{rem: SS} with $X=\U^k(B_r)$ and $\c=\S^{m, k}_{B_r}$, we obtain that 
\begin{align*}
\S^{m, k}_{B_r}(A^k_{\eta, r}) = \sup_{\psi \in \N^\N}\S^{m, k}_{B_r}((A^k_{\eta, r})_{\psi}), \quad (A^k_{\eta, r})_{\psi}=\bigcap_{l \ge 1} (E_{\psi|_{l}})^k_{\eta, r}, \quad \psi \in \N^\N.
\end{align*}
Thus, noting the monotonicity $\S^{m, k}_{B_r, \e} \le \S^{m, k}_{B_r, \delta}$ $(\delta \le \e)$ of the $\e$-Hausdorff measure defined in \eqref{SH}, the super-level set $\{\eta: \S^{m, k}_{r}(A^k_{\eta, r}) > a\}$ can be expressed in the following way:
\begin{align*}
\{\eta: \S^{m, k}_{B_r}(A^k_{\eta, r}) > a\} =  \bigcup_{\e>0} \bigcup_{\psi \in \N^\N} \{\eta: \S^{m, k}_{B_r, \e}\bigl(  (A^k_{\eta, r})_{\psi} \bigr) > a\}.
\end{align*}
Since the space $\mathcal S(\E)$ of Suslin sets is closed under Suslin's operation, it suffices to show that $\{\eta: \S^{m, k}_{B_r, \e}\bigl( ( A^k_{\eta, r})_{\psi} \bigr) >a\}$ is Suslin.

We equip $\U^k(B_r)$ with the $L^2$-transportation distance $\d_{\U^k}$ as defined in \eqref{d:L2T}, and equip $\U(B_r^c)$ with some distance $d$ generating the vague topology.  
By Proposition \ref{prop: usc} and noting that $(A^k_{\eta, r})_{\psi}$ is compact and that $\S^{m, k}_{B_r, \e}$ is (up to constant multiplication) the $m$-codimensional $\e$-spherical Hausdorff measure on $\U^k(B_r)$ associated with $\d_{\U^k}$, we conclude that $\{\eta: \S^{m, k}_{B_r, \e}\bigl( ( A^k_{\eta, r})_{\psi} \bigr)>a\}$ is open in $\U(B_r^c)$ for any $a \in \R$, $r>0, k \in \N$ and $m \le n$.  
\end{proof}

\subsection{Finite-codimensional Poisson measures} \label{subsec: finite-dim measure}
In this section, we construct the $m$-codimensional Poisson measure on $\U(\R^n)$, which is the first main result of this paper. 
By Proposition \ref{prop: meas}, the set function $\rho_r^m$ given in \eqref{def: sm} turned out to be well-defined in the sense that the space $\mathcal S(\E)$ of all Suslin sets in $\U(\R^n)$ is contained in its domain $\DD^m$. 
We show the following monotonicity result which allows us to pass to the limit of $\rho_r^m$ as $r\to \infty$.

\begin{thm} \label{thm: MI}
	The map $r \mapsto \rho_r^m(A)$ is monotone non-decreasing for any $A \in \mathcal S(\E)$. 
\end{thm}
The proof of Theorem \ref{thm: MI} is given at the end of this section. We can now introduce the {\it $m$-codimensional Poisson measure} on $\U(\R^n)$ as the monotone limit of $\rho^m_r$ on the space $\mathcal S(\E)$ of Suslin sets:
\begin{align} \label{eq: MCH}
	\rho^m(A) = \lim_{r \to \infty}\rho^m_r(A), \quad \forall A \in \mathcal S(\E).
\end{align} 
\begin{defn}[$m$-codimensional Poisson Measure] \label{defn: MCH} \normalfont
	Let $\mathfrak D^m$ be the completion of $\mathcal S(\E)$ with respect to $\rho^m$. 
	The measure $(\rho^m, \mathfrak D^m)$ is called the {\it $m$-codimensional Poisson measure} on $\U(\R^n)$. 
\end{defn}
\begin{rem}\normalfont
We give two remarks below:
\begin{itemize}
\item[(i)] Note $\rho^0 = \pi$, i.e. $0$-codimensional Poisson measure $\rho^0$ on $\U(\R^n)$ is the Poisson measure $\p$ on $\U(\R^n)$ by noting that the $m$-dimensional spherical Hausdorff measure $\SS^m$ and the $n$-dimensional Lebesgue measure $\mathbf L^n$ coincide when $m=n$ (see Remark \ref{rem: CSH}). 
\item[(ii)] The construction of $\rho^m$, a priori, depends on the choice of the exhaustion $\{B_r\} \subset \R^n$. However, in Proposition \ref{prop: IE}, we will see that it is not the case.
\end{itemize}
\end{rem}
\smallskip
The rest of this section is devoted to the proof of Theorem \ref{thm: MI}. 
Let us begin with a definition.
\begin{defn}[Section of Functions, Multi-Section]\label{def: CF} \normalfont
	Let $M, N \subset \R^n$ be two disjoint sets and $L=M \sqcup N$. For every $F: \U(L) \to \R$ and $\xi \in \U(M)$, define $F_{\xi, M}: \U(N)\to \R$ as 
	\begin{align} 
		F_{\xi, M}(\zeta):=F(\zeta + \xi), \quad \zeta\in \U(N). 
	\end{align}
	For a set $A \subset \U(\R^n)$, let $A_{\xi, \eta, M, N}$ denote the {\it multi-section} both at $\xi \in \U(M)$ and $\zeta \in \U(N)$:
	\begin{equation}
		A_{\xi, \zeta, M, N}:=\{\gamma \in \Upsilon(L^c): \gamma + \xi + \zeta \in A\}, \quad \text{and} \quad A_{\xi, \zeta, M, N}^k=A_{\xi, \zeta,M, N} \cap \U^{k}(L^c).
	\end{equation}
\end{defn}

\begin{lem} \label{lem: DI}
Let $A$ be a Suslin set in $\U(\R^n)$. Let $M, N \subset \R^n$ be two disjoint Borel sets. Set $L=M \sqcup N$.  %Fix $m \le n$, $k \in \N$ and $r>0$. 
Let $F:\U(L) \to \R$ be defined by $\gamma \mapsto F(\gamma):=\S^{m, k}_{L^c}(A^k_{\gamma, L})$. 
Then, 
\begin{align} \label{eq: MSE2}
F_{\xi, M}(\zeta) = \S^{m, k}_{L^c}(A^k_{\zeta, \xi, N , M}), \quad \forall \xi \in \U(M),\ \forall \zeta \in \U(N).
\end{align}
\end{lem}
\begin{proof}
The set $A^k_{\zeta, \xi, N , M}$ is Suslin by the same argument as in Lemma \ref{lem: sec}. Thus, $\S^{m, k}_{L^c}(A^k_{\zeta, \xi, N , M})$ is well-defined. 
By Definition \ref{def: CF}, we have that 
$$
F_{\xi, M}(\zeta) = F(\zeta+\xi)= \S^{m, k}_{L^c}(A^k_{\zeta+\xi, L})= \S^{m, k}_{L^c}\bigl(\{\gamma \in \Upsilon(L^c): \gamma + \xi + \zeta \in A\}\bigr) =\S^{m, k}_{L^c}(A^k_{\zeta, \xi, N , M}) \, . \qedhere
$$
\end{proof}
The next lemma is straightforward since the Poisson measures $\p_{M}$ and $\p_{N}$ are mutually singular.  
\begin{lem} \label{lem: DI2}
With the same notation $M, N$ and $L$ as in Lemma \ref{lem: DI}. 
For any bounded measurable function $G$ on $\U(L)$, 
\begin{align} \label{eq: DI2}
\int_{\U(L)} G(\eta) d\pi_{L}(\eta) = \int_{\U(N)} \int_{\U(M)}G_{\xi, M}(\zeta) d\p_{M}(\xi) d\p_{N}(\zeta).
\end{align}
\end{lem}

\begin{proof}[Proof of Theorem \ref{thm: MI}]
Let $\mathcal A_{r, \e}:=B_{r+\e}\setminus B_r$ be the annulus of width $\e$ and radius $r$. 
Fix $A \in\mathcal S(\E)$, $r>0$, $\e>0$ and $\zeta \in \U(B^c_{r+\e})$.   
We claim that
    \begin{equation}\label{eq:mainclaim}
        \S^{m, k}_{B_{r + \e}}(A^k_{\zeta, B_{r+\e}^c}) 
        \ge
        \sum_{j=0}^k \int_{\U(\mathcal A_{r, \e})}\S^{m, k}_{B_{r}}(A^{k-j}_{\zeta, \xi, B^c_{r+\e}, \mathcal A_{r, \e}}) d\S^{j}_{\mathcal A_{r, \e}}(\xi).
    \end{equation}
 Let us first show how \eqref{eq:mainclaim} concludes the proof. 
 For simplicity of notation, we set $M= \mathcal A_{r, \e}$, $N=B^c_{r+\e}$ and $L = M \sqcup N$. Then, \eqref{eq:mainclaim} is reformulated as follows: 
      \begin{align*}
        \S^{m, k}_{N^c}(A^k_{\zeta, N}) \ge \sum_{j=0}^k \int_{\U(M)}\S^{m, k}_{L^c}(A^{k-j}_{\zeta, \xi, N, M}) d\S^{j}_{M}(\xi).
    \end{align*}
 Then, by using Lemma \ref{lem: DI2} and Lemma \ref{lem: DI} we deduce
 \begin{align*}
 \rho^m_{r}(A)  & = {e^{-\SS^n(L^c)}}  \sum_{k=0}^\infty \int_{\U(L)}\S^{m, k}_{L^c}(A^k_{\eta, L}) d\p_{L}(\eta)
  \\& = {e^{-\SS^n(L^c)}}\sum_{k=0}^\infty \int_{\U(N)} \int_{\U(M)} \bigl( \S^{m, k}_{L^c}(A^k_{\zeta, L})\bigr)_{\xi,M} d\p_{M}(\xi)  d\p_{N} (\zeta)
   \\
  & = {e^{-\SS^n(L^c)}}\sum_{k=0}^\infty \int_{\U(N)} \int_{\U(M)} \S^{m, k}_{L^c}(A^k_{\zeta, \xi, N, M}) d\p_{M}(\xi) d\p_{N}(\zeta)
%     \\
%  & = {\color{red}e^{-\SS^n(L^c)}}e^{-\SS^n(M)}\sum_{k=0}^\infty \sum_{j=0}^\infty \int_{\U(N)} \int_{\U(M)} \S^{m, k}_{L^c}(A^k_{\zeta, \xi, N, M}) d\S^{j}_{M}(\xi)  d\p_{N}(\zeta)
%     \\
%  & \le \sum_{k=0}^\infty \sum_{j=0}^\infty \int_{\U(N)} \int_{\U(M)} \S^{m, k}_{L^c}(A^k_{\zeta, \xi, N, M}) d\S^{j}_{M}(\xi)  d\p_{N}(\zeta)
 \\
 &= {e^{-\SS^n(L^c)}}e^{-\SS^n(M)}\sum_{k=0}^\infty \sum_{j =0}^k \int_{\U(N)} \int_{\U(M)} \S^{m, k-j}_{L^c}(A^{k-j}_{\zeta, \xi, N, M}) d\S^{j}_{M}(\xi) d\p_{N}(\zeta)
  \\
 &\le {e^{-\SS^n(N^c)}}\sum_{k=0}^\infty \int_{\U(N)}\S^{m, k}_{N^c}(A^k_{\zeta, N}) d\p_{N}(\zeta)
 \\
 &=\rho^m_{r+\e}(A) \, .
	\end{align*}
 To show \eqref{eq:mainclaim},  it is enough to verify that, for any bounded measurable function $F$ on $\U(\R^n)$, 
		\begin{equation} \label{eq: M1}
		\int_{\U(N^c)} F_{\zeta, N}(\gamma) d \S^{m, k}_{N^c}(\gamma) \ge \sum_{j=0}^k \int_{\U(M)} \int_{\U(L^c)}  (F_{\zeta, N})_{\xi, M}(\gamma)d \S^{m,k-j}_{L^c}(\gamma) d \S^{j}_{M}(\xi).
	\end{equation}
 By the definition of $\S^{m, k}_{N^c}$, the L.H.S.\ of \eqref{eq: M1} can be deduced as follows:
\begin{align*}
\int_{\U(N^c)} F_{\zeta, N}(\gamma) d \S^{m, k}_{N^c}(\gamma) = \frac{1}{k!}\int_{(N^c)^{\otimes k}} (F_{\zeta, N}\circ \mathbf s_k)(\mathbf x_k) d \SS^{nk-m}_{N^c}(\mathbf x_k),
\end{align*}
whereby $\mathbf x_k:=(x_*,\ldots, x_{k-1})$ and $\mathbf x_* = x_*$. 
Furthermore, by the definition of $(F_{\zeta, N})_{\xi, M}$, the R.H.S.\ of \eqref{eq: M1} can be deduced as follows:
\begin{align*}
&\int_{\U(M)} \int_{\U(L^c)}  (F_{\zeta, N})_{\xi, M}(\gamma)d \S^{m,k-j}_{L^c}(\gamma) d \S^{j}_{M}(\xi)
\\
&=\int_{\U(M)} \int_{\U(L^c)}  (F_{\zeta, N})(\gamma + \xi)d \S^{m,k-j}_{L^c}(\gamma) d \S^{j}_{M}(\xi)
\\
&= \frac{1}{j!(k-j)!}\int_{M^{\times j}} \int_{(L^c)^{\times (k-j)}}  (F_{\zeta, N} \circ \mathbf s_{k}) (\mathbf x_{k-j}, \mathbf y_j) d \SS^{n(k-j)-m}_{L^c}(\mathbf x_{k-j})d \SS^{nj}_{M}( \mathbf y_j),
\end{align*}
whereby $(\mathbf x_{k-j}, \mathbf y_j) = (x_*, \ldots, x_{k-j-1}, y_*, \ldots, y_{j-1})$. 
	%\begin{equation}
	%	\int (g_{r+\e, \eta})_k d H^m_{k,r+\e} \ge \sum_{j=0}^k \int \int (g_{r,\xi + \eta})_{k-j} d H^m_{k-j,r} d H_{j,r,r+\e}(\xi)
	%\end{equation}
	%for any $g$ with bounded support.
	%From Remark \ref{rem:projectionfunctions}, we deduce
	%\begin{equation}
	%	\int (g_{r+\e, \eta})_k d H^m_{k,r+\e} 
	%	=
	%	\frac{1}{k!}\int_{B_{r+\e}^{\otimes k}} (g_{k+\eta(B_R)})_{\eta,r+\e}(x_1,\ldots,x_k) d {\mathcal{H}}^{nk-m}(x)
	%\end{equation}
	%and
	%{\color{blue}I'm using that $k-j + (\xi+\eta)(B_R) = k-j + \xi(B_R) + \eta(B_R) = k + \eta(B_R)$}
	%\begin{align*}
	 %&\int \left(\int  (g_{r,\xi + \eta})_{k-j} d H^m_{k-j,r}\right) d H_{j,r,r+\e}(\xi)
	%\\&	=
	%\frac{1}{j!(k-j)!}\int_{B_r^{\otimes (k-j)}} \int_{\mathcal{A}_{r,r+\e}^{\otimes j}} (g_{k+\eta(B_R(p))})_{\eta,r+\e}(x_1,\ldots,x_{k-j}, y_1,\ldots, y_j) d \mathcal{H}^{n(k-j)-m}(x) d \mathcal{H}^{nj}(y)
	%\end{align*}
	%where we identified $(g_{\ell+\eta(B_r)})_{\eta,r}$ with a symmetric function of $\ell$-variables. 
	Hence, in order to conclude \eqref{eq: M1}, it suffices to show the following inequality: for any bounded measurable symmetric function $f$ on $(\R^n)^{\times k}$,  
	\begin{align*}
\int_{B_{r+\e}^{\times k}} f( \mathbf x_k) d {\SS}_{B_{r+\e}}^{nk-m}( \mathbf x_k)
	 \ge \sum_{j=0}^k\frac{k!}{j!(k-j)!} \int_{B_r^{\times (k-j)}} \int_{\mathcal{A}_{r,\e}^{\times j}} f (\mathbf x_{k-j}, \mathbf y_j) d \SS^{n(k-j)-m}_{\mathcal{A}_{r,\e}}(\mathbf x_{k-j}) d \SS^{nj}_{B_r}(\mathbf y_j).
	\end{align*}
	%$f\in C_c((\R^n)^{\otimes k})$ which is symmetric, i.e. $f(x_1,\ldots,x_k) = f(x_{\sigma(1)},\ldots, x_{\sigma(k)})$ for any $\sigma\in \mathfrak S_k$ where $x_i\in \R^n$.
 By using the symmetry of $f$ and a simple combinatorial argument, we obtain 
	\begin{align*}
\int_{B_{r+\e}^{\times k}} f(\mathbf x_k) d {\SS}_{B_{r+\e}}^{nk-m}(\mathbf x_k) & = \sum_{j=0}^k \frac{k!}{j!(k-j)!} \int_{B_r^{\times (k-j)}} \int_{\mathcal{A}_{r,\e}^{\times j}} f(\mathbf x_{k-j}, \mathbf y_j) d \SS^{nk-m}_{B_{r+\e}} (\mathbf x_{k-j}, \mathbf y_j),
	\end{align*}
	while \cite[2.10.27, p.\ 190]{Fed69} implies 
	\begin{align*}
		&\int_{B_r^{\times (k-j)}} \int_{\mathcal{A}_{r,\e}^{\times j}} f (\mathbf x_{k-j}, \mathbf y_j) d \SS^{nk-m}_{B_{r+\e}}(\mathbf x_{k-j}, \mathbf y_j)
		\\ & \ge \int_{B_r^{\times (k-j)}} \int_{\mathcal{A}_{r,r+\e}^{\times j}}  f (\mathbf x_{k-j}, \mathbf y_j) d \SS^{n(k-j)-m}_{\mathcal{A}_{r,\e}}(\mathbf x_{k-j}) d \SS^{nj}_{B_r}(\mathbf y_j). \qedhere
	\end{align*}
\end{proof}

\subsection{Independence of $\rho^m$ from the exhaustion}
So far we have built the $m$-codimensional measure $\rho^m$ by passing to the limit a sequence of finite dimensional measures $\rho^m_r$. The latter have been constructed by relying on the exhaustion $\{ B_r\}_{r>0}$ of $\R^n$. Hence, a priori, $\rho^m$ depends on the chosen exhaustion. In this subsection we make a remark that this is actually not the case.

\medskip

Let $\Omega \subset \R^n$ be a compact set.
Following closely the proof in section \ref{subsec: finite-dim measure} we can prove that
\begin{align} 
	\rho^m_\Omega(A):= e^{-\SS^n(\Omega)}\sum_{k=1}^\infty\int_{\U(\Omega^c)}\S^{m, k}_{\Omega}(A^k_{\eta,\Omega^c})\ d\p_{\Omega^c}(\eta).
\end{align}
is well defined for any Suslin set $A$.

The next proposition can be proven by arguing as in Theorem \ref{thm: MI}. We omit the proof.
\begin{prop}[Independence from Exhaustion] \label{prop: IE}
	Let $0 < r < R < \infty$ and $\Omega \subset \R^n$ be a compact subset satisfying $B_r\subset \Omega \subset B_R$. Then
	\begin{equation}
		\rho^m_r(A) \le \rho^m_{\Omega}(A) \le \rho^m_R(A) \, ,
		\quad
		\text{for every Suslin set $A$} \, .
	\end{equation}
    In particular $\rho^m$ does not depend on the choice of the exhaustion.
\end{prop}

\section{Bessel capacity and finite-codimensional Poisson measure} \label{sec: cap}

In this section, we discuss a relation between {\it Bessel capacities} and finite-codimensional Poisson measures $\rho^m$. This will play a significant role to develop fundamental relations between potential analysis induced by $(\EE, \mathcal D(\EE))$ and theory of BV functions in Section~\ref{sec:BV} and Section~\ref{sec: CF}.
% and also to develop stochastic analysis in Section~\ref{sec: AGIF}.  

\begin{defn}[Bessel operator]
	Let $\alpha>0$ and $1\le p < \infty$. We set
	\begin{align} \label{defn: Bes}
		B_{\alpha,p}
		:= \frac{1}{\Gamma(\alpha/2)}\int_*^\infty e^{-t}t^{\alpha/2-1}T^{(p)}_t dt \, ,
	\end{align}
    where $T_t^{(p)}$ is the $L^p$-heat semigroup, see Section \ref{subsection:Sobolev}.
\end{defn}
Notice that $B_{\alpha,p}$ is well defined for $F\in L^p(\U(\R^n),\pi)$ and satisfies
\begin{equation}
	\| B_{\alpha, p} F \|_{L^p} \le \| F \|_{L^p} \, ,
\end{equation}
due to the contractivity of $T^{(p)}_t$ in $L^p(\U(\R^n), \pi)$.
%Denoting by $L_p$ the infinitesimal generator of $T^{(p)}_t$ in $L^p(\U(\R^n),\p)$ 
%One can easily show that $B_{\alpha, p}=(I-\Delta^{(p)})^{-\alpha/2}$ by the standard operator algebra.

\begin{defn}[Bessel capacity] \normalfont \label{defn: BC}
	Let $\alpha>0$ and $1\le p < \infty$. The {\it $(\alpha, p)$-Bessel capacity} is defined as
	\begin{align} \label{equ: BC}
		{\rm Cap}_{\alpha, p}(E):=\inf\{\|F\|^p_{L^p}: B_{\alpha, p} F \ge 1 \ \text{on $E$},\ F \ge 0\} \, ,
	\end{align}
    for any $E\subset \U(\R^n)$.
\end{defn}
%\subsection{Bessel Capacity}
%Let $L_2$ be the infinitesimal generator associated with $\{T_t\}$ on $L^2(\U, \p)$. Since $\{T_t\}$ can be extended uniquely to the $L^p$-contractive strongly continuous operator $\{T^{(p)}_t\}$  for $1 \le p <\infty$,\footnote{EB: this has been proven in one of the last sections.} we may introduce $L_p$ as the infinitesimal generator associated with $\{T^{(p)}_t\}$ in $L^p(\U, \p)$ for $1 \le p <\infty$. 
%
%Let $B_{\alpha, p}:=(I-L_p)^{-\alpha/2}$ be the {\it $(\alpha, p)$-Bessel operator} for $\alpha>0$ and $1 \le p <\infty$. Note that, since the operator $L_p$ is unbounded in $L^p(\U, \p)$, the  definition of $B_{\alpha ,p}$ by the formal operational calculus needs to be justified,  which, in turn,  can be  well-defined by the following formula:
%\begin{align} \label{defn: Bes}
%	(I-L_p)^{-\alpha/2} := \frac{1}{\Gamma(\alpha/2)}\int_*^\infty e^{-t}t^{\alpha/2-1}T^{(p)}_t dt.
%\end{align}
%The $L^p$-contraction of $\{T_t^{(p)}\}$ immediately implies the  $L^p$-contraction of the Bessel operatorl $B_{\alpha, p}$. In particular, $B_{\alpha, p}$ is a bounded  operator in $L^p(\U, \p)$. 
%
%The space $L^{\alpha, p}(\U, \p):=\{G=B_{\alpha, p}F: F \in L^p(\U, \p)\}$ equipped with the norm $\|G\|_{\alpha, p}:=\|F\|_p$ is called the {\it space of $(\alpha, p)$-Bessel functions}. We introduce the {\it $(\alpha, p)$-Bessel capacity} as follows:
%\begin{align} \label{defn: BC}
%	{\rm Cap}_{\alpha, p}(E):=\inf\{\|F\|^p_{p}: B_{\alpha, p} F \ge 1 \ \text{on $E$},\ F \ge 0\}.
%\end{align}
We are now ready to state the main theorem of this section. 
\begin{thm} \label{thm: CH}
	Let $\alpha p >m$. Then, 
	${\rm Cap}_{\alpha, p}(E) = 0$ implies $\rho^m(E)=0$ for any $E \in \mathcal S(\E)$. 
\end{thm}
%Before providing the rigorous proof, we explain the heuristic idea of the proof.
%{\it The heuristic proof:}
%Since $\rho^m(E) = \lim_{r \to \infty} \rho_r^m(E)$, it suffices to show that $\rho^m_r(E)=0$ for any $r>0$. 
We briefly explain the heuristic idea of proof. In view of the identities 
$$
\rho^m(E) = \lim_{r \to \infty} \rho_r^m(E) \, ,
$$
%it suffices to show that $\rho^m_r(E)=0$ for any $r>0$.
%Thanks to the identity
\begin{equation*}
	\rho^m_r(E) = e^{-\SS^{n}(B_r)}\sum_{k=1}^\infty\int_{\U(B_r^c)}\S_{B_r}^{m,k}(E_{\eta,r}^k) d\pi_{B_r^c}(\eta)\, ,
\end{equation*}
it is enough to prove that $\S_{B_r}^{m,k}(E_{\eta,r}^k)=0$ for $\p_{B_r^c}$-a.e.\ $\eta$, all $k \in \N$ and $r>0$.
%which has been now reduced to the problem of the finite-dimensional space.  
%Assume that we can prove that
This, together with the implication 
\begin{align} \label{eq: HI1}
{\rm Cap}_{\alpha, p}(E)=0 \implies {\rm Cap}^{\eta, r}_{\alpha, p}(E_{\eta,r}^k)=0, \quad \text{for $\p_{B_r^c}$-a.e.\ $\eta$ and all $k \in \N$ and $r>0$} \, ,
\end{align}
where ${\rm Cap}^{\eta, r}_{\alpha, p}$ is the Bessel $(\alpha, p)$-capacity on $\U^k(B_r)$, reduces the problem to the corresponding problem in the finite dimensional setting. To be more precise,  we will show that
% Then, by using the result from the finite-dimensional theory, we can prove 
$$
{\rm Cap}^{\eta, r}_{\alpha, p}(E_{\eta,r}^k)=0 \implies \S_{B_r}^{m,k}(E_{\eta,r}^k)=0 \, .
$$
%which concludes $\S_{B_r}^{m,k}(E_{\eta,r}^k)=0$. 
%which is a standard result of finite dimensional spaces.

In the rest of this section, we implement the aforementioned idea. The key point is to show \eqref{eq: HI1}, for which we introduce localisations of functional-analytic objects in Section~\ref{sec: LSF} and Section~\ref{sec: LE}. We then introduce {\it localised Bessel operators} and {\it localised Bessel capacities} in Section~\ref{sec: LBO}.

\subsection{Localisation of sets and functions} \label{sec: LSF}

\begin{lem} \label{lem: M2}
Let $A \subset \U(\R^n)$ be a $\p$-measurable set. Let $B \subset \R^n$ be a Borel set. Then, $A_{\eta, B}$ is $\p_{B^c}$-measurable for $\p_{B}$-a.e.\ $\eta \in \U(B)$. 
Moreover, if $\pi(A)=0$, then $\pi_{B^c}(A_{\eta,B})=0$ for a.e. $\eta\in \U(B)$.
\end{lem}

\begin{proof}
By hypothesis, there exist Borel sets $\underline{A} \subset A \subset \overline{A}$ so that $\p(\overline{A} \setminus \underline{A})=0$. By (i) in Remark~\ref{rem: SS}, $\underline{A}$ and $\overline{A}$ are Suslin. 
% Let $\mathsf A := \overline{A} \setminus \underline{A}$. 
By Lemma~\ref{lem: sec}, $\underline{A}_{\eta, B}$ and $\overline{A}_{\eta, B}$ are Suslin.  %and $\mathsf A_{\eta, B}$
By the standard disintegration argument as in Lemma~\ref{lem: DI2}, it holds that
\begin{align} \label{eq: II-DI}
0= \p(\overline{A} \setminus \underline{A}) = \int_{\U(B)} \p_{B^c}((\overline{A} \setminus \underline{A})_{\eta, B}) d\p_{B}(\eta).
\end{align}
Therefore, there exists a $\p_{B}$-measurable set $\Omega \subset \U(B)$ so that $\p_{B^c}((\overline{A} \setminus \underline{A})_{\eta, B}) =0$ for any $\eta \in \Omega$. By noting that $\underline{A}_{\eta, B} \subset A_{\eta, B} \subset \overline{A}_{\eta, B}$, we conclude that $A_{\eta,B}$ is $\pi_B$-measurable since, up to $\pi_B$ negligible sets, it coincides with a Suslin set and every Suslin set is $\p_B$-measurable by (ii) in Remark~\ref{rem: SS}. The proof of the first assertion is complete. 

If $\pi(A) =0$ the disintegration
\begin{equation}
	0= \p(A) = \int_{\U(B)} \p_{B^c}(A_{\eta, B}) d\p_{B}(\eta) \, ,
\end{equation} 
immediately gives the second assertion.
\end{proof}

\begin{cor}\label{lem: II}
	Let $A \subset \U(\R^n)$ be a $\p$-measurable set, $B \subset \R^n$ a Borel set, and let $g$ be a $\p$-measurable function on $\U(\R^n)$ with $g \ge 1$ $\p$-a.e.\ on $A$. 
	Then, for $\pi_B$-a.e. $\eta$ it holds
	\begin{align} \label{eq: II}
		g_{\eta, B} \ge 1, \quad \text{$\p_{B^c}$-a.e.\ on $A_{\eta, B}$} \, .
	\end{align}   
%	Here, $g_{\eta, B}$ has been defined in \eqref{def: CF}.
\end{cor}
\proof
Taking $\tilde{A} = A\setminus\{ g \ge 1\}$ and applying Lemma~\ref{lem: M2} with $\tilde{A}$ in place of $A$,  we obtain the conclusion.
\qed
%\begin{lem} \label{lem: II}
%Let $A \subset \U(\R^n)$ be a $\p$-measurable set and $B \subset \R^n$ be a Borel set.  Let $g$ be a $\p$-measurable function on $\U(\R^n)$ with $g \ge a$ for $a \in \R$ $\p$-a.e.\ on $A$. 
%Then, there exists a $\p$-measurable set $\Omega \subset \U(B)$ so that $\p_{B}(\Omega)=1$, and 
%\begin{align} \label{eq: II}
%g_{\eta, B} \ge a, \quad \text{$\p_{B^c}$-a.e.\ on $A_{\eta, B}$}, \quad \forall \eta \in \Omega.
%\end{align}
%Here, $g_{\eta, B}$ has been defined in \eqref{def: CF}. 
%\end{lem}
%\proof
%By hypothesis and (ii) in Remark \ref{rem: SS}, there exist Suslin sets $\underline{A} \subset A \subset \overline{A}$ so that $\p(\overline{A} \setminus \underline{A})=0$. Let $\mathsf A := \overline{A} \setminus \underline{A}$. By Lemma \ref{lem: sec}, $\underline{A}_{\eta, B}$, $\overline{A}_{\eta, B}$ and $\mathsf A_{\eta, B}$ are Suslin. Without loss of generality, we may assume that $g \ge 1$ everywhere on $\underline{A}$. 
%By the standard disintegration argument (\purple{reference}), it holds that
%\begin{align} \label{eq: II-DI}
%0= \p(\mathsf A) = \int_{\U(B)} \p_{B^c}(\mathsf A_{\eta, B}) d\p_{B}(\eta).
%\end{align}
%Therefore, there exists a $\p_{B}$-measurable set $\Omega \subset \U(B)$ so that $\p_{B^c}(\mathsf A_{\eta, B}) =0$ for any $\eta \in \Omega$. By noting that $\mathsf A_{\eta, B}=\overline{A}_{\eta, B} \setminus \underline{A}_{\eta, B}$, $\underline{A}_{\eta, B} \subset A_{\eta, B} \subset \overline{A}_{\eta, B}$ and $g_{\eta, B} \ge a$ everywhere on $ \underline{A}_{\eta, B}$, we concluded the statement.
%\qed

\begin{lem} \label{lem: DI1} 
Let $1 \le p <\infty$ and $r>0$. Let $F^n, F\in L^p(\U(\R^n),\p)$ such that $F^n\to F$ in $L^p(\U(\R^n), \p)$ as $n\to \infty$. Then, there exists a subsequence (non-relabelled) of $(F^n)$ and a measurable set $A_r \subset \U(\R^n)$ so that $\p_{B_r^c}(A_r)=1$ and 
$$
F^n_{\eta, r} \to F_{\eta, r}, \quad \text{in $L^p(\p_{B_r})$, for any $\eta \in A_r$}. 
$$ 
Note that $F_{\eta, r}:=F_{\eta, B_r^c}$ was defined in Definition~\ref{def: CF}.
\end{lem}

\begin{proof}
By Lemma~\ref{lem: DI2}, we have that 
\begin{align} \label{eq: L1}
	\int_{\U(B_r^c)} \biggl( \int_{\U(B_r)} |F^n_{\eta,r}-F_{\eta,r}|^p d\p_{B_r} \biggr)d\p_{B_r^c}(\eta)
	=
	\int_{\U(\R^n)} |F^n-F|^p d\p \to 0, \quad \text{as $n \to \infty$} \, .
\end{align}
In particular, up to subsequence $\int_{\U(B_r)} |F^n_{\eta, r}-F_{\eta,r}|^p d\p_{B_r} \to 0$ for $\p_{B_r^c}$-a.e.\ $\eta$, which completes the proof.		
\end{proof}

\subsection{Localisation of energies, resolvents and semigroups}\label{sec: LE}
In this section, we localise differential operators and related objects introduced in Section~\ref{subsection:Sobolev}.
%Recall that, for any $F \in \Cyl$, we can define the functional
%
%let $\Gamma^\U(F)$ be the square-field operator defined by
%%\footnote{ The lifting $l$ is omitted, but please take $l$ when necessary)}
%
%{\color{red}
%	For Cylinder functions we can write the derivative in a more explicit form. Is it really necessary to use the square-field operator notation?
%}
%{\color{blue}Introduce the explicit one and show that it coincides with the intrinsic one}
%
%$$\Gamma^{\U}(F)(\gamma):= \sum_{x \in {\rm supp} \gamma}\frac{1}{\gamma(\{x\})} \Gamma^{\R^n}\Bigl(F\bigl(\mathbf 1_{\R^n \setminus \{x\}}\cdot \gamma+\gamma(\{x\}) \delta_\bullet\bigr) - F\bigl(\mathbf 1_{\R^n \setminus \{x\}}\cdot \gamma  \bigr) \Bigr)(x),$$
%$$
%\EE(F)=\int_{\U(\R^n)} |\nabla F|_{T\U}^2 d\p \, ,
%$$
%which turns out to be closable as shown in section \ref{subsection:Sobolev}. The closure is denoted by $W^{1,2}(\U(\R^n),\p)$.
%
%(\cite[Thm.\ 3]{Yos96}, see also \cite[Cor.\ 4.1]{AKR98}). The closure of $\Cyl$ is denoted by $\F$.
%We denote by $\{T_t\}_{t>0}$ the semigroup associated with $(W^{1,2}(\U(\R^n),\p), \EE)$ and by by $\{G_\alpha\}_{\alpha>0}$ the resolvent operator, i.e.
%\begin{equation}
%	G_\alpha = \int_*^\infty e^{-\alpha t} T_t \, .
%\end{equation}

Let $r>0$. The {\it localised  energy $(\EE_r, \mathcal D(\EE_r))$} is defined as the following direct integral 
\begin{align} \label{prop: 0}
\EE_r(F) = \int_{\U(B_r^c)}\EE_{\U(B_r)}(F_{\eta, r}) d \pi_{B_r^c}(\eta) \,, \quad \mathcal D(\EE_r):=\{F \in L^2(\U(\R^n), \pi): \EE_r(F)<\infty\} \, .
\end{align}
The form is closed by \cite[Proposition V.3.1.1]{BH91}. For $F \in \CylF(\U(\R^n))$, 
\begin{equation}
	\EE_r(F) = \int_{\U(\R^n)} | \nabla_r F|_{T\U}^2 d \p \, \quad F \in \CylF(\U(\R^n)),
\end{equation}
where
\begin{equation}\label{eq: nabla r}
	\nabla_r F(\gamma,x) := \chi_{B_r}(x) \nabla F(\gamma,x) \, .
\end{equation}
See e.g., \cite[Prop.~3.9]{Suz24}.
%We stress that 
%\begin{align} \label{def: 1}
%\Gamma^{\U}_r(F)(\gamma):= \sum_{x \in {\rm supp} \gamma|_{B_r}}\frac{1}{\gamma(\{x\})} \Gamma^{\R^n}\Bigl(F\bigl(\mathbf 1_{\R^n \setminus \{x\}}\cdot \gamma+\gamma(\{x\}) \delta_\bullet\bigr) - F\bigl(\mathbf 1_{\R^n \setminus \{x\}}\cdot \gamma  \bigr) \Bigr)(x).
%\end{align}
%When $\Gamma^\U$ acts on a function $F$ defined on $\U(B_r)$, then we denote it by $\Gamma^{\U(B_r)}(F)$. 
%The closability follows from~\cite[Lemma 2.2]{Osa96}, or ~\cite[the first paragraph of the proof of Thm.\ 3.48]{DS21}. 
%The It turns out that $\EE_r$ is closable on $\Cyl$ in $L^2(\U, \p)$ (), and  the closure is denoted by $\F_r$.
 We denote by $\{G_\alpha^{r}\}_{\alpha>0}$ and $\{T_t^{r}\}_{t>0}$ the $L^2$-resolvent operator and the semigroup associated with $(\EE_{r}, \mathcal D(\EE_r))$, respectively. Recall that $\{G_{\alpha}^{\U(B_r)}\}_{\alpha}$ and $\{T^{\U(B_r)}_t\}$ denote the $L^2$-resolvent operator and the semigroup corresponding to $(\EE_{\U(B_r)}, H^{1,2}(\U(B_r), \p))$. The relation between $\{G_\alpha^{r}\}_{\alpha>0}$, $\{T_t^{r}\}_{t>0}$ and $\{G_{\alpha}^{\U(B_r)}\}_{\alpha}$, $\{T^{\U(B_r)}_t\}$ is given below. 
%For $F \in \Cyl$, define
%$$\EE_r(F):=\int_{\U(X)} \Gamma^{\U}_r(F) d\p, \quad r >0.$$
%The functional $(\EE_r, \Cyl)$ is closable. The closure of $\Cyl$ is denoted by $\D^{1,2}_r$.
%The next result is taken from~\cite[Lemma 2.2]{Osa96} or ~\cite[the second paragraph of the proof of Thm.~3.48]{DS21}.

\begin{prop}[{\cite[Corollary 4.11]{Suz23}}]\label{prop: 1}
	Let $\alpha>0$, $t>0$, and $r>0$ be fixed. Then, for any bounded measurable function $F$,  {it holds that}
	%\footnote{The domain of the superposition form must be identified with the closure of cylinder functions in the same approach as in Dyson paper} 
	\begin{align} \label{eq: R-1}
		G_\alpha^{r} F(\gamma) & = G_\alpha^{\U(B_r)} F_{\gamma|_{B_r^c}, r}(\gamma|_{B_r}) \, ,
		\\
		\label{eq: R-2}
		T_t^{r} F(\gamma) & = T_t^{\U(B_r)} F_{\gamma|_{B_r^c}, r}(\gamma|_{B_r}) \, ,
	\end{align}
	for $\pi$-a.e.\ $\gamma\in \U(\R^n)$.
\end{prop}

\begin{rem}\label{rem: p-sem} \normalfont
Although Proposition \ref{prop: 1} provides the statement only for the $L^2$-semigroups and resolvents, it is straightforward to extend it to the $L^p$-semigroups and resolvents for any $1 \le p<\infty$.
\end{rem}

\begin{prop}\label{prop: MD}
The form $(\EE_r, \mathcal D(\EE_r))$ is monotone non-decreasing in $r$, i.e. for any $s \le r$, 
$$\mathcal D(\EE_r) \subset \mathcal D(\EE_s), \quad \EE_s(F) \le \EE_r(F), \quad F \in \mathcal D(\EE_r). $$
Furthermore, the following two forms coincide: letting $\bar{\EE}(F):=\lim_{r \to \infty}\EE_r(F)$ and $\mathcal D(\bar{\EE})=\{F \in \cap_{r>0} \mathcal D(\EE_r): \lim_{r \to \infty}\EE_r(F)<\infty\}$,
$$(\bar\EE,  \mathcal D(\bar\EE)) = (\EE, H^{1,2}(\U(\R^n),\p)).$$
%$(\lim_{r \to \infty}\EE_r(F),  = \EE(F)$ for $F \in H^{1,2}(\U(\R^n),\p)$.
\end{prop}
\begin{proof}
The monotone increasing property is a direct application of \cite[Proposition 4.13]{Suz23}. 
The second assertion follows from the fact that $(\Delta, \CylF(\U(\R^n)))$ is essentially self-adjoint by~\cite[Theorem 5.3]{AKR98} and that $\bar\EE$ and $\EE$ coincide on~$\CylF(\U(\R^n))$.
\end{proof}
\begin{rem}\normalfont
 In \cite[Corollary 4.11, Proposition 4.13]{Suz23}, the statements deal with the case where the reference measure is the law of the $\mathsf{sine}_\beta$ point process. The case of the Poisson point process corresponds to $\beta=0$, and the same proofs there apply to the case of the Poisson point process in this paper.
\end{rem}

The next proposition shows the monotonicity property for
the resolvent operator $G_\alpha^r$ and the semigroup $T^r_t$.

\begin{prop} \label{prop: MGS}
The resolvent operator $\{G_\alpha^{r}\}_{\alpha}$ and the semigroup $\{T_t^{r}\}_t$ are monotone non-increasing on non-negative functions, i.e.,
\begin{align} \label{ineq: RM}
G_\alpha^{r} F \le G_\alpha^{s} F,  \quad T_t^{r} F \le T_t^{s} F,   \quad \text{for every non-negative}\, \, F \in L^2(\U(\R^n), \p), \quad s \le r.
\end{align}
Furthermore, $\lim_{r \to \infty} G_\alpha^{r} F = G_\alpha F$ and  $\lim_{r \to \infty} T_t^{r} F = T_t F$ for $F \in  L^2(\U(\R^n), \p)$ and $\alpha, t>0$.
\end{prop}

\begin{proof}
Thanks to the identity 
\begin{equation*}
	G_\alpha^r = \int_*^\infty e^{-\alpha t} T^r_t dt\, ,
\end{equation*}
it suffices to show \eqref{ineq: RM} only for $T^r_t$. By a direct application of~\cite[Theorem 3.3]{Ouh96} and the monotonicity of the Dirichlet form in Proposition \ref{prop: MD}, we obtain the monotonicity of the semigroup. 
The second part of the statement follows from the monotone convergence $\EE_r\uparrow \EE$ combined with \cite[S.14, p.372]{RS77}.
\end{proof}

\subsection{Localised Bessel operators} \label{sec: LBO}
Let $B^r_{\alpha, p}$ and $B^{\U(B_r)}_{\alpha, p}$ be the $(\alpha, p)$-Bessel operators corresponding to $\{T_t^r\}_{t>0}$ and $\{T_t^{\U(B_r)}\}_{t>0}$, respectively defined in the analogous way as in \eqref{defn: Bes}. The corresponding $(\alpha, p)$-Bessel capacities are denoted by ${\rm Cap}^r_{\alpha, p}$ and ${\rm Cap}^{\U(B_r)}_{\alpha, p}$ defined in the analogous way as in \eqref{equ: BC}

\begin{lem} \label{lem: MC}
${\rm Cap}^r_{\alpha, p}(E) \le {\rm Cap}_{\alpha, p}(E)$ for every $E\subset \U(\R^n)$ and $r>0$. 
\end{lem}
\proof
It suffices to show that $B^r_{\alpha, p} F \le B_{\alpha, p} F$ for any $F \ge 0$ with $F \in L^p(\U(\R^n), \p)$, which immediately follows from Proposition~\ref{prop: MGS} and \eqref{defn: Bes}. 
\qed
\begin{lem} \label{lem: MC2}
If ${\rm Cap}_{\alpha, p}(E)=0$, then ${\rm Cap}^{\U(B_r)}_{\alpha, p}(E_{\eta, r})=0$ for $\p_{B_r^c}$-a.e.\ $\eta$ and every $r>0$. 
\end{lem}
\proof
By Lemma~\ref{lem: MC} we may assume ${\rm Cap}^r_{\alpha, p}(E)=0$ for any $r>0$. Let $\{F_n\} \subset L^p(\U(\R^n), \p)$ be a sequence so that $F_n \ge 0$, $B^r_{\alpha, p}F_n \ge 1$ on $E$, and $\|F_n\|_{L^p}^p \to 0$. By Lemma~\ref{lem: II}, $(F_n)_{\eta, r} \ge 0$ for $\p_{B_r^c}$-a.e.\ $\eta$. 
Furthermore, by Lemma~\ref{lem: DI1}, there exists $A_r \subset \U(B_r^c)$ and a (non-relabelled) subsequence $(F_n)_{\eta, r}$ so that $\p_{B_r^c}(A_r)=1$, and for every $\eta \in A_r$, 
\begin{align} \label{eq: SCZ}
(F_n)_{\eta, r} \to 0, \quad \text{in $L^p(\U(B_r), \p_{B_r})$} \, . 
\end{align}
By Proposition~\ref{prop: 1} and Remark~\ref{rem: p-sem}, we have that
\begin{align} \label{eq: idL}
(B^r_{\alpha, p}F_n)_{\eta, r}  \notag
&= \biggl( \frac{1}{\Gamma(\alpha/2)}\int_*^\infty e^{-t}t^{\alpha/2-1}T^{r}_t F_ndt \biggr)_{\eta, r}
\\
&= \frac{1}{\Gamma(\alpha/2)}\int_*^\infty e^{-t}t^{\alpha/2-1}\bigl(T^{r}_t F_n\bigr)_{\eta, r} dt \notag
%\\
%&= \frac{1}{\Gamma(\alpha/2)}\int_*^\infty e^{-t}t^{\alpha/2-1}\bigl(T^{r}_tF_n\bigr)_{\eta, r} dt \notag
\\
&= \frac{1}{\Gamma(\alpha/2)}\int_*^\infty e^{-t}t^{\alpha/2-1}T^{\U(B_r)}_t (F_{n})_{\eta, r} dt \notag
\\
&=B^{\U(B_r)}_{\alpha, p}(F_n)_{\eta, r}.
\end{align}
Note that we dropped the specification of $p$ in the semigroups for notational simplicity in \eqref{eq: idL}. 

Since $B^r_{\alpha, p}F_n \ge 1$ on $E$,  by applying Corollary~\ref{lem: II}, we obtain that $(B^r_{\alpha, p}F_n)_{\eta, r} \ge 1$ on $E_{\eta, r}$ for $\p_{B_r^c}$-a.e.\ $\eta$. Thus, by \eqref{eq: idL}, $B^{\U(B_r)}_{\alpha, p}(F_n)_{\eta, r} \ge 1$ on $E_{\eta, r}$ for $\p_{B_r^c}$-a.e.\ $\eta$.
By \eqref{eq: SCZ}, we conclude that ${\rm Cap}^{\U(B_r)}_{\alpha, p}(E_{\eta, r})=0$ for $\p_{B_r^c}$-a.e.\ $\eta$ and any $r>0$. 
\qed

\subsection{Finite-dimensional counterpart}
In this section, we develop the finite-dimensional counterpart of Theorem~\ref{thm: CH}. 
The goal is to prove the following proposition.

\begin{prop} \label{lem: MC3}
Let $\alpha p>m$. 
If ${\rm Cap}^{\U^k(B_r)}_{\alpha, p}(E)=0$, then $\S_{B_r}^{m, k}(E)=0$ for any $k \in \N$.
%${\rm Cap}_{\alpha, p}(E)=0$, then ${\rm Cap}^{\eta, r}_{\alpha, p}(E_{\eta, r})=0$ for $\p_{B_r^c}$-a.e.\ $\eta$ and any $r>0$. 
\end{prop}

\begin{proof}
Recall that $T_t^{\Omega, \otimes k}$ is the $k$-tensor semigroup of $T_t^\Omega$ as defined in \eqref{z51}. Let $B_{\alpha, p}^{B_r^{\times k}}$ be the corresponding Bessel operator defined analogously as in \eqref{defn: Bes}, and ${\rm Cap}^{B_r^{\times k}}_{\alpha, p}$ be the corresponding $(\alpha, p)$-capacity.  

Let $\{F_m\} \subset L^p(\U(B_r), \p_{B_r})$ be a sequence so that $F_m \ge 0$ and $B_{\alpha, p}^{\U^k(B_r)} F_m \ge 1$ on $E\subset \U^k(B_r)$, and $\|F_m\|_{L^p} \to 0$. By Proposition~\ref{prop: IDP} and the definition of Bessel operator, we have
$$
B_{\alpha, p}^{\U^k(B_r)} F_m \circ \mathsf s_k = B_{\alpha, p}^{B_r^{\times k}} (F_m \circ \mathsf s_k) \, ,
$$
hence $F_m \circ \mathsf s_k \ge 0$, $B_{\alpha, p}^{B_r^{\times k}} (F_m \circ \mathsf s_k) \ge 1$ on $\mathsf s_k^{-1}(E)$. 
Furthermore, 
$$\|F_m \circ \mathsf s_k\|_{L^p(B_r^{\times k})} = C(k, n, r)\|F_m\|_{L^p(\U^k(B_r))} \to 0 \, ,
\quad
\text{as $m \to \infty$} \, ,
$$ 
where $C(k,n, r)>0$ comes from the constant appearing in front of the Hausdorff measure in the definition of $\p_{B_r}$. This implies that ${\rm Cap}^{B_r^{\times k}}_{\alpha, p}(\mathsf s_k^{-1}(E))=0$. 
We can now rely on standard capacity estimates in the Euclidean setting (see, e.g. \cite[Theorem 2.6.16]{Z89}) 
to conclude that $\SS^{nk-m}(\mathsf s_k^{-1}(E))=0$. 
Recalling \eqref{m-H}, we have that 
$$
\S_{B_r}^{m, k}(E) = \frac{1}{k!}(\mathsf s_k)_{\#}\SS^{nk-m}(E) = \frac{1}{k!}\SS^{nk-m}(\mathsf s_k^{-1}(E))=0 \, . \qedhere
$$
\end{proof}

\subsection{Proof of Theorem \ref{thm: CH}}
Let $E \in \mathcal S(\E)$ such that ${\rm Cap}_{\alpha, p}(E) = 0$. Thanks to Lemma~\ref{lem: MC2} we have ${\rm Cap}^{\U(B_r)}_{\alpha, p}(E) = 0$ for any $r>0$, hence $\S_{B_r}^{m,k}(E_{\eta,r}^k) =0$ for any $k\in \N$ as a consequence of Proposition~\ref{lem: MC3}. It implies
\begin{equation*}
	\rho^m_r(E) = e^{-\SS^{n}(B_r)}\sum_{k=1}^\infty\int_{\U(B_r^c)}\S_{B_r}^{m,k}(E_{\eta,r}^k) d\pi_{B_r^c}(\eta) =0 \, ,
\end{equation*}
for any $r>0$. Recalling that $\rho^m_r (E) \uparrow \rho^m(E)$ by \eqref{eq: MCH}, we obtain the sought conclusion.

 \section{Functions of Bounded Variation}\label{sec:BV}

 In this section, we introduce functions of bounded variations (called {\it BV functions}) on $\U(\R^n)$ following three different approaches: the variational approach (\Sec\ref{subs: VA}), the relaxation approach (\Sec\ref{subs: RA}), and the semigroup approach (\Sec\ref{subs: HA}). In Section \ref{sec: EDBV}, we prove that they all coincide.

 \subsection{Variational approach} \label{subs: VA}
 Let us begin by introducing a class of $BV$ functions through integration by parts. We then discuss localisation properties.

 \begin{defn}[BV functions I: variational approach]\normalfont \label{defn: BV}
 Let $\Omega \subset  \R^n$ be either a closed domain with smooth boundary or $\R^n$.
 For $F \in \cup_{p>1}L^p(\U(\Omega), \p_\Omega)$, we define the total variation as 
 \begin{align}\label{equ: TV}
 \V_{\U(\Omega)}(F)
 :=\sup\biggl\{ \int_{\U(\Omega)} (\nabla_{\U(\Omega)}^* V) F d\p_\Omega: V\in \CylV_*(\U(\Omega)),\ | V |_{T\U(\Omega)} \le 1 \biggr\}.
 \end{align}
 When $\Omega=\R^n$, we simply write $\V(F) := \V_{\U(\R^n)}(F)$. We say that {\it $F$ is BV in the variational sense} if $\V(F)<\infty$.
 \end{defn}

\begin{rem} \normalfont 
	The assumption $F\in \cup_{p>1}L^p(\U(\Omega), \p_\Omega)$ plays an important role in Definition \ref{defn: BV}, ensuring that $\int_{\U(\Omega)} (\nabla^*_{\U(\Omega)} V) F d \p_\Omega$ is well defined for any $V\in \CylV(\U(\Omega))$. Indeed, one can easily prove that $\nabla_{\U(\Omega)}^* V \in \cup_{1\le p < \infty} L^p(\U(\Omega), \p_\Omega)$ for any $V\in \CylV_*(\U(\Omega))$, but it is not $L^\infty(\U(\Omega), \p_\Omega)$ in general.
\end{rem}

\begin{rem} \normalfont 
  As it was shown in Remark \ref{rem: bounded test}, the set of $V\in \CylV_*(\U(\Omega))$ with $|V|_{T\U}\le 1$ is dense in $\CylV_*(\U(\Omega))$ with respect to the topology of point-wise convergence and the $L^p(\U(\Omega) \to T\U(\Omega), \pi_\Omega)$ topology for $1 \le  p<\infty$.
\end{rem}

 In order to localise the total variation we employ a family of cylinder vector fields concentrated on $B_r$, for some $r>0$.
	
 \begin{defn}\normalfont \label{defn: LBV}
 For $F \in \cup_{p>1}L^p(\U(\R^n), \p)$, we define the {\it localised total variation} as
  \begin{align}\label{defn: TVR}
 \V_r(F):=\sup\biggl\{ \int_{\U(\R^n)} (\nabla^* V) F d\p: V\in \CylV^r_*(\U(\R^n)),\ |V|_{T\U(\R^n)} \le 1 \biggr\} \, ,
 \end{align}
 where
 \begin{equation*}
 	\CylV^r_*(\U(\R^n)):=\Bigl\{V(\gamma, x)=\sum_{i=1}^kF_i(\gamma) v_i(x): F_i \in \CylF(\U(\R^n)),\ v_i \in C_*^{\infty}(B_r; \R^n), k \in \N \Bigr\} \, .
 \end{equation*}
 \end{defn}

The next result shows that $\V_{\U(B_r)}(F_{\eta,r}) < \infty$ for $\pi_{B_r^c}$-a.e. $\eta$ whenever $\V_r(F)<\infty$. It is the key step to perform our nonlinear dimension reduction. Indeed it allows to reduce the study of BV functions on $\U(\R^n)$ to their sections, which live on the finite dimensional space $\U(B_r)$.

\begin{prop} \label{lem: BV1}
	Let $r>0$ and $p>1$. For $F \in L^p(\U(\R^n), \p)$ with $\V_r(F)<\infty$, it holds
	\begin{align} \label{ineq: BV1}
		\int_{\U(B_r^c)}\V_{\U(B_r)}(F_{\eta, r}) d\p_{B_r^c}(\eta) = \V_r(F) \, .
	\end{align}
\end{prop}

Let us begin with a simple technical lemma.

 \begin{lem} \label{lem: BV0}
	Let $r>0$. For $V \in \CylV^r_*(\U(\R^n))$, and $F\in \CylF(\U(\R^n))$ it holds
	\begin{align} \label{ineq: BV1a}
		\int_{\U(B_r^c)} \biggl( \int_{\U(B_r)} F_{\eta,r}(\gamma) \nabla^*_{\U(B_r)}V_{\eta, r}(\gamma)  d\p_{B_r}(\gamma) \biggr) d\p_{B_r^c}(\eta) = \int_{\U(\R^n)} F \nabla^* V d\p \, .
	\end{align}
\end{lem}

\begin{proof}[Proof of Lemma \ref{lem: BV0}]
Recall that for $r>0$ and $\eta\in \U(B_r^c)$ we have $V_{\eta,r} \in \CylV_*(B_r)$.
By the divergence formula \eqref{eq: DIV} and the disintegration Lemma \ref{lem: DI2},
we have that 
\begin{align*} 
&\int_{\U(B_r^c)} \biggl( \int_{\U(B_r)} F_{\eta,r}(\gamma)\nabla^*_{\U(B_r)}V_{\eta, r}(\gamma)   d\p_{B_r}(\gamma) \biggr) d\p_{B_r^c}(\eta) 
\\
&=-\int_{\U(B_r^c)} \biggl( \int_{\U(B_r)} F_{\eta,r}(\gamma) \Bigl(\sum_{i=1}^k \nabla_{v_i} (F_i)_{\eta, r}(\gamma) + \sum_{i=1}^k(F_i)_{\eta, r}(\gamma)(\nabla^*_{\R^n} v_i)^\ast (\gamma) \Bigr)d\p_{B_r}(\gamma) \biggr) d\p_{B_r^c}(\eta) 
\\
&=-\int_{\U(B_r^c)}  \int_{\U(B_r)} \biggl( F\Bigl(\sum_{i=1}^k \nabla_{v_i} F_i + \sum_{i=1}^kF_i(\nabla^*_{\R^n} v_i)^\ast  \Bigr) \biggr)_{\eta, r}(\gamma)d\p_{B_r}(\gamma) d\p_{B_r^c}(\eta) 
\\
&=-\int_{\U(\R^n)}   F\Bigl(\sum_{i=1}^k \nabla_{v_i} F_i + \sum_{i=1}^kF_i(\nabla^*_{\R^n} v_i)^\ast  \Bigr) d\p 
\\
&= \int_{\U(\R^n)} F \nabla^* V d\p \, . \qedhere
\end{align*}
\end{proof}

\begin{proof}[Proof of Proposition \ref{lem: BV1}]
We first prove that
\begin{align}\label{eq:ge}
	\int_{\U(B_r^c)}\V_{\U(B_r)}(F_{\eta, r}) d\p_{B_r^c}(\eta) \ge  \V_r(F) \, .
\end{align}
Let $V_i \in \CylV^r_*(\U(\R^n))$ with $|V_i|_{T\U}\le 1$ so that 
$$ 
\V_r(F) = \lim_{i \to \infty} \int_{\U(\R^n)} (\nabla^* V_i) F d\p.
$$
Observe that $(V_i)_{\eta, r} \in \CylV_*(\U(B_r))$, then by definition of $\V_{\U(B_r)}(F_{\eta, r})$ we get
$$
\int_{\U(B_r)} ((\nabla^* V_i) F )_{\eta,r} d\p_{B_r}
=
\int_{\U(B_r)} (\nabla_{\U(B_r)}^* (V_i)_{\eta,r}) F_{\eta, r} d\p_{B_r}  \le \V_{\U(B_r)}(F_{\eta, r}), \quad i \in \N \, .
$$
Therefore, by Lemma \ref{lem: BV0}, 
\begin{align*}
\V_r(F) &=  \lim_{i \to \infty} \int_{\U(\R^n)} (\nabla^* V_i) F d\p 
\\
&=  \lim_{i \to \infty} \int_{\U(B_r^c)}  \int_{\U(B_r)} (\nabla_{\U(B_r)}^*  (V_i)_{\eta,r}) F_{\eta, r} d\p_{B_r} d\p_{B_r^c}(\eta) 
\\
&\le  \int_{\U(B_r^c)}   \V_{\U(B_r)}(F_{\eta, r}) d\p_{B_r^c}(\eta) \, ,
\end{align*}
which completes the proof of \eqref{eq:ge}.

\medskip

Let us now pass to the proof of the opposite inequality
\begin{align}\label{eq:le}
	\int_{\U(B_r^c)}\V_{\U(B_r)}(F_{\eta, r}) d\p_{B_r^c}(\eta) \le  \V_r(F) \, .
\end{align}
The idea of the proof is inspired by \cite[Proposition 3.2]{H10} in the case of the Wiener space. 
We divide it into three steps.

\noindent
{\bf Step 1.} We show the existence of $\{ V_i\, : \, i\in \N\} \subset \CylV_*(\U(B_r))$ such that $|V_i|_{T\U} \le 1$ and
\begin{equation} \label{eq: 5.6.1}
	\V_{\U(B_r)} (G)
	= \sup_{i\in \N} \int_{\U(B_r)} (\nabla_{\U(B_r)}^*V_i) G d\p_{B_r} \, ,
\end{equation}
for any $G\in \cup_{p>1} L^p(\U(B_r), \p_{B_r})$.

First we observe that there exists $\mathcal{F}F:=\{G_i\, : \, i\in \N \} \subset \CylF(\U(B_r))$ such that any cylinder function can be approximated strongly in $H^{1,q}(\U(B_r))$ for any $q<\infty$, by elements of $\mathcal{F}F$. 
Let $D\subset C_*^\infty(B_r; \R^n)$ be a countable dense subset, w.r.t.\ the $C^1$-norm:~$\|v\|_{C^1(B_r)}:=\|\nabla_{\R^n} v\|_{L^{\infty}(B_r)}~+~\|v\|_{L^{\infty}(B_r)}$.
We define the countable family
\begin{equation*}
	\mathcal{F}V:= \left\lbrace 
	 \beta V(\gamma,x) \phi_\alpha(|V|_{T_\gamma\U}) \, : \, V(\gamma, x)= \sum_{i=1}^m w_i(x) G_i(\gamma), \, \, \alpha, \beta\in \mathbb{Q}^+, \, \, m\in \N, \, \, w_i\in D , 
	 \, \, G_i\in  \mathcal{F}F  
	 \right\rbrace \, ,
\end{equation*}
where $\phi_\alpha \in C^\infty([0, \infty))$ satisfies $0\le \phi_\alpha \le 1$, $|\phi'_\alpha| \le 2/\alpha$ and $\phi_\alpha(t) = 1$ on $[0,1 + \alpha]$, $\phi(t) = 0$ on $[1 + 2\alpha, \infty)$.

Fix $\delta>0$, $q\in [1, \infty)$ and $V\in \CylV_*(\U(B_r))$ with $|V|_{T_\gamma \U}\le 1$. To prove \eqref{eq: 5.6.1} it suffices to show that there exists $W\in \mathcal{F}V$ with $|W|_{T\U}\le 1$ such that $\| \nabla^*_{\U(B_r)} (V-W) \|_{L^q(\U(B_r))} \le \delta$.

Fix $t\in (q, 2q)$ and $\e\in (0, 1/9)$. 
Letting $V= \sum_{i=1}^m F_i v_i \in \CylV_*(\U(B_r))$,  we pick $G_i\in \mathcal{F}F$ and $w_i\in D$ such that
\begin{equation}
  \sum_{i=1}^m  \bigl( \| v_i - w_i \|_{C^1(B_r)} + \| F_i - G_i\|_{L^t(\U(B_r))} + \| \nabla_{\U(B_r)} (F_i - G_i) \|_{L^t(\U(B_r))} \bigr) < \e ,
\end{equation}
and consider $\bar W:= \sum_{i=1}^m w_i G_i$. By using the divergence formula \eqref{eq: DIV}, we can obtain that 
\begin{equation}\label{eq: z}
	\int_{\U(B_r)}  |\nabla^*_{\U(B_r)} ( \bar W - V)|^{t} d\p_{B_r} 
	+
	\int_{\U(B_r)} \bigl| | \bar W|_{T_\gamma \U} - |V|_{T_\gamma \U}\bigr|^{t} d \pi_{B_r}
	\le C \e^{t} \, ,
\end{equation}
where $C = \max\{\|w_i\|_{C^1}, \|G_i\|_{L^t(\U(B_r))}, \|\nabla G_i\|_{L^t(\U(B_r))} : 1 \le i \le m\}$ does not depend on $\e$.
We assume without loss of generality that $\e, \e^{\frac{1}{10t}}\in \mathbb{Q}$ and set
\begin{equation} \label{eq: CPA}
	W := (1-2\e^{\frac{1}{10 t}})\phi_{\e^{\frac{1}{10t}}}\bigl(|\bar W|_{T_\gamma \U}^2 \bigr) \bar W \in \mathcal{F}V
	\, ,
\end{equation}
which satisfies 
$$
|W|_{T_\gamma\U}
=(1-2\e^{\frac{1}{10t}})\phi_{\e^{\frac{1}{10t}}}\bigl( |\bar W|_{T_\gamma \U}^2 \bigr)|\bar W |_{T_\gamma \U} \le (1-2\e^{\frac{1}{10t}})(1+ 2\e^{\frac{1}{10t}}) \le 1 \, .
$$
We now check that $\| \nabla^*_{\U(B_r)}(V-W)\|_{L^q(\U(B_r))} \le \delta$.
From the identity
\begin{equation*}
	\nabla^*_{\U(B_r)} W =
	(1-2\e^{\frac{1}{10t}}) \phi_{\e^{\frac{1}{10t}}} \bigl( |\bar W|_{T_\gamma \U}^2 \bigr) (\nabla_{\U(B_r)}^* \bar W) 
	-2(1-2\e^{\frac{1}{10t}}) \phi_{\e^{\frac{1}{10t}}}'\bigl(|\bar W|_{T_\gamma \U}^2 \bigr) |\bar W|_{T_\gamma \U}^2 \, ,
\end{equation*}
and the inequality 
\begin{equation*}
\bigl|\phi'_{\e^{\frac{1}{10t}}}\bigl(|\bar W|_{T_\gamma \U}^2 \bigr)\bigr| |\bar W|_{T_\gamma \U}^2
\le 2\e^{-\frac{1}{10t}}\chi_{\{1 + \e^{\frac{1}{10t}}  \le |\bar W|_{T_\gamma \U}^2 \le 1 + 2\e^{\frac{1}{10t}}\}} |\bar W|_{T_\gamma \U}^2
\le 5 \e^{-\frac{1}{10t}} \chi_{\{|\bar W|_{T_\gamma \U}^2 \ge 1+ \e^{\frac{1}{10t}}\}} \, ,
\end{equation*}
we obtain
\begin{align}
  &\|\nabla^*_{\U(B_r)} ( W - \bar W)\|_{L^q}  \le 
  \Bigl\| \Bigl( (1-2\e^{\frac{1}{10t}}) \phi_{\e^{\frac{1}{10t}}}(|\bar W|_{T_\gamma \U}^2) - 1\Bigr) (\nabla_{\U(B_r)}^* \bar W)\Bigr\|_{L^q(\U(B_r))} 
  + 5\e^{-\frac{1}{10t}}\Bigl\|\chi_{\{|\bar W|_{T_\gamma \U}^2 \ge 1+ \e^{\frac{1}{10t}}\}} \Bigr\|_{L^q(\U(B_r))} \notag
  \\& \le 5\e^{\frac{1}{10t}} \|\nabla^*_{\U(B_r)} \bar W\|_{L^q(\U(B_r))} + \Bigl\|\chi_{\{ |\bar W|_{T_\gamma \U}^2 \ge 1 + \e^{\frac{1}{10t}}\}} (\nabla_{\U(B_r)} ^* \bar W) \Bigr\|_{L^q(\U(B_r))} + 5\e^{-\frac{1}{10t}}\Bigl\|\chi_{\{|\bar W|_{T_\gamma \U}^2 \ge 1+ \e^{\frac{1}{10t}}\}} \Bigr\|_{L^q(\U(B_r))} \notag
  \\&
   \le C\Bigl(\| \nabla_{\U(B_r)}^* \bar W \|_{L^t(\U(B_r))}, t, q\Bigr) \Bigl( \e^{\frac{1}{10t}} + \e^{-\frac{1}{10t}}\Bigl\|\chi_{\{|\bar W|_{T_\gamma \U}^2 \ge 1+ \e^{\frac{1}{10t}}\}} \Bigr\|_{L^t(\U(B_r))} \Bigr)
   \, ,
\end{align}
where we estimated $\|\chi_{\{ |\bar W|_{T_\gamma \U}^2 \ge 1 + \e^{\frac{1}{10t}}\}} (\nabla_{\U(B_r)} ^* \bar W) \|_{L^q(\U(B_r))}$ by means of the H\"older inequality and using that $t<2q$.
The Chebyshev inequality and \eqref{eq: z} give
\begin{align*}
	\Bigl\|\chi_{\{|\bar W|_{T_\gamma \U}^2 \ge 1+ \e^{\frac{1}{10t}}\}} \Bigr\|_{L^t(\U(B_r))}
	& \le \Bigl\|\chi_{\{|\bar W|_{T_\gamma \U} \ge 1+ \e^{\frac{1}{3t}}\}} \Bigr\|_{L^t(\U(B_r))}
	 \le \Bigl\|\chi_{\{||\bar W|_{T_\gamma \U}- |V|_{T_\gamma\U}| \ge \e^{\frac{1}{3t}}\}} \Bigr\|_{L^t(\U(B_r))}
	\\& \le \left( \e^{-\frac{1}{3t}} \bigl\| |\bar W|_{T_\gamma \U}- |V|_{T_\gamma\U} \bigr\|_{L^1(\U(B_r))} \right)^{1/t}
	\le C \e^{\frac{1}{t} - \frac{1}{3t^2}}
	\le C \e^{\frac{1}{2t}} \quad (\e<1),
\end{align*}
where $C = \max\{\|w_i\|_{C^1}, \|G_i\|_{L^1(\U(B_r))}, \|\nabla G_i\|_{L^1(\U(B_r))} : 1 \le i \le m\}$ is independent of $\e$.
Therefore, we conclude
\begin{align*}
	\|\nabla^*_{\U(B_r)} ( W - V)\|_{L^q(\U(B_r))}
	& \le
    \|\nabla^*_{\U(B_r)} ( W - \bar W)\|_{L^q(\U(B_r))} + \|\nabla^*_{\U(B_r)} ( \bar W - V)\|_{L^q(\U(B_r))}
    \\&
    \le C( \e^{\frac{1}{10t}} + \e^{\frac{1}{5t}}) + \e \le \delta \, , %C(\| \nabla^* \bar W \|, t, q)
\end{align*}
provided $\e$ is small enough. The proof of \eqref{eq: 5.6.1} is complete.
\smallskip

\noindent
{\bf Step 2.} We conclude the proof of \eqref{eq:le}.

Note that the map $\gamma \mapsto F(\gamma) \nabla^*_{\U(B_r)}V(\gamma|_{B_r})$ is $\p$-measurable. Furthermore, by Lemma \ref{lem: M2}, $F_{\eta, B_r^c}$ is $\p_{B_r}$-measurable and 
the map 
$$ \U(B_r^c)\ni  \eta\mapsto \int_{\U(B_r)} (\nabla_{\U(B_r)}^* V) F_{\eta, B_r^c}d\pi_{B_r}$$
is $\p_{B_r^c}$-measurable. Therefore, the map $\eta \mapsto \V_{\U(B_r)}(F_{\eta, B_r^c})$ is $\p_{B_r^c}$-measurable.

Fix now $\e>0$ and define a sequence $\{C_j\, : \, j\in \N\}$ of subsets in $\U(B_r^c)$ so that $C_* = \emptyset$, and 
%\footnote{EB: I have never seen the contraction m'ble. I think it is better to avoid it and to split the formula in two lines}
\begin{equation*}
	\begin{split}
	C_j :=
	 \Big\{
	 \eta \in \U(B_r^c):\,  F_{\eta, r}\, \text{is}\,  \ \p_{B_r}\text{-measurable and},  
	 \qquad\qquad\qquad\qquad\qquad\qquad\qquad\qquad
	 \\ 
	 \int_{\U(B_r)}(\nabla_{\U(B_r)}^* V_j) F_{\eta, r} d\p_{B_r} \ge (1-\e) \V_{\U(B_r)}(F_{\eta, r}) \wedge \e^{-1}
	 \Big\} 
	\setminus \bigcup_{i=1}^{j-1}C_i \, ,
    \end{split}
\end{equation*}
where the family $\{V_i\, : \, i\in\N\}$ has been built in Step 1.

Then, $C_j$ is $\p_{B_r^c}$-measurable for any $j$ and $\pi_{B_r^c}(\U(B_r^c) \setminus \cup_{j=1}^\infty C_j)=0$. 
Set 
$$
W_n^\eta(\gamma) := W_n(\gamma + \eta) := \sum_{j=1}^n V_j(\gamma) \chi_{C_j}(\eta), \quad \gamma \in \U(B_r),\ \eta \in \U(B_r^c) \, .
$$  
We approximate $\chi_{C_j}$ by $\{F^i_j\}_{i \in \N} \subset \CylF(\U(B_r^c))$ with $|F^i_j| \le 1$ in the strong $L^{p'}(\U(B_r^c), \p_{B_r^c})$ topology, where $\frac{1}{p'} + \frac{1}{p} = 1$. Thus, setting $W_n^i(\gamma+\eta):=\sum_{j=1}^n V_j(\gamma) F^i_j(\eta)$, we see that 
$$
\int_{\U(B_r^c)} \| \nabla^*_{\U(B_r)}(W_n - W_n^i)(\cdot + \eta) \|_{L^{p'}(\U(B_r))} d \pi_{B_r^c}(\eta) \to 0
\quad
\text{as $i \to \infty$} \, .
$$
Notice that $W^i_n \in \CylV^r_*(\U(\R^n))$, hence
\begin{align} \label{eq: BV1-1}
	\lim_{i \to \infty}\int_{\U(B_r^c)} &\left( \int_{\U(B_r)} (\nabla_{\U(B_r)}^* W_n^i(\cdot+\eta)) f_{\eta,r} d\p_{B_r} \right) d\p_{B_r^c}(\eta)
	 \\& = \int_{\U(B_r^c)} \int_{\U(B_r)} (\nabla_{\U(B_r)}^* W^\eta_n) f_{\eta,r} d \p_{B_r} \p_{B_r^c}(\eta) \notag
	 \\& = \notag
	  \int_{\U(B_r^c)} \left( \sum_{j=1}^n \chi_{C_j}(\eta) \int_{\U(B_r)}(\nabla_{\U(B_r)}^* V_j) f_{\eta,r} d \p_{B_r} \right) d \p_{B_r^c}(\eta)
	 \\& 
	 \ge (1-\e) \int_{\U(B_r^c)} \left( \sum_{j=1}^n \chi_{C_j}(\eta) \V_{\U(B_r)}(f_{\eta, r}) \wedge \e^{-1}   \right) d \p_{B_r^c} \notag
	 \\&
	 = (1-\e) \int_{\cup_{j=1}^n C_j} \V_{\U(B_r)}(f_{\eta,r})\wedge \e^{-1} d \pi_{B_r^c}(\eta).  \notag
\end{align}
By Lemma \ref{lem: BV0}, 
\begin{equation}
	\int_{\U(\R^n)} (\nabla^* W_n^i) f d \p
	=
	\left( \int_{\U(B_r)} (\nabla_{\U(B_r)}^* W_n^i(\cdot+\eta)) f_{\eta,r} d\p_{B_r} \right) d\p_{B_r^c}(\eta) \, ,
\end{equation}
which along with \eqref{eq: BV1-1} gives the claimed inequality by letting $i\to \infty$ and $n\to \infty$.
\end{proof}

\subsection{Relaxation approach} \label{subs: RA}
In this subsection we introduce a second notion of functions with bounded variations. We rely on a relaxation approach.
%\footnote{We can define the measure $|\DDD_*F|$ as an outer measure. What do we want?}

\begin{defn}[BV functions II: Relaxation]\label{defn: BVrelaxation} \normalfont
	Let $F\in L^1(\U(\R^n), \pi)$, we define {\it the total variation of $F$} by 
	\begin{equation} \label{defn: BVR}
		| \DDD_* F |(\U(\R^n)) := \inf \{  \liminf_{n\to \infty} \| \nabla F_n \|_{L^1(T\U)} \, : \,  F_n \to F \, \text{in }\, L^1(\U(\R^n), \pi)\, , \,  F_n\in \CylF(\U(\R^n)) \} \, .
	\end{equation}   
   If $|\DDD_*F|(\U(\R^n)) < \infty$, we say that $F$ has finite {\it relaxed total variation}.
\end{defn}

\begin{defn}[Total variation pre-measure]
	\label{def: totvarpremeas}
	\normalfont
	If $|\DDD_*F|(\U(\R^n)) < \infty$, we define a map
	$$
	|\DDD_*F|: \{G\in \CylF(\U(\R^n))\, : \, \text{$G$ is non-negative}\} \to \R \, ,
	$$ 
	\begin{equation} \label{defn: BVRlocal}
		| \DDD_* F |[G] := \inf \left\lbrace  \liminf_{n\to \infty} \int_{\U(\R^n)} G |\nabla F_n |_{T\U}d \pi \, : \,  F_n \to F \, \text{in }\, L^1(\U(\R^n),\pi)\, , \,  F_n\in \CylF(\U(\R^n)) \right\rbrace \, .
	\end{equation}
\end{defn}

Notice that $|\DDD_*F|[G] \le \| G \|_{L^\infty} |\DDD_*F|$ and $|\DDD_*F|[G_1 + G_2] \ge |\DDD_*F|[G_1] + |\DDD_*F|[G_2]$. By construction,  $| \DDD_* F |[G]$ is the lower semi-continuous envelope of the functional $\CylF(\U(\R^n)) \ni F \mapsto \int_{\U(\R^n)} G |\nabla F |_{T\U}d \pi$. Therefore, the map $F \mapsto |\DDD_*F|(G)$ is lower semi-continuous with respect to the $L^1$-convergence for any non-negative $G\in \CylF(\U(\R^n))$.

\medskip
It will be shown in Corollary~\ref{cor: measure} that $|\DDD_*F|$ is represented by a finite measure $|\DDD F|$, i.e. 
\begin{equation*}
	|\DDD_*F|[G] = \int_{\U(\R^n)} G  d|\DDD F| \, \quad
	\text{for any non-negative $G\in \CylF(\U(\R^n))$} \, .
\end{equation*}

%\purple{\begin{rem}\normalfont
%	Our definitions above deviate slightly from the ones considered in the literature of BV functions on metric measure spaces \cite{AmbDiM14,MirandaJr}.
%	Given a metric measure structure $(X,d,m)$ and $F\in L^1(X,m)$, usually the total variation $|\DDD_* F|(X)$ is defined by considering regularisation sequences $F_n\in \text{Lip}_b(X)$ that are Lipschitz and bounded. The supports of the latter are fine enough to prove that, for any open set $A\subset X$, the map
%	\begin{equation}\label{eq:classictotalvar}
%		| \DDD_* F |(A) := \inf \{  \liminf_{n\to \infty} \| \nabla F_n \|_{L^1(A)} \, : \,  F_n \to F \, \text{in }\, L^1(A,m)\, , \,  F_n\in \text{Lip}_b(X) \} \, ,
%	\end{equation}
%	is a restriction to open set of a finite measure, provided $|\DDD_* F|(X) < \infty$.
%	To capture the geometric-analytic structure of the configuration space $\U(\R^n)$, which is an extended metric measure structure, it is natural to replace Lipschitz functions by cylinder functions as a class of test objects. The supports of the latter, however, are coarser than open sets, hence it is not natural to define the total variation measure via \eqref{eq:classictotalvar}. 
%	Definition \ref{def: totvarpremeas} provides the natural replacement of \eqref{eq:classictotalvar} in our setting.	
%\end{rem}
%}

\subsection{Heat semigroup approach} \label{subs: HA}
In this subsection we present the third approach to ${\rm BV}$ functions. We employ the heat semigroup to define the total variation of a function $F\in L^p(\U(\R^n), \pi)$, $p>1$.

\begin{prop} \label{prop: Tf}
	Let $F \in \cup_{p>1} L^p(\U(\R^n), \p)$. Then $\| \nabla T_t F\|_{L^1}<\infty$ for $t>0$ and the following limit exists
	\begin{equation}
		\T(F):= \lim_{t\to 0}\|\nabla T_t F\|_{L^1} \, .
	\end{equation}  
\end{prop}

\begin{defn}[BV functions III: Heat semigroup] \normalfont \label{defn: HS}
A function $F\in \cup_{p>1}L^p(\U(\R^n), \p)$  is {\it BV in the sense of the heat semigroup} if $\T(F)<\infty$. We define {\it the total variation of $F$} by $\T(F)$.
%\footnote{Do we define "the total variation measure" for the heat semigroup approach? For the other two approaches, we define it by $\|E\|$ and $|\DDD F|$, but for the semigroup approach, we only defined the total variation mass $T(F)(\U(\R^n))$.}
\end{defn}
To prove Proposition \ref{prop: Tf}, we need the {\it Bakry--\'Emery inequality} with exponent $q=1$, i.e. for any $t,s>0$, $F\in \cup_{p>1}L^p(\U(\R^n),\pi)$, it holds
\begin{equation}\label{eq: 1-BE}
	\int_{\U(\R^n)} |\nabla T_t F| d\p < \infty \, ,
	\quad\quad
	|\nabla T_{t+s} F| \le T_t|\nabla T_s F|
	\quad \p\text{-a.e.} \, .
\end{equation}
The inequality \eqref{eq: 1-BE} will be proven in Corollary \ref{rem: BE} in Section \ref{sec:BE}. Let us now use it to show Proposition \ref{prop: Tf}.

\begin{proof}[Proof of Proposition \ref{prop: Tf}]
Let $F \in L^p(\U, \p)$ for $p>1$. 
By \eqref{eq: 1-BE}, we see that 
\begin{align*}
	\|\nabla T_tF\|_{L^1} \le \liminf_{s \to 0}\|\nabla T_{t+s} F\|_{L^1}\le  \liminf_{s \to 0} \|\nabla T_sF\|_{L^1} \, .
\end{align*}
By taking $\limsup_{t \to 0}$, we obtain $\limsup_{t \to 0}\|\nabla T_tF\|_{L^1} \le  \liminf_{s \to 0}\|\nabla T_sF\|_{L^1}$, which concludes the proof. 
\end{proof}

\subsection{$p$-Bakry--\'Emery inequality}\label{sec:BE}
In order to complete the proof of Proposition \ref{prop: Tf}, 
we show the $p$-Bakry--\'Emery inequality for the {\it Hodge heat flow}, which implies in turn the scalar version \eqref{eq: 1-BE} of the $p$-Bakry--\'Emery inequality. 
It will play a significant role also in the proof of Theorem~\ref{prop: IBV}.
Recall that, for $F=\Phi(f_1^*,\ldots, f_k^*) \in \CylF(\U(\R^n))$, 
\begin{align} \label{eq:GLC}
\nabla F(\gamma, x)&=\sum_{i=1}^k\partial_i\Phi(f_1^*\gamma,\ldots, f_k^*\gamma)\nabla_{\R^n}f_i(x) \ ,
\\
\Delta F(\gamma)&= \sum_{i, j=1}^k \partial_{ij}^2 \Phi (f_1^*\gamma, \ldots, f_k^*\gamma) \langle \nabla_{\R^n}f_i, \nabla_{\R^n}f_j \rangle_{T_\gamma \U}
 + \sum_{i=1}^k \partial_{i} \Phi(f_1^*\gamma, \ldots, f_k^*\gamma) (\Delta_{\R^n} f_i)^*\gamma \ , \notag
\end{align}
where $\langle \nabla_{\R^n}f_i, \nabla_{\R^n}f_j \rangle_{T_\gamma \U}:=(\langle \nabla_{\R^n}f_i, \nabla_{\R^n}f_j \rangle_{T_x \U})^*\gamma := \int_{\R^n} \langle \nabla_{\R^n}f_i, \nabla_{\R^n}f_j \rangle_{\R^n}(x) d\gamma(x)$. See e.g., \cite[(4.7)]{AKR98} for the proofs.

\begin{defn}[Hodge Laplacian] \label{defn: HL1} \normalfont
	For $V=\sum_{k=1}^mF_k v_k$ with $F_k=\Phi_k ((f^k_1)^\ast, \ldots, (f^k_\ell)^\ast)$, define Hodge Laplacian of $V$ as
	\begin{align} \label{defn: HL}
		\Delta_H V(\gamma, x)
		&:=\sum_{k=1}^m \sum_{i,j=1}^\ell \partial_{ij}^2\Phi_k\bigl( (f^k_1)^\ast\gamma,\ldots ,(f^k_\ell)^\ast\gamma \bigr)\Bigl(\bigl\langle \nabla_{\R^n} f^k_i, \nabla_{\R^n} f^k_j \bigr\rangle_{T_x\R^n}\Bigr)^{*}\gamma \ v_k(x) 
		\\
		&\quad +\sum_{k=1}^m \sum_{i=1}^\ell \partial_{i}\Phi_k \bigl( (f^k_1)^\ast\gamma,\ldots,(f^k_\ell)^\ast\gamma \bigr)(\Delta_{\R^n} f_k(x))^{*}\gamma \ v_k(x) \notag
		\\
		&\quad +\sum_{k=1}^m\Phi_k\bigl((f^k_1)^\ast\gamma,\ldots,(f^k_\ell)^\ast\gamma\bigr) \Delta_{H, \R^n} v_k(x), 
		\\
		&\quad +2\sum_{k=1}^m \sum_{i=1}^\ell \partial_i\Phi_k \bigl( (f^k_1)^\ast\gamma,\ldots,(f^k_\ell)^\ast\gamma \bigr)
		(\nabla_{\R^n} f_i^k \cdot \nabla_{\R^n}) v_k(x)
		\, \notag
	\end{align}
    where $\Delta_{H, \R^n} v_k$ is the Hodge Laplacian of $v_k\in C^\infty(\R^n;\R^n)$, and $(\nabla_{\R^n} f_i^k \cdot \nabla_{\R^n}) v_k(x)$ is the vector field whose $i$th coordinate coincides with $\bigl\langle \nabla_{\R^n} f^k_i, \nabla_{\R^n} (v_k)_i \bigr\rangle_{T_x\R^n}$.
     It turns out that $\Delta_H V$ does not depend on the choice of both the representative of $V$ and the inner and outer functions of $F_k$ (see \cite[Theorem 3.5]{ADL01}). 
\end{defn}

%\begin{rem} \label{r:SE}
%Similar expressions to \eqref{eq:GLC} and \eqref{defn: HL} hold in the case where  $\Cyl(\U(\R^n))$ is replaced by exponential cylinder functions ${\rm ECyl}(\U(\R^n))$ defined in \eqref{defn: Ecyl}. See e.g., \cite[Proposition 4.1]{AKR98}.
%\end{rem}
For the proof of Theorem~\ref{thm: IT} below, we introduce the following space of exponential cylinder functions {\it with Schwartz inner functions}:
  $$
  \ECylF_{\mathcal S}(\U(\R^n)):=\mathrm{Span}_{\R}\Bigl\{
  \exp\bigl\{\log (1+f)^\ast \bigr\} \,  : \, 
  f \in \mathcal S \,,\ -\delta \le f \le 0 \, \,  \text{for some $\delta\in (0,1)$} 
  \Bigr\} \, ,
  $$  
  where $\mathcal S$ is the space of Schwartz functions in $\R^n$ (i.e., functions in $\R^n$ whose derivatives are all rapidly decreasing). We note that $T_t \ECylF_{\mathcal S}(\U(\R^n)) \subset \ECylF_{\mathcal S}(\U(\R^n))$ for every $t>0$, and that $(\Delta, \ECylF_{\mathcal S}(\U(\R^n)))$ is essentially self-adjoint in $L^2(\U(\R^n), \pi)$ exactly by the same proof as in \cite[Theorem 4.2]{AKR98}. 
\begin{rem}\normalfont
Exponential cylinder functions have been originally discussed in~\cite{AKR98}, where they choose a larger class of inner functions. We introduced $\ECylF_{\mathcal S}(\U(\R^n))$ with inner functions in the space $\mathcal S$ of Schwartz functions for the proof of Theorem~\ref{thm: IT}, where we need to choose a smaller class of inner functions to approximate $\ECylF_{\mathcal S}(\U(\R^n))$ by cylinder functions in a sufficiently good way. See the last paragraph of the proof of~Theorem~\ref{thm: IT}. 
\end{rem}

We define the corresponding energy functional:
\begin{equation}
	\EE_H(V, W) := \langle - \Delta_H V, W \rangle_{L^2(T\U, \pi)}
	= \int_{\U(\R^n)}  \mathbf \Gamma^\U(V, W) d\p \, , 
	\quad V,W \in  \CylV(\U(\R^n))\, ,
\end{equation}
where $\mathbf \Gamma^\U$ denotes the square field operator associated with $\Delta_H$.
By \cite[Theorem 3.5]{ADL01}, the form $\EE_H$ is closable on $\CylV(\U(\R^n))$ and the corresponding closure is denoted by $\mathcal D(\EE_H)$ and the corresponding (Friedrichs) extension of $\CylV(\U(\R^n))$ is denoted by $\mathcal D(\Delta_H)$. Let $\{\mathbf T_t\}$ denote the corresponding $L^2$-semigroup. It holds that
\begin{equation}\label{eq: important inclusion}
	\mathbf T_t V \in \mathcal D(\EE_H) \, , \quad
	\text{for any $t\ge 0$ and $V\in \CylV(\U(\R^n))$} \,.
\end{equation}

The following intertwining property holds.

\begin{thm} \label{thm: IT}
	$\nabla T_t F =\mathbf T_t \nabla F$ for any $t \ge 0$ and for any $F \in H^{1,2}(\U(\R^n), \p)$.
\end{thm}
\begin{proof}
We apply~\cite[Theorem 2.1]{S06} with $\mathcal D=\CylF(\U(\R^n))$, $D = \nabla$, $A=\Delta$, $\hat{A} = \Delta_H$, $\hat{T}_t = \mathbf T_t$, $R=0$, which concludes the sought statement. To do so, we verify Conditions (i)--(iv) of  ~\cite[Theorem 2.1]{S06}.
Condition (i) and (ii) are straightforward by construction. Using the commutation $\nabla_{\R^n} \Delta_{\R^n} = \Delta_{H, \R^n} \nabla_{\R^n}$ and the representation \eqref{eq:GLC} and \eqref{defn: HL}, we can readily verify Condition~(iv), i.e., $\nabla \Delta F = \Delta_H \nabla F$ for any $F \in \CylF(\U(\R^n))$.  

We now verify Condition (iii), viz., $(\lambda - \Delta) \CylF(\U(\R^n)) \subset  H^{1,2}(\U(\R^n), \pi)$ is dense for sufficiently large $\lambda>0$. We prove it with $\lambda =0$, viz., $\Delta \CylF(\U(\R^n)) \subset  H^{1,2}(\U(\R^n), \pi)$ is dense. 
We first prove that $\Delta \ECylF_\mathcal S(\U(\R^n)) \subset  H^{1,2}(\U(\R^n), \pi)$ is dense.
Define $L:=\{F \in \Delta \mathcal D(\Delta): F \in H^{1,2}(\U(\R^n), \pi)\}$. By Lemma~\ref{lem: ACS111} below, $\Delta \mathcal D(\Delta) \subset L^2(\U(\R^n), \pi)$ is dense. Furthermore, 
$$T_t\Delta \mathcal D(\Delta) = \Delta T_t \mathcal D(\Delta) \subset \Delta \mathcal D(\Delta) \cap H^{1,2}(\U(\R^n), \pi)\ .$$ In particular, $T_t\Delta \mathcal D(\Delta)  \subset L$. Combining \cite[(4.26)]{AmbGigSav14} with the fact that $\EE$ coincides with the Cheeger energy associated with the $L^2$-transportation distance~${\sf d}_\U$ and the Poisson measure $\pi$ (see \cite[Proposition 2.3]{EH15}), we have the following regularisation inequality
\begin{align}\label{e:gi}
\EE(T_t F) \le \frac{\|F\|_{L^2}^2}{2t} \qquad t>0 \ .
\end{align}
Therefore, combined with the density $\Delta \mathcal D(\Delta) \subset L^2(\U(\R^n), \pi)$, the space~$\mathcal T:=\cup_{t>0}T_t\Delta \mathcal D(\Delta)$ is weakly dense in $H^{1,2}(\U(\R^n), \pi)$. As $\mathcal T$ is a convex subset in $H^{1,2}(\U(\R^n), \pi)$, by Mazur's lemma, 
\begin{align}\label{s:SD}
\text{$\mathcal T$ is  strongly dense in $H^{1,2}(\U(\R^n), \pi)$} \, .
\end{align} 
% and the density $\Delta \mathcal D(\Delta) \subset L^2(\U(\R^n), \pi)$, we have that $T_t\Delta \mathcal D(\Delta)  \subset H^{1,2}(\U(\R^n), \pi)$ is dense. 
For every $G \in \mathcal T = \cup_{t>0}T_t \Delta  \mathcal D(\Delta) = \cup_{t>0}\Delta  T_t \mathcal D(\Delta)$ with an expression $G = \Delta T_t F$ with $F \in \mathcal D(\Delta)$ for some $t>0$, we can take $F_n \in \ECylF_\mathcal S(\U(\R^n))$ so that 
\begin{align}\label{e:ESA1}
\|\Delta F_n-\Delta F\|_{L^2} + \|F_n - F\|_{L^2} \to 0
\end{align} 
by the essential self-adjointness of $(\Delta, \ECylF_\mathcal S(\U(\R^n)))$. Furthermore, it can be readily verified that 
\begin{align}\label{e:ESA1}
\|\Delta T_t F_n-\Delta T_t F\|_{L^2} + \|T_t F_n - T_t F\|_{L^2} \to 0
\end{align} 
by the $L^2$-contraction property of $T_t$ and the commutation $\Delta T_t = T_t \Delta$ for $t>0$. Noting $T_{t} F_n \in \ECylF_\mathcal S(\U(\R^n))$ by the stability of $\ECylF_\mathcal S(\U(\R^n))$ under the action of $T_t$, the formula~\eqref{e:ESA1} particularly shows that the sequence~$(\Delta T_t F_n)_{n \in \N} \subset \Delta\ECylF_\mathcal S(\U(\R^n))$ approximates $G = \Delta T_t F\in \mathcal T$ in the strong $L^2$-topology. Furthermore, by using~\eqref{e:gi} again, we have the uniform energy bound:
\begin{align}\label{in:gib}
\sup_{n \in \N}\EE(\Delta T_tF_n) =\sup_{n \in \N}\EE(T_t\Delta  F_n)  \le \sup_{n \in \N}\frac{1}{2t} \|\Delta F_n\| <\infty.\end{align}
 For every $H \in \mathcal D(\Delta)$, 
\begin{align} \label{e:IBP}
&\int_{\U(\R^n)} \bigl\langle \nabla(\Delta T_t F_n-G), \nabla H \bigr\rangle_{T_\gamma \U} d\pi(\gamma)  + \int_{\U(\R^n)} (\Delta T_t F_n-G) H d\pi  
\\
&= - \int_{\U(\R^n)} (\Delta T_t F_n-\Delta T_t F)\Delta H d\pi  +\int_{\U(\R^n)} (\Delta T_t F_n-\Delta T_t F) H d\pi \notag
\\
&\xrightarrow{n \to \infty} 0 \ . \notag
\end{align}
By the uniform bound~\eqref{in:gib} and the fact that $\mathcal D(\Delta)$ is dense in $H^{1,2}(\U(\R^n), \pi)$, \eqref{e:IBP} shows that $(\Delta T_t F_n)_{n \in \N} \subset \Delta\ECylF_\mathcal S(\U(\R^n))$ converges to $G = \Delta T_t F\in \mathcal T$ weakly in $H^{1,2}(\U(\R^n), \pi)$. Thus, $\Delta\ECylF_\mathcal S(\U(\R^n))$ approximates $\mathcal T$ in the weak $H^{1,2}(\U(\R^n), \pi)$ topology.  By \eqref{s:SD} and the fact that  $\Delta\ECylF_\mathcal S(\U(\R^n))$ is a convex subspace in $H^{1,2}(\U(\R^n), \pi)$, by applying Mazur's lemma again, we conclude that  $\Delta\ECylF_\mathcal S(\U(\R^n))$ is strongly dense in $H^{1,2}(\U(\R^n), \pi)$. 

Therefore, to complete the verification of Condition (iii), it suffices to prove that $\Delta\CylF(\U(\R^n))$ approximates $\Delta\ECylF_\mathcal S(\U(\R^n))$ in $H^{1,2}(\U(\R^n), \pi)$. The idea of the proof is, however, the same as in that of \cite[Proposition~4.1]{AKR98}: for $F= \exp\bigl\{\log (1+f)^\ast \bigr\} \in \ECylF_\mathcal S(\U(\R^n))$, we can take an approximation  $f_n \in C_c^\infty(\R^n)$ of the inner function $f \in \mathcal S$ so that $F_n= \exp\bigl\{\log (1+f_n)^\ast \bigr\} \in \CylF(\U(\R^n))$ converges to $F$ in a sufficiently good way to conclude that $\Delta\CylF(\U(\R^n))$ approximates $\Delta\ECylF_\mathcal S(\U(\R^n))$ in $H^{1,2}(\U(\R^n), \pi)$. As this proof is mostly a repetition of~\cite[Proposition~4.1]{AKR98}, we omit the details here. 
\end{proof}
%and the fact that  $\Delta\CylF(\U(\R^n))$ is stable under taking convex combination, $\Delta\CylF(\U(\R^n))$ approximates $T_t\Delta \mathcal D(\Delta)$ in the strong $H^{1,2}(\U(\R^n), \pi)$ topology. As  $T_t\Delta \mathcal D(\Delta)$ is dense in $H^{1,2}(\U(\R^n), \pi)$, the proof is complete. 

\begin{lem} \label{lem: ACS111}
%$\mathcal R(\Delta_p) \subset L^p(\U, \p)$ is dense.
For $F \in L^2(\U(\R^n), \p)$, there exists $F_n \in \mathcal D(\Delta)$ so that $\|\Delta F_n - F\|_{L^2} \to 0.$
\end{lem}
\begin{proof}
We first show that $\Delta G_{\alpha}F \to \Delta G_{\beta}F$ in $L^2(\U(\R^n), \p)$ for every $F \in L^2(\U(\R^n), \pi)$ as $\alpha \to \beta$ for $\alpha, \beta>0$. 
By the resolvent equality $G_\alpha - G_\beta = (\beta-\alpha)G_\alpha G_\beta$, we have that 
$$
\|\Delta(G_\alpha-G_\beta) F\|_{L^2} = (\beta-\alpha)\|\Delta G_\alpha G_\beta F\|_{L^2} = (\beta-\alpha)\| G_\alpha \Delta G_\beta F\|_{L^2}.
$$
By the $L^2$-contraction of $\alpha G_\alpha$, 
%Since $G_\alpha H \to G_\beta H$ in $L^2(\U(\R^n), \p)$ as $\alpha \to \beta$ for every $H \in L^2(\U(\R^n), \p)$, 
we obtain 
$$
(\beta-\alpha)\| G_\alpha \Delta G_\beta F\|_{L^2} \le \frac{\beta-\alpha}{\alpha^2}\|\Delta G_\beta F\|_{L^2} \to 0, \quad \alpha \to \beta.
$$
Thus, $\Delta G_{\alpha}F \to \Delta G_{\beta}F$ as $\alpha \to \beta$ in $L^2(\U(\R^n), \p)$. 
%{\color{blue}
%	What we are looking for simply follows from
%	\begin{equation}
%		\| G_\alpha \Delta_pG_\beta f\|_p
%		=
%		\| G_\alpha (\beta G_\beta f - f)\|_p
%		\le 
%		\alpha^{-1}\|\beta G_\beta f - f\|_p
%		\le 2 \alpha^{-1} \| f \|_{L^p}
%	\end{equation}
%	where we used that
%	\begin{equation}
%		\| G_\alpha f \|_{L^p} \le \int_*^\infty e^{-t\alpha} \| T_t f \|_{L^p} \le \alpha^{-1} \| f\|_{L^p}
%	\end{equation}
%}

We now prove the sought statement. 
Let $F_n:=(1/(\alpha-1))G_{\alpha + 1/n}F \in \mathcal D(\Delta)$. Then, by the general identity $(\alpha-\Delta)G_{\alpha} = {\rm Id}$, and by the convergence $\Delta G_{\alpha}F \to \Delta G_{\beta}F$ in $L^2(\U(\R^n), \p)$ proven above, we have 
$$\Delta F_n = \frac{1}{\alpha-1}\Delta G_{\alpha + 1/n}F \xrightarrow{n \to \infty}  \frac{1}{\alpha-1}\Delta G_{\alpha}F = \frac{(\alpha-1)}{(\alpha-1)}F = F, \quad F \in L^2(\U(\R^n), \p). \qedhere$$
\end{proof}

\bigskip

\begin{thm}\label{thm: GE1}
	Let $F \in \D(\EE_H)$. Then $|\mathbf T_t F|_{T\U} \le T_t|F|_{T\U}$ $\p$-a.e.\ for every $t \ge 0.$
	In particular ${\bf T}_t$ can be extended to the $L^p$-velocity fields $L^p(T\U(\Omega), \pi_\Omega)$ for every $1\le p < \infty$.	
\end{thm}

\begin{proof}
By the Weitzenb\"ock formula \cite[Theorem 3.7]{ADL01} on $\U(\R^n)$, we can express $\Delta_H = \nabla^*  \nabla + R^\U$, where $R^\U$ is the lifted curvature tensor from the base space $\R^n$. Since $\R^n$ is flat, we can easily deduce $R^\U = 0$. 

Now, setting $\mathbf \Gamma(V, W):=\mathbf \Gamma^\U(V, W) + 2 R^\U(V, W) = \mathbf \Gamma^\U(V, W)$ we can apply \cite[Theorem 3.1]{S97} (see the proof of \cite[Theorem 3.1]{S97} for $p=1$) and \cite[Proposition 3.5]{S97}, to get the sought conclusion of the first assertion.

We now prove the second assertion. Let $V\in L^p(T\U(\Omega), \pi_\Omega)$. Then, the density of cylinder vector fields gives the existence of a sequence $V_n\in \CylF(\R^n) \subset \D(\EE_H)$ such that $|V_n - V|_{T\U}\to 0$ in $L^p(\U(\R^n), \p)$ as $n\to \infty$. We can define
\begin{equation}
	{\bf T}_t V := \lim_{n\to \infty} {\bf T}_t V_n \, .
\end{equation}
The existence of the limit follows from 
\begin{equation}
	|{\bf T}_t V_n - {\bf T}_t V_m|_{T\U} \le T_t|V_n - V_m|_{T\U} \, ,
\end{equation}
as well as the independence of the limit from the approximating sequence $(V_n)_{n\in\N}$.
\end{proof}

\begin{thm}[$p$-Bakry--\'Emery estimate]\label{rem: BE}

%\footnote{Can we prove BE for the BV norm?
   %        \\EB: yes, I think it can be done. Why do you want to prove it? In case, it is necessary to show that the heat semigroup can be extended on measures}

Let $p>1$. The following assertions hold:
\begin{itemize}
\item[(i)] $T_t: H^{1,p}(\U(\R^n), \p) \to H^{1,p}(\U(\R^n), \p)$ is a continuous operator for every $t>0$.

\item[(ii)] For every $F\in H^{1, p}(\U(\R^n), \p)$, 
	\begin{equation} \label{eq: SWP}
		|\nabla T_t F |_{T\U}^p \le T_t | \nabla F |_{T\U}^p \quad \text{$\pi$-a.e.}\, .
	\end{equation}

\item[(iii)] Let $1<p \le 2$. For every $F\in L^p(\U(\R^n),\p)$ it holds that 
	\begin{align} \label{ineq: REH}
	 \|\nabla T_t F\|_{L^p(T\U)} \le  C(p) t^{-1/2} \| F \|_{L^p}, \quad  t>0.
	\end{align}
In particular, $T_t: L^p(\U(\R^n), \p) \to H^{1,p}(\U(\R^n), \p)$ is a continuous operator for every $t>0$.

\item[(iv)]  For every $t,s>0$, $F\in L^p(\U(\R^n),\pi)$, it holds that $\|\nabla T_t F\|_{L^1(T\U(\R^n))}<\infty$ and 
\begin{equation}\label{eq: 1-BE-2}
	|\nabla T_{t+s} F|_{T\U} \le T_t|\nabla T_s F|_{T\U}
	\quad \p\text{-a.e.} \, .
\end{equation}
\end{itemize}
	\end{thm}

\begin{proof}
{\bf (i).} By Theorem \ref{thm: IT} and Theorem \ref{thm: GE1}, for any $F\in \CylF(\U(\R^n))$ it holds that 
	\begin{equation} \label{BEC-1}
		|\nabla T_t F|_{T\U} = |\mathbf T_t \nabla F|_{T\U} \le T_t |\nabla F|_{T\U}
		\quad\pi \text{-a.e.} \, .
	\end{equation}
	A simple application of Jensen's inequality to \eqref{BEC-1} gives
	\begin{equation} \label{BEC-2}
		|\nabla T_t F |_{T\U}^p \le T_t | \nabla F |_{T\U}^p \,,
		\quad
		\text{for $F \in \Cyl$ and $p\ge 1$} \, .
	\end{equation}
Let $F_n \in \Cyl$ be a $H^{1,p}(\U(\R^n), \p)$-Cauchy sequence. Then, by \eqref{BEC-2} and the invariance $\p(T_tf)=\p(f)$,
\begin{align} \label{eq: BE-1}
\int_{\U(\R^n)}|\nabla T_t(F_n-F_m)|^p_{T\U} d\p \le \int_{\U(\R^n)}T_t |\nabla (F_n - F_m)|^p_{T\U} d\p =  \int_{\U(\R^n)}|\nabla (F_n - F_m)|^p_{T\U} d\p\to 0. 
\end{align} 
Since $H^{1,p}(\U(\R^n), \p)$ is the closure of $\Cyl$ w.r.t.\ the norm $\|\nabla \cdot \|_{L^p(T\U)}+\| \cdot\|_{L^p(\U, \p)}$, by \eqref{eq: BE-1}, the operator $T_t$ is extended to $H^{1,p}(\U(\R^n), \p)$ continuously. The proof of the first assertion is complete.

{\bf (ii).} Let $F \in H^{1,p}(\U(\R^n), \p)$ and take $F_n \in \Cyl$ converging to $F$ in $H^{1,p}(\U(\R^n), \p)$. Then, by the lower semi-continuity of $|\nabla \cdot|^p_{T\U}$ w.r.t.\ the $L^p$-strong convergence,  the continuity of the $L^p$-semigroup $T_t$ and the inequality \eqref{BEC-2}, we obtain
\begin{align*}
|\nabla T_{t+s} F|^p_{T\U} &= |\nabla T_{t} T_s F|^p_{T\U} \le \liminf_{n \to \infty}|\nabla T_{t}T_s F_n|^p_{T\U}  
\le  \liminf_{n \to \infty}T_t|\nabla T_s F_n|_{T\U}^p  
\le  T_t|\nabla T_s F|_{T\U}^p.
\end{align*}
Here the last equality follows from the assertion (i).

{\bf (iii).} Let $p>1$ be fixed.
	For any $F\in \CylF(\U(\R^n))$ satisfying $F\ge 0$,  it holds
	\begin{align*}
		p(p-1)\int_*^t \int_{\U(\R^n)}|\nabla T_s F|_{T\U}^2|T_sF|^{p-2} d\p ds 
		= \int_{\U(\R^n)} | F |^{p} d \p - \int_{\U(\R^n)} | T_t F |^p d \p
		\le \int_{\U(\R^n)} | F |^p d \p \, ,
	\end{align*}
where the first equality follows by the following argument:
 \begin{align*}
 \frac{d}{dt} \int_{\U(\R^n)} |T_t F|^p d \p &= p \int_{\U(\R^n)} |T_t F|^{p-1}\Delta T_t F d \p 
 \\
 &= - p \int_{\U(\R^n)} \Bigl\langle \nabla |T_t (F)|^{p-1}, \nabla T_t F\Bigr\rangle_{T_\gamma \U} d \p(\gamma)
  \\
 &= - p(p-1) \int_{\U(\R^n)} \Bigl\langle |T_t F|^{p-2} \nabla T_t F , \nabla T_t F\Bigr\rangle_{T_\gamma\U} d\p(\gamma)
   \\
 &= - p(p-1) \int_{\U(\R^n)} |T_t F|^{p-2}  \bigl| \nabla T_t F \bigr|^2_{T_\gamma \U} d \p(\gamma)\ .
 \end{align*}
 
   % \sout{It follows by studying first $\frac{d}{dt} \int_{\U(\R^n)} |T_t (F+ \e)| d \p$, and after letting $\e\to 0$.}
    
    By the contraction property of $T_t$, we obtain
    \begin{align*}
    	\int_*^t \int_{\U(\R^n)} |\nabla T_s F|_{T\U}^p d\pi ds
    	& \le \left( \int_*^t \int_{\U(\R^n)} |T_s F|^{p} d \pi  d s \right)^{\frac{2-p}{2}} 
    	\left(   \int_*^t \int_{\U(\R^n)} |\nabla T_s F|_{T\U}^2|T_sF|^{p-2} d\p ds  \right)^{\frac{p}{2}}
    	\\& 
    	\le C t^{\frac{2-p}{2}} \| F \|_{L^p}^{p} \, .
    \end{align*}
    We now employ the Bakry--\'Emery inequality \eqref{BEC-2} combined with the contraction property of $T_t$ to show that $s \to \int_{\U(\R^n)} |\nabla T_s F |_{T\U}^p d \p$ is non-increasing, which yields 
    \begin{equation}
    	t \int_{\U(\R^n)}  |\nabla T_{t} F|_{T\U}^p d\pi
    	\le \int_*^t \int_{\U(\R^n)}|\nabla T_{s} F|_{T\U}^p d\pi ds
    	\le  C t^{\frac{2-p}{2}} \| F \|_{L^p}^{p}  \, .
    \end{equation}
This implies our conclusion for cylinder functions.
    We extended it to any $F\in L^p(\U(\R^n),\p)$ by means of a density argument. Indeed, given $F\in L^p(\U(\R^n), \pi)$, we can find $F_n\in \CylF(\U(\R^n))$ such that $F_n\to F$ in $L^p$. The continuity of the semigroup $T_t$ gives $T_t F_n \to T_t F$ in $L^p$, while the lower semi-continuity of the functional $G \to \int_{\U(R^n)} |\nabla G|_{T\U(\R^n)}^p d \pi$ with respect to the $L^p$ convergence for $p>1$ yields
     \begin{equation*}
     	\int_{\U(\R^n)}  |\nabla T_{t} F|_{T\U}^p d\pi
     	\le 
     	\liminf_{n\to \infty}
     	 \int_{\U(\R^n)}  |\nabla T_{t} F_n|_{T\U}^p d\pi
     	\le C t^{-1/2} \| F \|_{L^p} \, .
     \end{equation*}     

{\bf (iv).} Note that the assertion in the case of $1 < p \le 2$ implies the one in the case of $p>2$ by  $L^p(\U(\R^n), \pi) \subset L^q(\U(\R^n), \p)$ whenever $1 \le q \le p$. 
Thus, we only need to prove it in the case of $1 < p \le 2$. Let $F \in L^p(\U(\R^n), \p)$. Then, by the assertion (iii), $T_s F \in H^{1,p}(\U(\R^n), \p)$. Take $G_n$ converging to $T_sF$ in $H^{1,p}(\U(\R^n), \p)$. Then, up to taking a subsequence from $\{G_n\}$, and by making use of \eqref{BEC-1}, we conclude that 
\begin{align*}
|\nabla T_{t+s} F|_{T\U} &= |\nabla T_{t} T_s F|_{T\U} = \lim_{n \to \infty}|\nabla T_{t}G_n|_{T\U}  
\le  \lim_{n \to \infty}T_t|\nabla G_n|_{T\U}  
= T_t|\nabla T_s F|_{T\U} \, . \qedhere
\end{align*}
\end{proof}

 \begin{rem} \normalfont 
 In \cite{EH15} (see also \cite{DS22}), the $2$-Bakry--\'Emery estimate was proved in the case of the configuration space over a complete Riemannian manifold with Ricci curvature bound. For the purpose of the current paper, however, we need a stronger estimate, i.e., the $p$-Bakry--\'Emery estimate \eqref{eq: SWP} for arbitrary $1<p<\infty$ and also the regularity estimate \eqref{ineq: REH} of the heat semigroup, both of which do not follow only from the $2$-Bakry--\'Emery inequality.  
 \end{rem}

   \subsection{Equivalence of BV functions}
   \label{sec: EDBV}

   In Section \ref{sec:BV}, we introduced the three different definitions (the variational/the relaxation/the semigroup approaches) of BV functions. In this section we show that the three different definitions of BV functions
   are equivalent.

   \begin{thm}[Equivalence of BV functions] \label{prop: IBV}
   	Let $F\in L^2(\U(\R^n), \pi)$. Then, 
   	$$\V(F) = |\DDD_*F|(\U(\R^n)) = \T(F)\, .$$ 
   \end{thm}
   The proof of Theorem \ref{prop: IBV} will be given later in this section. 
   Thanks to Theorem \ref{prop: IBV}, we can introduce a universal definition of BV functions for $L^2(\U(\R^n), \p)$-functions. 
   \begin{defn}[BV functions] \label{defn: BVU}
   	A function $F\in  L^2(\U(\R^n))$ belongs to ${\rm BV}(\U(\R^n))$ if 
   	$$
   	\V(F) = | \DDD_*F|(\U(\R^n)) = \T(F) < \infty\, .
   	$$	
   \end{defn}

   We prepare several lemmas for the proof of Theorem \ref{prop: IBV}.
   \begin{lem}\label{lem: divheat}
   	For any $V\in \CylV(\U(\R^n))$ and $t\ge 0$ it holds
   	\begin{equation}\label{z11}
   		(\nabla^* \mathbf T_{t} V) = T_{t} (\nabla^* V) \, .
   	\end{equation}
   	In particular $(\nabla^*\mathbf T_{t} V)\in L^p(\U(\R^n))$ for every $1< p<\infty$.
   \end{lem}

   \begin{proof}
   	Let $F\in \Cyl$. By the $\p$-symmetry of $T_t$ and Theorem \ref{thm: IT}, we have that 
   	\begin{align*}
   		\int_{\U(\R^n)} F\,  T_t(\nabla^* V) d \pi
   		&=
   		\int_{\U(\R^n)} T_t F\,  (\nabla^* V) d \pi
   		=
   		-\int_{\U(\R^n)} \langle V(\gamma,\cdot ) , \nabla T_tF(\gamma) \rangle_{T\U} d \pi
   		\\
   		&=
   		-\int_{\U(\R^n)} \langle V(\gamma,\cdot ) , \mathbf T_{t} \nabla F(\gamma) \rangle_{T\U} d \pi
   		=
   		-\int_{\U(\R^n)} \langle \mathbf T_{t} V(\gamma,\cdot ) , \nabla F(\gamma) \rangle_{T\U} d \pi
   		\\
   		&=
   		\int_{\U(\R^n)} F (\nabla^*\mathbf T_t V)  d \pi,
   		%\\
   		%&\le V(f).
   	\end{align*}
   	which immediately implies \eqref{z11}.
   \end{proof}

   Let us now introduce $\mathbf D^p(T\U(\R^n), \p)$, the space of vector fields with divergence in $L^p(\U(\R^n), \p)$, as the closure of 
        $\CylV(\U(\R^n)) \subset L^p(T\U(\R^n), \pi)$ with respect to the norm $\| V \|_{L^p} + \| \nabla^* V \|_{L^p}$.
        
        In the case $p=2$, we have the following inclusion
        \begin{equation}\label{eq: important inclusion2}
        \mathcal{D}(\EE_H) \subset 	\mathbf D^2(T\U(\R^n), \p) \, ,
        \end{equation}
       as a consequence of the inequality $\| \nabla^* V\|_{L^2} \le \EE_H(V,V)$ for 
     every $V\in \CylV(\U(\R^n))$. 
    
   	\begin{lem} \label{lem: DP}
   		Let $1<p<\infty$ and $1<p'<\infty$ such that $1/p + 1/p' = 1$.
   		If $F\in L^{p'}(\U(\R^n),\p)$ then
   		\begin{equation}
   			\V(F) = \sup\biggl\{ \int_{\U(\R^n)} (\nabla^* V) F d\p: V\in \mathbf D^p(T\U(\R^n), \p), \ | V |_{T\U} \le 1 \biggr\} \, .
   		\end{equation}
   	    %where $\mathbf D^p(\U(\R^n))$ is the closure of $\CylV(\U(\R^n)) \subset L^p(T\U(\R^n))$ with respect to the norm $\| V \|_{L^p} + \| \nabla^* V \|_{L^p}$.
   	\end{lem}
   
   \begin{proof}
   	Let $V\in \mathbf D^p(T\U(\R^n), \p)$ with $|V|_{T\U}\le 1$,
   	to conclude the proof we just need to build a sequence $(W_n)_{n\in \N} \subset\CylV(\U(\R^n))$ such that $|W_n|\le 1$ and $\| \nabla^*V - \nabla^*W_n\|_{L^p} \to 0$ as $n\to \infty$.
   	To that aim we first consider a sequence $V_n\in \CylV(\U(\R^n))$ such that $\| V -V_n \|_{L^p} + \| \nabla^*V - \nabla^* V_n\|_{L^p} \to 0$ as $n\to \infty$, which exists by definition. We now define $W_n$ by cutting $V_n$ of as we did in \eqref{eq: CPA} in the proof of Proposition \ref{lem: BV1}.
	  \end{proof}

   \begin{proof}[Proof of Theorem \ref{prop: IBV}]

   We first show the inequality $|\DDD_* F|(\U(\R^n)) \le \V(F)$ for $F \in  L^2(\U(\R^n), \p)$. We assume without loss of generality that $\V(F)<\infty$.
   Let $F \in L^2(\U(\R^n), \p)$. Set $F_n=T_{1/n}F \in  H^{1, 2}(\U(\R^n), \p)$. 
   By the symmetry of $\mathbf T_t$ in $L^2(T\U, \p)$ and Lemma \ref{lem: divheat}, we have that, for any $V \in \CylV(\U(\R^n))$ with $|V|_{T\U} \le 1$, it holds 
   \begin{align} \label{eq: ABF}
   	\int_{\U(\R^n)} F_n \nabla^* V d \pi 
   	=
   	\int_{\U(\R^n)}  T_{1/n}( \nabla^* V)F d \pi   	
%   	-\int_{\U(\R^n)} \langle V , \nabla F_n \rangle_{T\U} d \pi
%   	\\
%   	&=
%   	-\int_{\U(\R^n)} \langle V , \mathbf T_{1/n} \nabla F \rangle_{T\U} d \pi \notag
%   	\\
%   	&=
%   	-\int_{\U(\R^n)} \langle \mathbf T_{1/n} V , \nabla F \rangle_{T\U} d \pi \notag
%   	\\
   	=
   	\int_{\U(\R^n)} \nabla^*(\mathbf T_{1/n} V)  F d \pi. 
   \end{align}
   The inclusion \eqref{eq: important inclusion} and \eqref{eq: important inclusion2} imply that $\mathbf T_{1/n}V\in \mathbf D^2(T\U(\R^n), \p)$, while Theorem \ref{thm: GE1} ensures that $|\mathbf T_{1/n} V|_{T\U}  \le T_{1/n}|V|_{T\U} \le 1$.
   %\footnote{I completed the argument here.}Let $\mathbf H^{1,2}(T\U, \p)$ be the closure of $\CylV(\U(\R^n))$ with respect to the norm $\|V\|_{\mathbf H^{1,2}}$ defined by $\|V\|_{\mathbf H^{1,2}}^2:=\|V\|^2_{L^2(T\U)}+\|\nabla_H V\|^2_{L^2(T\U)}$ where $\nabla_H$ is the covariant derivative defined in \cite[Definition 2.4]{ADL01} and the closability follows from \cite[Theorem 3.3]{ADL01}. Since $\|\nabla^* V\|^2_{L^2(T\U)} \le \|\nabla_H V\|^2_{L^2(T\U)}$ by definition, $\mathbf H^{1,2}(T\U, \p) \subset \mathbf D^p(T\U, \p)$. Since $\mathbf T_{1/n} V \in \mathbf H^{1,2}(T\U, \p)$ by the general theory of functional analysis,  we can deduce that  $\mathbf T_{1/n} V \in \mathbf D^p(T\U, \p)$. Furthermore,  by Theorem \ref{thm: GE1}, $|\mathbf T_{1/n} V|_{T\U}  \le T_{1/n}|V|_{T\U} \le 1$ $\p$-a.e..  
   Therefore,  we can apply Lemma \ref{lem: DP} to \eqref{eq: ABF} to obtain $\|\nabla F_n\|_{L^1} \le \V(F)$.  
   Since $F_n \in H^{1,2}(\U(\R^n), \pi)$ and $\CylF(\U(\R^n))$ is dense in $H^{1,2}(\U(\R^n), \pi)$, we have 
   $
   |\DDD_*F_n|(\U(\R^n)) \le \|\nabla F_n\|_{L^1}
   $,
   by definition.  
   By the lower semi-continuity of $|\DDD_*F|(\U(\R^n))$ with respect to the $L^2$-convergence, it holds
   \begin{align*}
   	|\DDD_*F|(\U(\R^n)) 
   	\le  \liminf_{n \to \infty} |\DDD_*F_n|(\U(\R^n))  
   	\le  \liminf_{n \to \infty} \|\nabla F_n\|_{L^1(T\U, \p)}  
   	%\\
   	%&\le - \lim_{n \to \infty}\int_{\U} f_n \nabla^* V d \pi 
   	%\\
   	%&= - \int_{\U} f \nabla^* V d \pi 
   	\le \V(F).
   \end{align*}	
%   {\color{red}
%   This concludes the proof of $\V(F) = |\DDD_*F|(\U(\R^n))$.
%   }

   We now prove $\T(F) \le |\DDD_*F|(\U(\R^n))$.  Let $F_n\in \Cyl$ such that $F_n \to F$ in $L^1(\U(\R^n), \p)$ and $\| \nabla F_n \|_{L^1(T\U)} \to |\DDD_* F|(\U(\R^n))$. 
   Then, by the $1$-Bakry--\'Emery inequality \eqref{BEC-1} on cylinder functions,  
   \begin{align*}
   	\|\nabla T_tF\|_{L^1} \le \liminf_{n \to \infty}\|\nabla T_{t} F_n\|_{L^1}\le  \liminf_{n \to \infty} \|\nabla F_n\|_{L^1} = |\DDD_*F|(\U(\R^n)).
   \end{align*}
   Thus, $\T(F) \le |\DDD_*F|(\U(\R^n))$.

   \bigskip
   
   Finally we prove $\V(F) \le \T(F)$ for every $F\in L^p(\U(\R^n))$. For $F \in L^p(\U(\R^n), \p)$ and $V \in \CylV(\U(\R^n))$ with $|V|_{T\U} \le 1$, we have that
   $$
   \int_{\U(\R^n)} T_t F \nabla^* V d\p = \int_{\U(\R^n)} \langle \nabla T_t F, V \rangle d\p \le  \int_{\U(\R^n)} |\nabla T_t F|_{T\U} d\p.
   $$
   Since $T_t F \to F$ in $L^p(\U(\R^n), \p)$, we obtain that 
   $$\int_{\U(\R^n)} F \nabla^* V d\p \le \lim_{t \to 0} \int_{\U(\R^n)} |\nabla T_t F|_{T\U} d\p.$$
   Thus, we conclude $\V(F) \le \T(F)$.
   
   \bigskip

    The proof of Theorem \ref{prop: IBV} was given above. However, for the sake of completeness, we include a proof of the inequality $\V(F) \le |\DDD_*F|(\U(\R^n))$, which holds in the more general case of $F \in L^p(\U(\R^n),\p)$ with $1 < p \le \infty$.
   	
   	Let $F \in L^p(\U(\R^n), \pi)$ for some $p>1$ and $|\DDD_* F|(\U(\R^n))<\infty$. Let $F_n\in \Cyl$ such that $F_n \to F$ in $L^1(\U(\R^n))$ and $\| \nabla F_n \|_{L^1(T\U)} \to |\DDD_* F|(\U(\R^n))$. Let $F_{n, M}:=(F_n \vee -M)\wedge M$ and $F_{M}:=(F\vee -M)\wedge M$. Then, $F_{n,M} \to F_M$ in $L^1(\U(\R^n), \p)$ and $\| \nabla F_{n,M} \|_{L^1(T\U)}  \le \| \nabla F_n \|_{L^1(T\U)}$. Thus, $\limsup_{n\to\infty}\| \nabla F_{n,M} \|_{L^1(T\U)} \le  |\DDD_* F|(\U(\R^n))$. By the integration by parts formula \eqref{eq: IbP1}, it holds	
   	\begin{equation*}
   		\int_{\U(\R^n)} F_{n,M} \nabla^* V d \pi
   		=
   		-\int_{\U(\R^n)} \langle V , \nabla F_{n,M}\rangle_{T\U} d \pi
   		\le 
   		\| \nabla F_{n,M} \|_{L^1(T\U)}
   		\le
   		\| \nabla F_{n} \|_{L^1(T\U)}  \, ,
   	\end{equation*}
   	for any $V\in \CylV(\U(\R^n))$ with $|V|_{T\U} \le 1$.
   	By taking a (non-relabelled) subsequence from $\{F_{n,M}\}$ so that $F_{n,M} \to F_M$ $\pi$-a.e.,
   	%\footnote{We can take an a.e.\ converging subsequence from an $L^1$-converging sequence. See, e.g., Corollary 3 in https://terrytao.wordpress.com/2010/10/02/245a-notes-4-modes-of-convergence/ {\color{blue}EB: yes I agree, that is super standard, does not need any justification}}
   	and using the dominated convergence theorem (note that $|F_{n,M} \nabla^* V| \le M |\nabla^* V| \in L^1(\U(\R^n), \p)$ uniformly in $n$), we obtain that
   	\begin{equation*}
   		\int_{\U(\R^n)} F_{M} \nabla^* V d \pi = \lim_{n \to \infty}\int_{\U(\R^n)} F_{n,M} \nabla^* V d \pi \le \liminf_{n \to \infty} \| \nabla F_{n} \|_{L^1(T\U)} \le |\DDD_*F|(\U(\R^n)),
   	\end{equation*}
   	for any $V\in \CylV(\U(\R^n))$ with $|V|_{T\U}\le 1.$
   	Since $F_M \to F$ in $L^p(\U(\R^n), \p)$ as $M \to \infty$ by the hypothesis $F \in L^p(\U(\R^n), \p)$, we conclude  $\V(F) \le |\DDD F|(\U(\R^n))$. 
   \end{proof}
   
  \begin{rem}\normalfont
  The proof of all the inequalities except $|\DDD_*F|(\U(\R^n)) \le \V(F)$ remains true for every $1<p<\infty$.
   In order to prove the inequality $|\DDD_*F|(\U(\R^n)) \le \V(F)$ in full generality following the same strategy we need show that $\mathbf T_t V \in\mathbf D^p(T\U(\R^n), \p)$ for $1<p<\infty$ and $V \in \CylV(T\U)$. This should follow, for instance, from the $L^p$-boundedness of vector-valued Riesz transforms, and will be addressed in a future work.
   \end{rem}

\section{Sets of finite perimeter}
\label{sec:reduced boundary}

In this section we introduce and study the notion of {\it set with finite perimeter}. Let us begin with a definition

%
%In this section we introduce the notion of \textit{reduced boundary} for a Borel set $E\subset \U(\R^n)$. 
%We start by studying the boundary of the finite dimensional projections $E_{\eta, r}$ with $r>0$ and $\eta\in B_r^c$. Let us begin with the definition of {\it set of finite perimeter}.

\begin{defn}[Sets of finite perimeter] \normalfont
	Let $\Omega \subset \R^n$ be either a closed domain or the Euclidean space $\R^n$. A Borel set $E\subset \U(\Omega)$ is said to have finite perimeter if $\V_{\U(\Omega)}(\chi_E) < \infty$.
\end{defn}
We refer the reader to Definition \ref{defn: BV} for the introduction of the total variation $\V_{\U(\Omega)}(\cdot)$.

\subsection{Sets of finite perimeter in $\U(B_r)$}
We first develop the necessary theory in the configuration space $\U(B_r)$, in which every argument essentially comes down to finite-dimensional geometric analysis since only finitely many particles are allowed to belong to $B_r$.  

Let us recall the decomposition $\U(B_r) = \bigsqcup_{k\ge 0} \U^k(B_r)$, where $(\U^k(B_r), \d_{\U^k} , \pi_{B_r}^k)$ is the $k$-particle configuration space $\U^k(B_r)$ over $B_r$ equipped with the $L^2$-transportation distance $\d_{\U^k}$ and $\pi_{B_r}^k:=\p_{B_r}|_{\U^k(B_r)}$.  
We introduce the reduced boundary in $\U(B_r)$. 
\begin{defn}[Reduced boundary in $\U(B_r)$] \normalfont
	Fix $r>0$. Given $E\subset \U(B_r)$, set
	$E^k:= E\cap \U^{k}(B_r)$ and define
	\begin{equation*}
		\partial_{\U(B_r)}^* E := \bigsqcup_{k\ge 0} \partial_{\U^{k}(B_r)}^* E^k \, ,
	\end{equation*}
	\begin{equation*}
		\partial_{\U^{k}(B_r)}^* E^k 
		:= \left\lbrace  \gamma\in \U^{k}(B_r) \, : \, 
		\limsup_{s\to 0} \frac{\pi_{B_r}^k( {\sf B}^k_s(\gamma)\cap E^k)}{\pi_{B_r}^k({\sf B}^k_s(\gamma))}> 0,
		\quad
		\limsup_{s\to 0} \frac{\pi_{B_r}^k({\sf B}^k_s(\gamma)\setminus E^k)}{\pi_{B_r}^k({\sf B}^k_s(\gamma))}> 0
		\right\rbrace \, ,
	\end{equation*}	
    where ${\sf B}^k_s(\gamma)$ denotes the metric ball of radius $s>0$ centred at $\gamma\in \U^k(B_r)$ w.r.t.\ $\d_{\U^k}$.
\end{defn}

We can readily show that the $m$-codimensional Hausdorff measure $\rho^m_{\U^k(B_r)}$ w.r.t. $\d_{\U^k}$ coincides with the push-forward measure of the $m$-codimensional spherical Hausdorff measure $\rho^m_{B_r^{\times k}}$ on $B_r^{\times k}$ w.r.t. the quotient map ${\mathbf s}_k$: 
\begin{align} \label{eq: QHHQ}
\rho^m_{\U^k(B_r)} =  ({\mathbf s}_k)_{\#} \rho^m_{B_r^{\times k}}  = ({\mathbf s}_k)_{\#} \mathbf S_{B_r^{\times k}}^{nk-m},
\end{align}
where $\mathbf S_{B_r^{\times k}}^{nk-m}$ is the $m$-codimensional spherical Hausdorff measure on $B_r^{\times k}$ and ${\mathbf s}_k$ is the quotient map $B_r^{\times k} \to \U^k(B_r)$ as defined in Section \ref{sec: Pre}. Having this in mind, we prove the following Gau\ss--Green formula in $\U(B_r)$.

\begin{prop}[Gau\ss--Green formula in $\U(B_r)$]\label{lemma:finiteperimeterBr}
	Fix $r>0$. If $E \subset \U(B_r)$ is a set of finite perimeter then there exists a vector field $\sigma_E : \U(B_r) \to T\U(B_r)$ such that $|\sigma_E|_{T\U(B_r)}=1$ $\rho^1_{\U(B_r)}$-a.e. on $\partial^*_{\U(B_r)}E$, and 
	\begin{equation}
		\int_{E} (\nabla^* V) d \p_{B_r} 
		= \int_{\partial^*_{\U(B_r)}E} \langle V, \sigma_E \rangle d \rho^1_{\U(B_r)}
		\quad \text{for  $V\in \CylV(\U(B_r))$.}
	\end{equation}  
    Moreover $\V_{\U(B_r)}(\chi_E) = \rho^1_{\U(B_r)}(\partial^*_{\U(B_r)}E)$.
\end{prop}
\begin{proof}
	Exploiting the decomposition $\U(B_r) = \bigsqcup_{k\ge 0} \U^k(B_r)$, where each $\U^k(B_r)$ is a connected component, we reduce our analysis to the study of $E^k:=E\cap \U^k(B_r)$.

   	Set $\mathbf E^k:=\mathbf s_k^{-1}(E^k)$. Given 
   	\begin{equation*}
   		V = \sum_{k=1}^m \Phi(f_{1,k}^*, \ldots, f_{n_k,k}^*) v_k
   		\in \CylV(\U(B_r)) \, ,
   	\end{equation*}
   	we can define $\mathbf V\in C_*^\infty(B_r^{\times k} ;\R^{nk})$ as 
   	\begin{equation*}
   			\mathbf V(x_1, \ldots , x_k) = \sum_{k=1}^m \Phi(f_{1,k}(x_1) + \ldots + f_{n_k,k}(x_k), \ldots , f_{n_k,k}(x_1)+ \ldots + f_{n_k,k}(x_k)) v_k(x_1, \ldots, x_k)
   			\, .
   	\end{equation*}
   	Notice that $| \mathbf V|_{\R^{nk}}\le 1$ whenever $| V |_{T\U} \le 1$. It is now immediate that $\mathbf E^k$ is of finite perimeter on $B_r^{\times k}$. Thus, standard results og geometric measure theory on the Euclidean space $\R^{nk}$ (see e.g., \cite[Thm.\ 5.8.2]{Z89}), we obtain 
\begin{equation} \label{eq: FPGG}
		\int_{\mathbf E^k} (\nabla^* \mathbf V) d \SS_{B_r^{\times k}}^{nk}
		= \int_{\partial^*_{B_r^{\times k}}\mathbf E^k} \langle \mathbf V, \sigma_{\mathbf E^k} \rangle d \rho^1_{B_r^{\times k}}
		\quad \text{for $\mathbf V\in C_*^\infty(B_r^{\times k};\R^{nk})$} \, .
	\end{equation}  
Here $\sigma_{\mathbf E^k}$ is a vector field $\sigma_{\mathbf E^k} : B_r^{\times k} \to \R^{nk}$ such that $|\sigma_{\mathbf E^k}|_{\R^{nk}}=1$ $\rho^1_{B_r^{\times k}}$-a.e. on $\partial^*_{B_r^{\times k}}\mathbf E^k$.
By passing to the quotient by means of the map $\mathbf s_k$ in  both sides of \eqref{eq: FPGG} and using \eqref{eq: QHHQ}, we get the sought conclusion.
\end{proof}

\begin{rem}
	\normalfont
	An alternative proof of Proposition \ref{lemma:finiteperimeterBr} can be given by employing the theory of $\text{RCD}$ spaces (see \cite{Ambrosio18} and references therein). Indeed $(\U^k(B_r), \d^k , \pi_{B_r}^k)$ is 
	an $\text{RCD}(0,kn)$ space and $E^k$ is of finite perimeter. Hence we can apply \cite[Theorem 2.2]{BPS} to get the integration by parts formula, written in terms of the total variation measure $|D\chi_{E^k}|$. From \cite[Corollary 4.7]{ABS} we deduce the identity $|D\chi_{E^k}| = \rho^1_{\U(B_r)}|_{\partial^*_ {\U^k(B_r)}E^k}$.
\end{rem}

Let us now prove a measurability statement. The proof follows arguing exactly in the same way as in the proof of Proposition \ref{prop: meas}, thus, we omit it.
\begin{lem}\label{lemma:measurability1}
	Fix $r>0$. If $F: \U(\R^n) \to \R$ is a Borel function, then 
	\begin{equation*}
		\U(B_r^c)\ni \eta \to \int_{\U(B_r)} F_{\eta,r} d \rho^1_{\U(B_r)} \quad \, \text{is $\pi_{B_r^c}$-measurable}\, .
	\end{equation*}	
\end{lem}

\subsection{Sets of finite perimeter on $\U(\R^n)$}
We now study sets of finite perimeter on the configuration space $\U(\R^n)$ by employing the already developed theory for the space $\U(B_r)$. The main idea is to reduce a set $E \subset \U(\R^n)$ to its sections $E_{\eta, r} \subset \U(B_r)$ and apply the results for sets of finite perimeter in $\U(B_r)$, combined with the disintegration argument. We finally let $r \to \infty$ to recover the information on the perimeter of the original set $E$.

Let us begin by introducing the definition of the reduced boundary in $\U(\R^n)$.

\begin{defn}[Reduced boundary in $\U(\R^n)$] \label{defn: RB}
	Let $E\subset \U(\R^n)$ be a Borel set. For every $r>0$ we set
	\begin{equation}
		\partial_r^*E := \{ \gamma\in \U(\R^n) \, : \, \gamma|_{B_r} \in \partial_{\U(B_r)}^* E_{\gamma|_{B_r^c},r}    \} \, .
	\end{equation}
   The reduced boundary of $E$ is defined as 
   \begin{equation} \label{defn: RB1}
   	\partial^* E := \liminf_{i\to \infty,\, i\in \N} \partial_i^* E = \bigcup_{i>0}\, \, \bigcap_{j>i,\, j\in \N} \partial_j^* E \, .
   \end{equation}
\end{defn}

\begin{rem} \normalfont
	We defined $\partial^*E$ by taking the liminf along the sequence $\{\partial^*_{i}E\}_{i\in \N}$. This choice is completely arbitrary and, as we will see in the sequel (cf. Theorem \ref{thm: fpr}), if we change the defining sequence, then the reduced boundary can change, but only up to an $\| E \|$-negligible set, where $\| E\|$ is the perimeter measure that will be defined later. Thus, the reduced boundary is well-defined up to $\| E \|$-negligible sets. 
\end{rem}

Notice that, for every $\eta\in \U(B_r^c)$ it holds
\begin{equation}\label{z3}
	(\partial_r^* E)_{\eta,r} = \partial_{\U(B_r)}^* E_{\eta,r} \, .
\end{equation}

\begin{lem}
	If $E$ is a Borel subset of $\U(\R^n)$, then $\partial_r^* E$ and $\partial^* E$ are Borel.
\end{lem}
\begin{proof}
	%We follow closely the proof of \cite[Proposition 2.10]{H10}. 
	Since $\partial^* E = \liminf_{r \to \infty} \partial^*_r E$, it suffices to show the Borel measurability of $\partial^*_r E$ for every $r>0$.
	
	{\bf Step 1:} We prove the following statement: for every $k\in \N$ and $s > 0$ the function
	\begin{equation}
		\{\gamma \in \U(\R^n) \, : \, \gamma(B_r) = k \}     \ni \gamma \mapsto \frac{\pi_{B_r}^k({\sf B}^k_s(\gamma|_{B_r})\cap E^k_{\gamma|_{B_r^c}, r})}{\pi_{B_r}^k({\sf B}^k_s(\gamma|_{B_r}))}
	\end{equation}
    is Borel. 
    
    \medskip
Since the Borel measurability of the map $\gamma \mapsto \pi_{B_r}^k({\sf B}^k_s(\gamma|_{B_r}))$ is easy, we only give a proof of the Borel measurability of the map $\gamma \mapsto \pi_{B_r}^k({\sf B}^k_s(\gamma|_{B_r})\cap E^k_{\gamma|_{B_r^c}, r})$.     
    
    Let us identify $\{\gamma \in \U(\R^n) \, : \, \gamma(B_r) = k \} \simeq
    \U^k(B_r) \times \U(B_r^c)$. It allows us to introduce the product topology $\tau_p$ on $\{\gamma \in \U(\R^n) \, : \, \gamma(B_r) = k \}$, that is coarser than the vague topology $\tau_v$ as a consequence of the following observation: since $B_r^c$ is open, the vague topology $\tau_v$ on $\U(B_r^c)$ coincides with the relative topology induced by $\U(\R^n)$. Thus, it suffices to see that the vague topology on $\U(B_r)$ is coarser than the relative topology induced by $\U(\R^n)$. For this purpose, we only need to show that, for any $\phi \in C_c(B_r)$ (note that $\phi$ does not necessarily vanish at the boundary of $B_r$), there exists an extension $\tilde{\phi} \in C_c(\R^n)$ so that $\tilde{\phi}= \phi$ on $B_r$.  Given $\phi \in C_c(B_r)$,  we take $\Phi\in C(\R^n)$ which is the extension of $\phi$ to $\R^n$ given by the Tietze extension theorem. Let us now pick $\kappa \in C_c(\R^n)$ such that $\kappa=1$ on $B_r$ and $\kappa=0$ on $B_{2r}^c$. Then, it holds $\tilde{\phi}:=\kappa\Phi\in C_c(\R^n)$ and $\tilde{\phi}=\phi$ in $B_r$, which concludes the sought statement.

By the inclusion $\tau_p \subset \tau_v$ of the topologies, we have the inclusion of the corresponding Borel $\sigma$-algebras $\mathscr B(\tau_p) \subset \mathscr B(\tau_v)$. Since the map
    \begin{equation}
    	\U^k(B_r)\times \U^k(B_r) \times \U(B_r^c)
    	\ni (\gamma_1, \gamma_2, \eta)
    	\to
    	\chi_{E}(\gamma_1 + \eta) \chi_{{\sf B}^k_s(\gamma_2)}(\gamma_1) \, ,
    \end{equation}
    is $\mathscr B(\tau_p)$-measurable, it is also $\mathscr B(\tau_v)$-measurable. Hence,  Fubini's theorem gives that
    \begin{equation}
    	 \U^k(B_r) \times \U(B_r^c)
    	\ni (\gamma_2, \eta)
    	\to
    	\int_{\U^k(B_r)}\chi_{E}(\gamma_1 + \eta) \chi_{{\sf B}^k_s(\gamma_2)}(\gamma_1) d \p^k_{{\sf B}_r}(\gamma_1)
    	=
    	\pi_{B_r}^k( {\sf B}^k_s(\gamma_2|_{B_r})\cap E^k_{\eta, r}) \, ,
    \end{equation}
    is $\mathscr B(\tau_v)$-measurable as well.

    \bigskip

    {\bf Step 2:} Fix $k\in \N$ and set
    \begin{align*}
    	&A^{k, r}_1  := \left\lbrace\gamma\in \U(\R^n) \, : \, \limsup_{s\to 0} \frac{\pi_{B_r}^k({\sf B}^k_s(\gamma|_{B_r})\cap E^k_{\gamma|_{B_r^c}, r})}{\pi_{B_r}^k({\sf B}_s(\gamma|_{B_r}))} > 0\right\rbrace \, ,
    	\\&
    	A^{k, r}_2 :=
    	\left\lbrace \gamma\in \U(\R^n) \, : \, \limsup_{j\to \infty}  \frac{\pi_{B_r}^k({\sf B}^k_{2^{-j}}(\gamma|_{B_r})\cap E^k_{\gamma|_{B_r^c}, r})}{\pi_{B_r}^k({\sf B}^k_{2^{-j}}(\gamma|_{B_r}))} > 0 \right\rbrace \, .
    \end{align*}
    Then $A_1^{k, r} = A_2^{k, r}$.
    
    \medskip
    
    Observe that $A^{k, r}_2\subset A^{k, r}_1$. The converse inequality follows from the following observation. If $2^{-j} \le s \le 2^{-j + 1}$ then
    \begin{align*}
    	\frac{\pi_{B_r}^k({\sf B}^k_{s}(\gamma|_{B_r})\cap E^k_{\gamma|_{B_r^c}, r})}{\pi_{B_r}^k({\sf B}^k_{s}(\gamma|_{B_r}))}
    	& \ge 
    	\frac{\pi_{B_r}^k({\sf B}^k_{2^{-j}}(\gamma|_{B_r})\cap E^k_{\gamma|_{B_r^c}, r})}{\pi_{B_r}^k({\sf B}^k_{2^{-j}}(\gamma|_{B_r}))} \, \frac{\pi_{B_r}^k({\sf B}^k_{2^{-j}}(\gamma|_{B_r}))}{\pi_{B_r}^k({\sf B}^k_{s}(\gamma|_{B_r}))}
    	\\& \ge 
    	C(k,n) \frac{\pi_{B_r}^k({\sf B}^k_{2^{-j}}(\gamma|_{B_r})\cap E^k_{\gamma|_{B_r^c}, r})}{\pi_{B_r}^k({\sf B}^k_{2^{-j}}(\gamma|_{B_r}))} \, ,
    \end{align*}
    where we used the estimate $C(n,k)^{-1}e^{-\mathbf{L}^n(B_r)} s^{nk} \le \pi_{B_r}^k({\sf B}^k_{s}(\gamma)) \le C(n,k)e^{-\mathbf{L}^n(B_r)} s^{nk}$ for any $s< r/5$, $\gamma\in \U(B_r)$ and some constant $C(n,k)\ge 1$ depending only on $n$ and $k$. Indeed, the latter estimate can be obtained by the following observation: letting $\gamma=\{x_1, \ldots, x_k\}$, we have
    \begin{equation*}
    	B_r^{\times k} \cap \mathbf s_k^{-1}({\sf B}^k_{s}(\gamma))
    	= B_r^{\times k} \cap \bigcup_{\sigma_k \in \mathfrak S_k}
    	B_s(\mathbf x_{\sigma_k}) \, ,
    \end{equation*} 
    hence 
    \begin{equation*}
    	\pi_{B_r}^k({\sf B}_s^k(\gamma))
    	= \frac{e^{-\mathbf{L}^n(B_r)}}{k!} \mathbf{L}^{kn}(B_r^{\times k}  \cap \mathbf s_k^{-1}({\sf B}^k_{s}(\gamma)))
    	\le e^{-\mathbf{L}^n(B_r)} C(n,k) s^{nk} \, ,
    \end{equation*}
    recall that $\mathbf L^n$ denotes the $n$-dimensional Lebesgue measure. 
    The opposite inequality follows from 
           \begin{align*}
    	\pi_{B_r}^k({\sf B}_s^k(\gamma))
    	&= \frac{e^{-\mathbf{L}^n(B_r)}}{k!} \mathbf{L}^{kn}(B_r^{\times k} \cap \mathbf s_k^{-1}({\sf B}^k_{s}(\gamma)))
	\\
	& \ge  \frac{e^{-\mathbf{L}^n(B_r)}}{k!} \mathbf{L}^{kn}(B_r^{\times k}  \cap B_s(\mathbf x_{\sigma_k}))
    	\ge e^{-\mathbf{L}^n(B_r)} C(n,k) s^{nk} \, .
    \end{align*}

    \bigskip

    {\bf Step 3:} We conclude the proof. Thanks to Step 1 and Step 2 we know that $A_1^{k, r}$ is Borel for every $k\in \N$ and $r>0$. The same arguments as in Step 1 and Step 2 apply to the Borel measurability for the following set:
    \begin{equation}
    	\left\lbrace\gamma\in \U(\R^n) \, : \, \limsup_{s\to 0} \frac{\pi_{B_r}^k( {\sf B}^k_s(\gamma|_{B_r})\setminus E^k_{r,\gamma|_{B_r^c}})}{\pi_{B_r}^k({\sf B}^k_s(\gamma|_{B_r}))} > 0\right\rbrace \, ,
    \end{equation}
    hence, $\partial^*_r E$ is a Borel set.
\end{proof}

\subsection{Perimeter measures} In this subsection, based on the variational approach, we introduce the perimeter measure $\|E\|$ for a set $E \subset \U(\R^n)$ satisfying $\mathcal V(\chi_E)<\infty$. In order to construct $\|E\|$, we first introduce a localised perimeter measure $\|E\|_r$ on $\U(\R^n)$, and show the monotonicity of $\|E\|_r$ as $r \to \infty$.

\begin{defn}\normalfont \label{defn: PME0}
 For every Borel set $E\subset \U(\R^n)$ with $\V_r(\chi_E) < \infty$, we define 
 \begin{equation} \label{defn: PME}
		\| E \|_r := \rho^1_{\U(B_r)}|_{(\partial_r^*E)_{\eta,r}} \otimes \pi_{B_r^c}(\eta)
		\quad
		\text{on}\, \,  \U(\R^n) \, ,
	\end{equation}
	which is equivalently defines as follows: for every bounded Borel measurable function $F$ on $\U(\R^n)$, 
 \begin{align} \label{eq: DOE}
 \int_{\U(\R^n)} F d\| E \|_r := \int_{\U(B_r^c)} \left( \int_{\U(B_r)} F_{\eta, r} d \rho^1_{\U(B_r)}|_{\partial^*_{\U(B_r)} E_{\eta,r}} \right) d \p_{B_r^c}(\eta).
 \end{align}
 \end{defn}
\begin{lem} \label{lem: BV2-1}
	Let $r>0$. For every Borel set $E \subset \U(\R^n)$ with $\mathcal V(\chi_E)<\infty$, $\|E\|_r$ is a well-defined finite Borel measure. 
\end{lem}

\begin{proof}
Let us first show that $\|E\|_r$ is well-defined. The map $\gamma \mapsto F_{\eta, r}(\gamma)$ is $\rho^1_{\U(B_r)}|_{\partial^* E_{\eta,r}}$-measurable by Lemma \ref{lem: sec}. On account of the definition \eqref{eq: DOE}, we only need to show that the map
\begin{equation}\label{z2}
 \U(B_r^c)  \ni	\eta \to \int_{\U(B_r)} F_{\eta, r} d \rho^1_{\U(B_r)}|_{\partial^*_{\U(B_r)} E_{\eta,r}} \, ,
\end{equation}
is $\pi_{B_r^c}$-measurable for any Borel function $F: \U(\R^n) \to \R$.
To show it,  we use \eqref{z3} and rewrite
\begin{equation*}
	%\int_{\U(B_r)} h_{\eta, r} d \nu_{\eta,r}
	%=
	\int_{\partial_{\U(B_r)}^* E_{\eta,r}} F_{\eta, r} \,  d \rho^1_{\U(B_r)}
	=
	\int_{(\partial_r^* E)_{\eta,r}} F_{\eta, r} \,  d \rho^1_{\U(B_r)}
	=
	\int_{\U(B_r)} (\chi_{\partial_r^* E} F)_{\eta, r} \, d \rho^1_{\U(B_r)} \, .
\end{equation*}
Now, the claimed conclusion follows from Lemma \ref{lemma:measurability1} by observing that $\chi_{\partial_r^* E} F$ is a Borel function.

The finiteness of the measure $\|E\|_r$ is immediate by Proposition \ref{lemma:finiteperimeterBr} and Proposition \ref{lem: BV1}, indeed
\begin{equation*}
 \| E \|_r(\U(\R^n)) = \int_{\U(B_r^c)} \V_{\U(B_r)}((\chi_E)_{\eta,r}) d \pi_{B_r^c}(\eta) = \V_r(\chi_E) \le \V(\chi_E) < \infty. \qedhere
\end{equation*}
\end{proof}

%Now we prove the Gau\ss--Green formula w.r.t.\ $\|E\|_r$. 
\begin{lem} \label{lem: BV2}
Let $r>0$. For every Borel set $E \subset \U(\R^n)$ with $\mathcal V_r(\chi_E)<\infty$,  there exists a vector field $\sigma_{E, r} :\U(\R^n) \to T \U(\R^n)$ such that
	\begin{itemize}
		\item[(i)] $\sigma_{E, r}(\gamma) \in T_\gamma \U(\R^n)$ satisfies $\sigma_{E, r}(\gamma,x)=0$ for $x\in B_r^c$;
		\item[(ii)] $|\sigma_{E, r}|_{T\U}=1$, $\|E\|_r$-a.e.;
		\item[(iii)] for every $V\in \CylV^r_*(\U(\R^n))$, 
		%the map $\gamma \mapsto \langle V, \sigma_r\rangle_{T_\gamma\U}$ is $\|E\|_r$-measurable in the sense that 
		\begin{align} \label{eq: IP2}
			\int_{E} (\nabla^* V)  d\p = \int_{\U(\R^n)} \langle V, \sigma_{E, r}\rangle_{T\U} d\|E\|_r \, .
		\end{align} 
		\item[(iv)] $\V_r(\chi_E) = \|E\|_r(\U(\R^n))$, and for every non-negative function $F\in \CylF(\U(\R^n))$ it holds
		\begin{equation}\label{z9}
			\int_{\U(\R^n)} F d \|E\|_r
			=
			\sup\left\lbrace \int_{E} (\nabla^* F V)  d\p \, : \, V\in \CylV^r_*(\U(\R^n)),\, |V|_{T\U}\le 1 \right\rbrace \, .
		\end{equation}
	\end{itemize}
\end{lem}

\begin{proof}
By Proposition \ref{lem: BV1}, there exists a measurable set $\Omega_r \subset \U(B_r^c)$ so that  $\p_{B_r^c}(\Omega_r)=1$ and $\V_{\U(B_r)}(\chi_{E_{\eta, r}})<\infty$ for every $\eta \in \Omega_r$.
By Proposition \ref{lemma:finiteperimeterBr}, for every $\eta \in \Omega_r$, there exists a unique $T\U(B_r)$-valued Borel measurable map $\sigma_{\eta, r}$ on $\U(B_r)$ so that $|\sigma_{\eta, r}|_{T\U(B_r)}=1$ $\rho^1_{\U(B_r)}|_{\partial^* E_{\eta,r}}$-a.e., and
\begin{align} \label{eq: IP2-1}
	\int_{E_{\eta, r}} (\nabla^* V_{\eta, r})  d\p_{B_r} 
	=
	\int_{\partial^*_{\U(B_r)} E_{\eta,r}} \langle V_{\eta, r}, \sigma_{\eta, r}\rangle_{T\U(B_r)} d\rho^1_{\U(B_r)}, \quad V \in \CylV^r_*(\U(\R^n)) \, ,
\end{align} 
where we used $V_{\eta, r} \in \CylV_*(\U(B_r))$ whenever $V \in \CylV^r_*(\U(\R^n))$. 
By taking the integral with respect to $\p_{B_r^c}$, and arguing as in~\eqref{prop: 0} we obtain
\begin{align} \label{eq: IP2-2}
	\int_{E} (\nabla^* V)  d\p
	& =
	\int_{\U(B_r^c)} \int_{E_{\eta,r}} (\nabla_{r}^* V_{\eta, r}) d\p_{B_r} d\p_{B^c_r}(\eta) 
	\\& 
	=
	\int_{\U(B_r^c)} \int_{\partial^*_{\U(B_r)} E_{\eta,r}} \langle V_{\eta, r}, \sigma_{\eta, r}\rangle_{T\U(B_r)} d\rho^1_{\U(B_r)} d\p_{B_r^c}(\eta) \, . \notag
\end{align} 
Note that the map $\eta  \mapsto \int_{\partial^*_{\U(B_r)} E_{\eta,r}} \langle V_{\eta, r}, \sigma_{\eta, r}\rangle_{T\U(B_r)} d\rho^1_{\U(B_r)}$ is $\p_{B_r^c}$-measurable since, in view of \eqref{eq: IP2-1}, it is equal to a $\pi_{B_r^c}$-measurable function, and therefore, the argument \eqref{eq: IP2-2} is justified. 
For $\gamma\in \U(\R^n)$ we define 
\begin{equation} \label{eq: DUV}
	\sigma_{E, r}(\gamma):=
	\begin{cases}
	\sigma_{\gamma|_{B_r^c}, r}(\gamma|_{B_r}) \, \, &\text{if $\gamma|_{B_r^c} \in \Omega_r$},
	\\
	\sigma_r(\gamma)=0 \quad & \text{otherwise}\, .
	\end{cases}
\end{equation}
Let us now observe that, for any $V\in \CylV(\U(\R^n))$, we have
   \begin{equation}\label{z4}
		(\langle V, \sigma_{E, r} \rangle_{T\U(\R^n)})_{\eta,r} 
		= 
		\langle V_{\eta,r}, \sigma_{\eta,r} \rangle_{T\U(B_r)} \, .
	\end{equation}	 
By combining the definition \eqref{defn: PME} of $\|E\|_r$ with \eqref{eq: IP2-2}, \eqref{eq: DUV} and \eqref{z4},  we deduce the assertion (iii).

\medskip

The assertion (i) follows from the definition \eqref{eq: DUV}, and the assertion (ii) follows from 
$$(|\sigma_{E, r}|_{T\U(\R^n)})_{\eta,r} = | \sigma_{\eta,r}| = 1, \quad \text{$ \rho^1_{\U(B_r)}|_{\partial^* E_{\eta,r}}$-a.e..}$$

\medskip

We now prove (iv). We first prove the equality $\V_r(\chi_E)  = \|E\|_r(\U(\R^n))$.  From (iii) and (ii) we deduce 
\begin{align}
	\V_r(\chi_E) 
	&=\sup\biggl\{ \int_{\U(\R^n)} (\nabla^* V) f d\p: V\in \CylV^r(\U(\R^n)),\ |V|_{T\U(\R^n)} \le 1 \biggr\}  \notag
	\\
	&=\sup\biggl\{\int_{\U(\R^n)} \langle V, \sigma_{E, r}\rangle_{T\U} d\|E\|_r: V\in \CylV^r(\U(\R^n)),\ |V|_{T\U(\R^n)} \le 1 \biggr\}  \notag
	%\\
	%&\le  \int_{\U(\R^n)}|\sigma_r|_{T\U} d \| E\|_r \notag
	\\
	& \le  \|E\|_r(\U(\R^n)) \, . \notag
\end{align}
Furthermore, Proposition \ref{lem: BV1} and Lemma \ref{lemma:finiteperimeterBr} imply
\begin{equation}
	\V_r(\chi_E) 
	\ge \int_{\U(B_r^c)} \V_{\U(B_r)}((\chi_{E})_{\eta, r}) d\p_{B_r^c}(\eta)
	=
	\int_{\U(B_r^c)} \rho^1_{\U(B_r)}(\partial^*_{\U(B_r)} E_{\eta,r}) d\pi_{B_r^c}(\eta)
	= \| E \|_r(\U(\R^n)) \, .
\end{equation}
Thus, the proof of the equality $\V_r(\chi_E)  = \|E\|_r(\U(\R^n))$ is complete. 

Let us finally address \eqref{z9}. From the equality $\V_r(\chi_E)  = \|E\|_r(\U(\R^n))$,  we deduce the existence of a sequence $V_k\in \CylV^r_*(\U(\R^n))$ such that $|V_k|_{T\U}\le 1$, and
\begin{equation*}
	\lim_{k\to \infty} \int_{\U(\R^n)} \langle V_k, \sigma_{E, r} \rangle_{T\U} d \|E\|_r = \int_{\U(\R^n)} d \|E\|_r \, ,
\end{equation*}
hence, 
\begin{align*}
	\lim_{k\to \infty} \int_{\U(\R^n)}| V_k -\sigma_{E, r}|^2_{T\U} d \|E\|_r
	& =
	\lim_{k\to \infty} \int_{\U(\R^n)}( |V_k|^2_{T\U} + |\sigma_{E, r}|^2_{T\U} -2\langle V_k, \sigma_{E, r} \rangle_{T\U}) d \|E\|_r
	\\& \le 
	\lim_{k\to \infty} 2\int_{\U(\R^n)} (1-\langle V_k, \sigma_{E, r} \rangle_{T\U}) d \|E\|_r
	 =0 \, .
\end{align*}
Therefore, for every $F \in \Cyl$
\begin{align*}
	\lim_{k\to \infty} & \int_{\U(\R^n)} F \langle V_k, \sigma_{E, r} \rangle_{T\U} d \|E\|_r
	%	\\&= \int_{\U(\R^n)} h\,  d \nu_r + \lim_{k\to \infty} \int_{\U(\R^n)} h( \langle V_k, \sigma_r \rangle_{T\U} -|\sigma_r|^2_{T\U}) d \nu_r
	%	\\& \ge \int_{\U(\R^n)} h\,  d \nu_r - \limsup_{k\to \infty} \int_{\U(\R^n)} h| V_k - \sigma_r|_{T\U} d \nu_r
	= \int_{\U(\R^n)} F\,  d \|E\|_r \, ,
\end{align*}
in particular, by making use of \eqref{eq: IP2} with $V=FV_k$,  it holds that
\begin{equation}
	\int_{\U(\R^n)} F\,  d \|E\|_r
	\le 
	\sup\left\lbrace \int_{E} (\nabla^* F V)  d\p \, : \, V\in \CylV^r_*(\U(\R^n)),\, |V|_{T\U}\le 1 \right\rbrace \, .
\end{equation}
The converse inequality follows form $|\sigma_{E,r}|_{T\U}=1$ $\| E\|_r$-a.e. and the fact that $F$ is non-negative:
\begin{equation}
	\int_{E} (\nabla^* F V)  d\p 
	= \int_{\U(\R^n)} F \langle V_k, \sigma_{E, r} \rangle_{T\U} d \|E\|_r
	\le \int_{\U(\R^n)} |F| d \|E\|_r
	= \int_{\U(\R^n)} F d \|E\|_r \, . \qedhere
\end{equation}
\end{proof}

\begin{cor}\label{cor:mon}
	If $\V(\chi_E)<\infty$, then $r \mapsto \| E \|_r(A)$ is monotone non-decreasing for every Borel measurable set $A$.
\end{cor}
\begin{proof}
	In view of the density of cylinder functions on $L^2(\U(\R^n),\p)$ it is enough to check that
	\begin{equation*}
		r \to \int_{\U(\R^n)} F d \| E \|_r
		\quad
		\text{is non-decreasing} \, ,
	\end{equation*}
    for every non-negative $F\in \CylF(\U(\R^n))$, which easily follows from \eqref{z9} and the inclusion $\CylV^s_*(\U(\R^n)) \subset \CylV^r_*(\U(\R^n))$ for $s\le r$.
\end{proof}

By the monotonicity of $r \mapsto \|E\|_r$ in Corollary \ref{cor:mon}, we may define the limit measure as follows:

\begin{defn}[Perimeter measure] 
	\label{defn: PM}\normalfont
	Given $E\subset \U(\R^n)$ with $\mathcal V(\chi_E)<\infty$, we define the perimeter measure as
	\begin{equation}
		\| E \| (A) := \lim_{r \to \infty} \| E \|_r(A)
		\quad \text{for every Borel set $A$}\, .
	\end{equation}	
\end{defn}

We finally obtain the Gau\ss--Green formula for the perimeter measure $\|E\|$. For a Borel set $E \subset \U(\R^n)$ with $\mathcal V(\chi_E)<\infty$, let $L^2(T\U, \|E\|)$ be the completion of $\CylV(\U)$ with respect to $\|\cdot\|_{L^2(T\U, \|E\|)}$ analogously in~\eqref{d:VFN}.
\begin{thm}[Gau\ss--Green formula for $\|E\|$] \label{thm: GGP}
For a Borel set $E \subset \U(\R^n)$ with $\mathcal V(\chi_E)<\infty$,  there exists a unique element $\sigma_E \in L^2(T\U, \|E\|)$ such that $|\sigma_E|_{T\U}=1$ $\|E\|$-a.e.\ and 
\begin{align} \label{eq: IP3}
			\int_{E} \nabla^* V  d\p = \int_{\U(\R^n)} \langle V, \sigma_{E}\rangle_{T\U} d\|E\| \quad  V \in \CylV(\U(\R^n)).
\end{align} 
\end{thm}
\proof
%{\color{red}We assume without loss of generality that $\|E\|(\U(\R^n)) =1$.}
Note that, for any $V \in \CylV(\U(\R^n))$, there exists $r>0$ so that  $V \in \CylV^r_*(\U(\R^n))$. 
	Thus, by (iii) in Lemma \ref{lem: BV2}, for any $V \in \CylV(\U(\R^n))$, there exists $r>0$ and $\sigma_{E, r}: \U(\R^n) \to T\U$ so that $|\sigma_{E, r}|=1$ $\|E\|$-a.e., and 
	\begin{align} %\label{eq: CF}
		\int_{E} \nabla^* V  d\p 
		%\int_{-\infty}^\infty \int_{\{F>t\}} (\nabla^* V) d \pi d t \notag
		%\\& 
		&= \int_{\U(\R^n)} \langle V, \sigma_{E, r}\rangle_{T\U} d\| E\|_r  \notag
		\\&
		\le \| E \|(\U(\R^n))^{1/2} \| V \|_{L^2(T\U, \|E\|_r)} \notag
		\\&
		\le \| E \|(\U(\R^n))^{1/2} \| V \|_{L^2(T\U, \|E\|)}. \notag
	\end{align} 
	The last inequality followed from the monotonicity in Corollary \ref{cor:mon}.
    %for any $V\in \CylV(\U)$. 
    In particular, the linear operator $L$ defined as
	\begin{equation}
		L : L^2(T \U(\R^n), \|E\|) \to \R \, , \quad
		L^2(T \U(\R^n), \|E\|) \ni V \mapsto L(V):= \int_E \nabla^* V  d\p \, ,
	\end{equation}
    is a well-defined continuous operator on the Hilbert space $L^2(T \U(\R^n), \|E\|)$ and satisfies $\| L \| \le \| E \|(\U(\R^n))^{1/2}$.
    Therefore, the Riesz representation theorem in the Hilbert space $L^2(T \U(\R^n), \|E\|)$ gives the existence of $\sigma_E \in L^2(T \U(\R^n), \|E\|)$ so that 
    \begin{equation}%\label{z12}
    	\|\sigma_E\|_{L^2(T\U, \|E\|)} \le \| E \|(\U(\R^n))^{1/2} \, ,
    	\quad \quad
    	\int_{E} \nabla^* V  d\p = \int_{\U(\R^n)} \langle V, \sigma \rangle d\|E\| 
    	\quad V\in \CylV(\U(\R^n)) \, . \notag
    \end{equation}
   It suffices to show that $|\sigma|_{T\U} =1$ $\|E\|$-a.e. By (iv) in Lemma \ref{lem: BV2} and Corollary \ref{cor:mon},  we deduce that 
    \begin{align*}
    %{\color{red}1 =} 
    \|E\|(\U(\R^n)) &=  \lim_{r \to \infty}\|E\|_r(\U(\R^n)) =  \lim_{r \to \infty}\V_r(\chi_E)
    \\
    & = \lim_{r \to \infty}\sup_{V\in \CylV^r_r(\U(\R^n))\, ,|V|_{T\U}\le 1} \int_{E} \nabla^* V  d\p 
   \\& 
   \le  \int_{\U(\R^n)} |\sigma|_{T\U} d\|E\|
    \le \| E \|(\U(\R^n))^{1/2}\| \sigma \|_{L^2(T\U, \|E\|)} 
    \\& \le\| E \|(\U(\R^n))
     \, ,
    \end{align*}
    which yields $|\sigma|_{T\U} =1$ $\|E\|$-a.e.\ as a consequence of the characterisation of the equality for the H\"older inequality. 
%\black{The proof is a combination of Lemma \ref{lem: BV2}, Definition \ref{defn: PM} and the fact that, for any $V \in \CylV(\U(\R^n))$, there exists $r>0$ so that $V \in \CylV^r_*(\U(\R^n))$.}\footnote{We do not know what $\sigma_E$ is. This is something like  the projective limit of $\sigma_{E, r}$ as $r \to \infty$.}.
\qed
%\begin{rem} \normalfont
%\blue{KS. Explain the definition of $\|E\|:=\lim_{r \to \infty}\|E\|_r$ does not depend on the choice of the exhaustion $B_r$. EB: I propose to remove that comment} 
%\end{rem}

% 
%We prove the lower semi-continuity of $\|E\|$ with respect to the $L^1$-convergence. 
%\begin{prop}[Lower semi-continuity of $\|E\|$] \label{prop: LSC}
%Let $E_n, E \subset \U(\R^n)$ with $\mathcal V(\chi_{E_n})<\infty$ and $\mathcal V(\chi_{E})<\infty$ respectively. If $\chi_{E_n}$ converges to $\chi_{E}$ in $L^1(\U(\R^n), \p)$, then, 
%\begin{align*}
%\|E\|(A) \le \liminf_{n \to \infty} \|E_n\|(A), \quad \text{for any open set $A \subset \U(\R^n)$.}
%\end{align*}
%\end{prop}
%\proof 
%Since $\|E\|(A)$ is the monotone increasing limit of $\|E\|_r(A)$, it suffices to show the lower semi-continuity $\|E\|_r(A) \le \liminf_{n \to \infty}\|E_n\|_r(A)$ for any $r>0$. Fix $r>0$. 
%%By (iv) in Lemma \ref{lem: BV2}, for any non-negative function $F\in \CylF(\U(\R^n))$ it holds
%%		\begin{equation*}
%%			\int_{\U(\R^n)} F d \|E\|_r
%%			=
%%			\sup\left\lbrace \int_{E} (\nabla^* F V)  d\p \, : \, V\in \CylV^r_*(\U(\R^n)),\, |V|_{T\U}\le 1 \right\rbrace \, .
%%		\end{equation*}
%		\purple{to be continued\ldots}
%\qed

\subsection{Perimeters and one-codimensional Poisson measures} In this subsection, we prove one of the main results in this paper. Namely, the perimeter measure $\|E\|$ based on the variational approach (Definition \ref{defn: PM}) coincides with the $1$-codimensinal Poisson measure $\rho^1$ (Definition \ref{defn: MCH}) restricted to the reduced boundary $\partial^*E$ of $E$ (Definition \ref{defn: RB}). 

\begin{thm}\label{thm: fpr}
	Let $E\subset \U(\R^n)$ be a set with $\mathcal V(\chi_E)<\infty$. Then, 
	$$\| E \| = \rho^1|_{\partial^*E}.$$
\end{thm}

Before giving the proof, we prove a lemma. 
\begin{lem} \label{lem: MNB}
	Let $E \subset \U(\R^n)$ be a set with $\mathcal V(\chi_E)<\infty$. Then, for any $r>0$, $\e>0$, it holds
	\begin{equation}
		(\partial_{r}^*E)_{\eta, r} \subset (\partial_{r+\e}^*E)_{\eta, r} \, 
		\quad\text{up to $\rho_{\U(B_{r})}^1$-negligible sets for $\pi_{B_r^c}$-a.e.\ $\eta$.}
	\end{equation}	 
 Namely, there exists a measurable set $\Omega_{r, \e} \subset \U(\R^n)$ so that $\p_{B_r^c}(\Omega_{r, \e})=1$ and for any $\eta \in \Omega_{r, \e}$, it holds that 
\begin{align} \label{ineq: MB}
\rho_{\U(B_{r})}^1\Bigl( (\partial_{r}^*E)_{\eta, r} \setminus  \partial_{r+\e}^*E)_{\eta, r} \Bigr)=0 \, .
\end{align}
\end{lem}
\begin{proof}
By \eqref{z3} and the definition \eqref{defn: PME} of the perimeter measure $\|E\|_r$, we see that 
\begin{align} \label{eq: OCE}
	\infty>\|E\|(A) \ge \|E\|_{r+\e}(A)  \notag
	&= \int_{\U(B_{r+\e}^c)} \rho_{\U(B_{r+\e})}^1(\partial_{\U(B_{r+\e})}^*E_{\eta, r+\e} \cap A_{\eta, r+\e}) d\pi_{B_{r+\e}^c}(\eta)  \notag
	\\
	&= \int_{\U(B_{r+\e}^c)} \rho_{\U(B_{r+\e})}^1\bigl((\partial_{r+\e}^*E)_{\eta, r+\e} \cap A_{\eta, r+\e}\bigr) d\pi_{B_{r+\e}^c}(\eta)  \notag
	\\
	&= \int_{\U(B_{r+\e}^c)} \rho_{\U(B_{r+\e})}^1\bigl((\partial_{r+\e}^*E \cap A)_{\eta, r+\e}\bigr) d\pi_{B_{r+\e}^c}(\eta) \, .
\end{align}
By the monotonicity $\|E\|_{r+\e}(A) \ge \|E\|_r(A)$ in Corollary \ref{cor:mon}, we obtain that 
$$ \int_{\U(B_{r+\e}^c)} \rho_{\U(B_{r+\e})}^1\bigl((\partial_{r+\e}^*E \cap A)_{\eta, r+\e}\bigr) d\pi_{B_{r+\e}^c}(\eta) \ge  \int_{\U(B_{r}^c)} \rho_{\U(B_{r})}^1\bigl((\partial_{r}^*E \cap A)_{\eta, r}\bigr) d\pi_{B_{r}^c}(\eta).$$
Taking $A=\U(\R^n) \setminus \partial_{r+\e}^*E$, we have that 
\begin{align*}
	0&=\int_{\U(B_{r+\e}^c)} \rho_{\U(B_{r+\e})}^1\bigl((\partial_{r+\e}^*E \cap A)_{\eta, r+\e}\bigr) d\pi_{B_{r+\e}^c}(\eta) 
	\\
	&\ge  \int_{\U(B_{r}^c)} \rho_{\U(B_{r})}^1\bigl((\partial_{r}^*E \cap A)_{\eta, r}\bigr) d\pi_{B_{r}^c}(\eta)\, .
	% \\
\end{align*}
Thus, $\rho_{\U(B_{r})}^1\bigl((\partial_{r}^*E \cap A)_{\eta, r}\bigr)=0$ for $\pi_{B_r^c}$-a.e.\ $\eta$, which implies that 
$$\bigl( \partial_{r}^*E \cap (\U(\R^n) \setminus \partial_{r+\e}^*E)\bigr)_{\eta, r} = (\partial_{r}^*E)_{\eta, r} \setminus \bigl( (\partial_{r+\e}^*E)_{\eta, r} \cap (\partial_{r}^*E)_{\eta, r} \bigr)$$ is $\rho_{\U(B_{r})}^1$-negligible for $\pi_{B_r^c}$-a.e.\ $\eta$. 
\end{proof} 

%We now start the proof of Theorem \ref{thm: fpr}.
%\\

\begin{proof}[Proof of Theorem \ref{thm: fpr}]
Fix $r>0$ and $\eta\in \U(B_r^c)$. It holds
\begin{equation}
	(\partial^* E)_{\eta,r} : = \left(\bigcup_{i>0} \bigcap_{j>i} \partial_j^* E \right)_{\eta,r} 
	=
	\bigcup_{i>0} \bigcap_{j>i} (\partial_j^* E)_{\eta,r} \, .
\end{equation}
The monotonicity formula \eqref{ineq: MB} in Lemma \ref{lem: MNB} gives the existence of $\Omega_{r, j} \subset \U(\R^n)$ so that $\pi_{B_r^c}(\Omega_{r, j})=1$, and for any $\eta \in \Omega_{r, j}$
 \begin{equation*}
	(\partial_r^* E)_{\eta,r} \subset 	(\partial_j^* E)_{\eta,r} \quad j \ge r
	\quad
	\text{up to a $\rho_{\U(B_r)}^1$-negligible set}\, .
\end{equation*}
Take $\Omega_r = \cap_{j \ge r, j \in \N} \Omega_{r, j}$. 
Then $\pi_{B_r^c}(\Omega_r)=1$, and by using \eqref{z3}, we obtain that for any $\eta \in \Omega_r$, 
\begin{equation*}
	\partial_{\U(B_r)}^* E_{\eta,r} = (\partial_r^*E)_{\eta,r}
	\subset (\partial^*E)_{\eta,r}
	\quad
	\text{up to a $\rho_{\U(B_r)}^1$-negligible set}\, .
\end{equation*}
This implies that for any Borel set $A \subset \U(\R^n)$, 
\begin{align*} %\label{ineq: MB-1}
 \rho^1_{\U(B_r)}(\partial_{\U(B_r)}^* E_{\eta,r}\cap A_{\eta,r}) \le  \rho^1_{\U(B_r)}((\partial^*E\cap A)_{\eta,r}), \quad \eta \in \Omega_r.
 %\rho_{\U(B_{r})}^1\Bigl( (\partial_{r}^*E)_{\eta, r} \setminus \bigl( (\partial^*E)_{\eta, r} \cap (\partial_{r}^*E)_{\eta, r} \bigr) \Bigr)=0.
\end{align*}
Thus, by noting that $\pi_{B_r^c}(\Omega_r)=1$ and recalling Definition \ref{defn: PM}, Definition \ref{defn: MCH},  we obtain 
\begin{align*}
	\| E \| (A)  & := \lim_{r \to \infty} \|E\|_r(A)
	\\&
	= \lim_{r \to \infty} \int_{\U(B_r^c)} \rho^1_{\U(B_r)}(\partial_{\U(B_r)}^* E_{\eta,r}\cap A_{\eta,r}) d \p_{B_r^c}(\eta)
	\\&
	\le 
	\lim_{r \to \infty} \int_{\U(B_r^c)} \rho^1_{\U(B_r)}((\partial^*E\cap A)_{\eta,r}) d \p_{B_r^c}(\eta)
	\\&
	= \rho^1(A\cap \partial^* E) \, .
\end{align*}

In order to conclude the proof,  it is enough to check that
\begin{equation} \label{ineq: OIER}
	\| E \|(\U(\R^n)) \ge \rho^1(\partial^*E) \, .
\end{equation}
Indeed, given any Borel set $A$, by making use of the already proven inequality $\| E \| \le \rho^1|_{\partial^* E}$, we obtain 
\begin{equation*}
	\| E \|(\U(\R^n)) = \| E \|(A) + \|E\|(A^c) \le \rho^1(A\cap \partial^*E) + \rho^1(A^c\cap \partial^*E)
	=\rho^1(\partial^*E) \le \| E \|(\U(\R^n)) \, .
\end{equation*}
Thus, $\|E \|(A) + \|E\|(A^c) = \rho^1(A\cap \partial^*E) + \rho^1(A^c\cap \partial^*E)$ for any Borel set $A$. Assume that there exists a Borel set $A$ so that $\|E\|(A)<\rho^1(A\cap \partial^*E)$. Since  $\| E \| \le \rho^1(\cdot \cap \partial^* E)$, it implies 
$$\|E \|(A) + \|E\|(A^c) < \rho^1(A\cap \partial^*E) + \rho^1(A^c\cap \partial^*E),$$
 which is a contradiction.

We now prove \eqref{ineq: OIER}. Let $s<r$. By recalling Definitions \ref{defn: PME0}, \ref{defn: LFP} of $\|E\|_r$ and $\rho_r^1$ respectively and using the monotonicity of $\rho_r^1$ in Theorem \ref{thm: MI}, we have
\begin{align*}
   \|E\|_r(\U(\R^n))
	=\int_{\U(B_r^c)} \rho^1_{\U(B_r)}((\partial_r^*E)_{\eta,r}) d \p_{B_r^c}(\eta)
	= \rho^1_r(\partial_r^*E)
	\ge  \rho^1_s(\partial_r^*E) \, ,
\end{align*}
hence
\begin{align*}
	\| E \| (\U(\R^n)) = \lim_{i \to \infty} \|E\|_i(\U(R^n))
	\ge \liminf_{i \to \infty} \rho^1_s(\partial_i^*E)
	\ge
	\rho_s^1(\liminf_{i\to \infty}\partial_i^*E)
	= \rho_s(\partial^*E)
	 \, .
\end{align*}
Passing to the limit $s\to \infty$, we conclude \eqref{ineq: OIER}. 
\end{proof}

\section{Total variation and Gau{\rm\ss}--Green formula}  \label{sec: CF}

In this section, we prove a relation between the coarea with respect to the perimeter measure $\|E\|$ and the variation $|\DDD_{*}F|$ obtained via relaxation of Cylinder functions. As an application, we introduce the total variation {\it measure} $|\DDD F|$ for BV functions $F$, and prove the Gau\ss--Green formula.
%the coarea formula, which provide a useful way to construct sets of finite perimeters from BV functions. 
%As an application, we introduce the total variation measure in Definition \ref{defn: TVM}. The coarea formula will also play a key role for proving the Gua\ss--Green formula in Subsection \ref{sec: GG}. 

\subsection{Total variation measures via coarea formula} \label{sec:coarea}
Recall that, for $F \in {\rm BV}(\U(\R^n))$, the map $\CylF(\U(\R^n)) \ni G \mapsto |\DDD_*F|[G]$ is defined by the relaxation approach in Definition \ref{defn: BVrelaxation}.
The main result of this subsection is the following formula:
\begin{thm}\label{thm:coarea}
	Let $F\in L^2(\U(\R^n), \p)\cap {\rm BV}(\U(\R^n))$. Then, 
	\begin{align} \label{eq:FV}
	&\mathcal V(\chi_{\{F>t\}})<\infty \quad \text{a.e.}\ t \in \R,
	\end{align} 
and the following formula holds:
	\begin{align}\label{eq:coarea}
	& \int_{-\infty}^{\infty} \left( \int_{\U(\R^n)} G d\bigl\| {\{ F > t\}} \bigr\| \right) d t
		= 
		|\DDD_* F|[G],
		\quad 
		\text{for any non-negative $G\in \CylF(\U(\R^n))$} \, .
	\end{align}
\end{thm}

The proof of Theorem \ref{thm:coarea} will be given later in this section. Before discussing the proof, we study several consequences of Theorem \ref{thm:coarea}.
By \eqref{eq:FV}, the left-hand side of \eqref{eq:coarea} makes sense with $G\equiv 1$ since the right-hand side $|\DDD_*F|[1]<\infty$ is finite due to $F \in {\rm BV}(\U(\R^n))$ and Theorem \ref{prop: IBV}. This leads us to provide the following definition of the total variation measure. 
%map $G \mapsto |\DDD F|[G]$ turns out to be represented as $\int_{\U(\R^n)} G d |D F|$ where
%$$
%|\DDD F| := \int_{-\infty}^\infty \| \{F>t\}\| d t \, .
%$$

%\begin{prop}\label{prop: FTV}
%If $F \in {\rm BV}(\U(\R^n))$, then $\chi_{\{F>t\}} \in {\rm BV}(\U(\R^n))$ for a.e.\ $t \in [0,\infty]$.
%\end{prop}

\begin{defn}[Total variation measure]  \label{defn: TV} \normalfont
For $F \in L^2(\U(\R^n), \p)\cap{\rm BV}(\U(\R^n))$, define {\it the total variation measure $|\DDD F|$} as follows:
\begin{align}\label{eq:coarea1}
	 |\DDD F|:=\int_{-\infty}^{\infty} \ \bigl\| {\{ F > t\}} \bigr\|  d t.
\end{align}
\end{defn}

%{\color{blue}
%EB: I suggest that we move this explanation on the introduction.	
%}
%\begin{rem}\normalfont
%Our approach is different from the standard one in the sense that we define the total variation measure $|\DDD F|$ by using the coarea  \eqref{eq:coarea1} of the perimeter measure $\|\{F>t\}\|$. Instead, in the standard approach,  one first constructs both of $|\DDD F|$ and $\|\{F>t\}\|$ and show the coarea formula. The reason for this non-standard treatment is due to the well-definedness of the measure $\|E\|_r$ in Lemma \ref{lem: BV2-1}, part of which relies on the expression of $\|E\|_r$ in terms of the one-codimensional Poisson measure $\rho^1$. We do not have such an expression for the total variation measure $|\DDD F|$, which makes it difficult to construct a well-defined measure $|\DDD F|$ in a similar way to $\|E\|$. Our approach is, however, consistent with the standard one, which will be seen in Corollary \ref{cor: PMTV} and in the Gau\ss--Green formula in Theorem \ref{thm: BV}. %independently of the coarea and prove the coarea formula. 
%\end{rem}

% restriction of the finite Borel measure in the left-hand side of \eqref{eq:coarea}. Thus, we can extend $|\DDD F|(A)$ to any Borel set $A \subset \U(\R^n)$ by adopting the left-hand side of \eqref{eq:coarea} as the definition of $|\DDD F|(A)$. 

%\begin{defn}[Total variation measure] \label{defn: TVM} \normalfont
%Let $F: \U(\R^n) \to \R$ be a function with $|\DDD F|(\U(\R^n))<\infty$. The Borel measure $|\DDD F|$ on $\U(\R^n)$ is called {\it total variation measure}.
%\end{defn}

We now investigate relations between the total variation measure $|\DDD\chi_E|$ and the perimeter measure $\|E\|$ defined in Definition \ref{defn: PM} and the $(1,2)$-capacity ${\rm Cap}_{1,2}$ defined in Definition \ref{defn: BC}.  
\begin{cor}[Total variation and perimeters] \label{cor: PMTV}
	%If $|D h|(\U(\R^n)) < \infty$, then the map $A \to |D h|(A)$ is the restriction of a finite measure to open sets. With a slight abuse of notation \eqref{eq:coarea} can be extended to any Borel set.
Let $E \subset  \U(\R^n)$ satisfy $|\DDD \chi_E|(\U(\R^n))<\infty$. %$\mathcal V(\chi_E)<\infty$. 
Then, $$|\DDD \chi_E| = \|E\| \quad \text{as measures} \, .$$
\end{cor}
\begin{proof}
By Theorem \ref{prop: IBV}, $\mathcal V(\chi_E)<\infty$ and $\|E\|$ is well-defined. Noting that 
\begin{align*}
\{ \chi_E > t\} =
\begin{cases}
\U(\R^n) \quad &t \le 0;
\\
E \quad &0<t \le 1;
\\
\emptyset \quad & t>1,
\end{cases}
\end{align*}
and $\|\U(\R^n)\|=0$ and $\|\emptyset\|=0$, 
we obtain that 
	\begin{equation*}
		|\DDD \chi_E|(A)  = \int_{-\infty}^{\infty} \bigl\| {\{ \chi_E > t\}} \bigr\|(A) d t 
		=  0 + \|E\|(A) +0
		= \|E\|(A)
		\quad 
		\text{for every Borel set A} \, . \qedhere
	\end{equation*}
\end{proof}

\begin{cor}[Total variation and capacity]\label{cor: measure}
Let $F \in L^2(\U(\R^n), \p)\cap{\rm BV}(\U(\R^n))$. 
For any Borel set $A\subset \U(\R^n)$, 
$${\rm Cap}_{1,2}(A)=0 \implies |\DDD F|(A) =0.$$
\end{cor}

\begin{proof}
	Let ${\rm Cap}_{1,2}(A)=0$. By Theorem \ref{thm:coarea}, and Theorem \ref{thm: fpr}, we can write
	\begin{equation}
		|\DDD F|(A)
		=
		\int_{-\infty}^{\infty} \| {\{ F > t\}} \|(A) d t
		= 
		\int_{-\infty}^{\infty} \rho^1(\partial^*\{F>t\}\cap A) d t \, ,
	\end{equation}
	hence it suffices to show that $\rho^1(\partial^* \{F>t\}\cap A)=0$. This follows from the absolute continuity of $\rho^1$ with respect to ${\rm Cap}_{1,2}$ obtained in Theorem \ref{thm: CH}. 
\end{proof}

%\begin{prop} \label{prop: BVC}
%	Let $w \in BV(\U) \cap L^p(\U, \p)$ for $p>1$. Then, 
%	$$\text{$|Dw|(A) =0$ whenever ${\rm Cap}_{1,2}(A)=0$}.$$ 
%\end{prop}
%\proof
%Let ${\rm Cap}_{1,2}(A)=0$. By the coarea formula obtained by Theorem \ref{thm:coarea} and Corollary \ref{cor: measure}, we can express 
%\begin{equation}\label{eq:coarea}
%	\int_{-\infty}^{\infty} \| {\{ w > t\}} \|(A) d t
%	= 
%	|D w|(A).
%\end{equation}
%Since $\|w>t\|$ can be expressed by the one-codimensional measure $\rho^1$ restricted on the reduced boundary $\partial^* \{w>t\}$ by Theorem \ref{thm: BV}, it suffices to show that $\rho^1(A \cap \partial^* \{w>t\})=0$. This follows from the absolute continuity of $\rho^1$ with respect to ${\rm Cap}_{1,2}$ obtained by Theorem \ref{thm: CH}.  We finished the proof.
%\qed
%

	%\footnote{Why do we need the non-negativity. Compare the corresponding statement \eqref{z9}.
	%EB: it is fundamental, the statement is clearly false for $F$ that changes sign. Also in \eqref{z9} we used the non-negative assumption on $F$ (I modified the statement accordingly)} 
	
\subsection{Proof of Theorem \ref{thm:coarea}}
This subsection is devoted to the proof of Theorem \ref{thm:coarea}. Let us begin with two propositions.

   \begin{prop}\label{prop: IBPV}
	Let $E\subset \U(\R^n)$ be a set with $\mathcal V(\chi_E)<\infty$. Then, for every non-negative
	function $G\in \CylF(\U(\R^n))$ it holds
	\begin{equation}\label{eq:supcaracterisation}
		\int_{\U(\R^n)} G d \| E \|
		=
		\sup\left\lbrace
		\int_{E} (\nabla^* G V) d \pi \, : \, V\in \CylV(\U(\R^n)),\, |V|_{T\U}\le 1
		\right\rbrace
		\, .
	\end{equation}
    In particular, the following hold:
    \begin{itemize}
    	\item[(i)] if $F_k\in \CylF(\U(\R^n))$, and $F_k\to \chi_E$ in $L^1(\U(\R^n),\pi)$ as $k\to \infty$, then
    	\begin{equation*}\label{eq:lsc}
    		\liminf_{k\to \infty} \int_{\U(\R^n)} G |\nabla F_k|_{T\U} d \pi \ge \int_{\U(\R^n)} G d \|E \| \, ,
    		\quad \text{for non-negative $G\in \CylF(\U(\R^n))$} \, ;
    	\end{equation*}
        \item[(ii)] if $\chi_{E_k}\to \chi_E$ in $L^1(\U(\R^n),\pi)$ as $k\to \infty$, where $(E_k)_k$ are sets of finite perimeter, then
        \begin{equation*}\label{eq:lsc sfp}
        	\liminf_{k\to \infty} \int_{\U(\R^n)} G d \| E_k \| \ge \int_{\U(\R^n)} F d \|E \| \, ,
        	\quad \text{for  non-negative $G\in \CylF(\U(\R^n))$} \, .
        \end{equation*}
    \end{itemize}
\end{prop}

\begin{proof}
	Fix $\e>0$. We pick $r>0$ such that $\int_{\U(\R^n)} G d\| E \|_r \ge \int_{\U(\R^n)} G d\| E \| - \e$.
	From \eqref{z9} we deduce the existence of $V\in \CylV^r_*(\U(\R^n))$ with $|V|_{T\U} \le 1$ such that $\int_E (\nabla^* GV) d\p \ge \int_{\U(\R^n)} G d\| E \|_r - \e$, yielding
	\begin{equation*}
		\int_{\U(\R^n)} G d\| E \|
		\le
		\int_E (\nabla^* GV) d\p + 2\e
		\le
		\sup\left\lbrace
		\int_{E} (\nabla^* G V) d \pi \, : \, V\in \CylV(\U(\R^n)),\, |V|_{T\U}\le 1
		\right\rbrace + 2\e \, .
	\end{equation*}
   By taking $\e \to 0$, the one inequality is proved.  
   
We now prove the converse inequality. Take a representative $G=\Phi(f_1^\ast,\ldots, f_k^\ast)$ and take $r>0$ so that $\cup_{i=1}^k{\rm supp}[f_i] \subset B_r$. By the divergence formula \eqref{eq: DIV}, we can easily see 
\begin{align*}
 \sup\left\lbrace \int_{E} (\nabla^* G V)  d\p \, : \, V\in \CylV^r_*(\U(\R^n)),\, |V|_{T\U}\le 1 \right\rbrace   
\\
= \sup\left\lbrace \int_{E} (\nabla^* G V)  d\p \, : \, V\in \CylV(\U(\R^n)),\, |V|_{T\U}\le 1 \right\rbrace \, .
\end{align*}
By combining it with the formula \eqref{z9} and the monotonicity of $r \mapsto \|E\|_r$ in Corollary \ref{cor:mon}, the converse inequality is proved.

    \medskip

	Let us now prove (i) and (ii). Fix $\e>0$. By Theorem \ref{thm: GGP}, 
	we can take $V\in \CylV(\U(\R^n))$ such that $|V|_{T\U}\le 1$ and 
	\begin{equation*}
		\int_{E} (\nabla^* G V) d \pi 
		\ge \int_{\U(\R^n)} G d \|E \| - \e \, .
	\end{equation*}
    Let $k_j$ be a subsequence such that $\lim_{j\to \infty} \int_{\U(\R^n)} G |\nabla F_{k_j}|_{T\U} d \pi = \liminf_{k\to \infty}\int_{\U(\R^n)} G |\nabla F_{k}|_{T\U} d \pi $,
	it holds
	\begin{align*}
		\int_{\U(\R^n)} G d \|E \| - \e & \le \int_{E} (\nabla^* G V) d \pi  = \lim_{j\to \infty} \int_{\U(\R^n)} F_{k_j} (\nabla^* G V) d \pi =  \lim_{j\to \infty} \int_{\U(\R^n)} G \langle \nabla F_{k_j}, V\rangle_{T\U} d \pi
		\\& \le \liminf_{k\to \infty} \int_{\U(\R^n)} G |\nabla F_k|_{T\U} d \pi \, .
	\end{align*}
Furthermore, by using Theorem \ref{thm: GGP} with $V$ being $GV$, we deduce that 
    \begin{align*}
    	\int_{\U(\R^n)} G d \|E \| - \e & \le \int_{E} (\nabla^* G V) d \pi  = \lim_{j\to \infty} \int_{E_{k_j}} (\nabla^* G V) d \pi =  \lim_{j\to \infty} \int_{\U(\R^n)} G \langle  V,\sigma_{E_{k_j}}\rangle_{T\U} d \| E_{k_j}\|
    	\\& \le \liminf_{k\to \infty} \int_{\U(\R^n)} G d \| E_{k_j}\| \, . \qedhere
    \end{align*}
\end{proof}

%The main ingredient of Theorem \ref{thm:coarea} is the following.
\begin{prop}\label{prop:coarea}
	For any $F \in \CylF(\U(\R^n))$ it holds
	\begin{equation}
		\int_{-\infty}^{\infty}\int_{\U(\R^n)} G\, d \bigl\| {\{F>t\}} \bigr\| d t
		=
		\int_{\U(\R^n)} G\,  |\nabla F|_{T\U} d \p \, ,
		\quad \text{for non-negative $G\in \CylF(\U(\R^n))$} \, .
	\end{equation}
%	In particular
%	\begin{equation}
%		\int_{-\infty}^{\infty} \bigl\| {\{G>t\}} \bigr\|(A) d t
%		=
%		\int_{A}  |\nabla G|_{T\U} d \p \, ,
%		\quad \text{for any Borel set $A$} \, .
%	\end{equation}
\end{prop}

\begin{proof}
	The map
	\begin{equation}
		\R\ni 	t \to m(t):= \int_{\{F> t\}} G |\nabla F|_{T\U} d \p
	\end{equation}
	is monotone and finite since $|\nabla F|_{T\U}\in L^1(\U(\R^n))$. Let $t\in \R$  be a point on which the map $t \mapsto m(t)$ is differentiable and set
	\begin{equation}
		g_\e(s) :=
		\begin{cases}
			1 & s\le t \\
			\e^{-1}(t - s) + 1		& t\le s \le t+\e \\
			0 & s> t+\e \, .
		\end{cases}
	\end{equation}
	Notice that $g_\e \circ F \to \chi_{\{F>t\}}$ in $L^p(\U(\R^n))$ for any $p\in [1,\infty)$ as $\e\to 0$.
	Indeed, 
	\begin{equation}
		\int_{\U(\R^n)} | g_\e \circ F - \chi_{\{F>t\}}|^p d
		\p
		\le 2^p \p(\{t\le F \le t+\e\}) \to 0, 
		\quad \text{as $\e\to 0$} \, .
	\end{equation}
	Standard calculus rules give
	\begin{equation}
		\int_{\U(\R^n)} G | \nabla (g_\e \circ F)|_{T\U} d \pi 
		\le
		\e^{-1} \int_{\{t<F\le t+\e\}}G |\nabla F|_{T\U} d \p
		\le \frac{m(t+\e) - m(t)}{\e} \, ,
	\end{equation}
	while \eqref{eq:lsc} in Proposition \ref{prop: IBPV} implies
	\begin{equation}\label{z5}
		\int_{\U(\R^n)} G\, d\| {\{F>t\}} \| 
		\le \liminf_{\e\to 0} \int_{\U(\R^n)} G | \nabla (g_\e \circ F)|_{T\U} d \pi 
		= m'(t) \, .
	\end{equation}
	Since $m$ is differentiable for a.e. $t\in \R$, the one inequality  comes by integrating \eqref{z5}.

	\bigskip
	
	Let us prove the converse inequality. Let $V\in \CylV(\U(\R^n))$ such that $| V |_{T\U}\le 1$. Then, by Theorem \ref{thm: GGP},  we deduce
	\begin{align*}
		\int_{\U(\R^n)} F(G \nabla^*  V) d \pi 
		& =
		\int_{-\infty}^\infty \int_{\{F>t\}}  (G \nabla^*  V) d \pi d t
		\\&
		\le \int_{-\infty}^\infty \int_{\U(\R^n)} G d \| \{F>t\}\| d t \, ,	
	\end{align*}
	which easily yields the sought conclusion.
\end{proof}

\begin{proof}[Proof of Theorem \ref{thm:coarea}]
	Let $F\in L^2(\U(\R^n),\p)$ such that $|\DDD F|(\U(\R^n))< \infty$ and $G\in \CylF(\U(\R^n))$ be non-negative. By definition there exists a sequence $(F_n)\subset \Cyl$ such that $F_n\to F$ in $L^1(\U(\R^n),\p)$ and $\int_{\U(\R^n)} G |\nabla F_n|_{T\U} d\pi \to |\DDD_* F|[G]$. From Proposition \ref{prop:coarea} we get
	\begin{equation}\label{z10}
		\int_{-\infty}^{\infty} \int_{\U(\R^n)} G d\| {\{F_n >t\}} \| d t
		= 
		\int_{\U(\R^n)} G |\nabla F_n|_{T\U} d \p \, ,
	\end{equation}
    and passing to the limit for $n\to \infty$ we deduce
    \begin{equation}
    	\int_{-\infty}^{\infty} \int_{\U(\R^n)} G d\| {\{F >t\}} \| d t
    	\le 
    	|\DDD_* F|[G] \, ,
    \end{equation}
    as a consequence of (ii) in Proposition \ref{prop: IBPV} and Fatou's Lemma. In particular $\{F>t\}$ is of finite perimeter for a.e.-$t\in \R$. 
    
    Let us now fix $\e>0$ and consider $V\in \CylV(\U(\R^n))$ such that $| V |_{T\U}\le 1$ and $\V(F) - \e 
    \le 
    \int_{\U(\R^n)} F (\nabla^* V) d \pi$. By Theorem \ref{thm: GGP}, we have
    \begin{align*}
    	|\DDD_* F|(\U(\R^n)) -\e
    	 & = 
    	\V(F) - \e
    	\le 
    	\int_{\U(\R^n)} F (\nabla^* V) d \pi 
    	 =
    	\int_{-\infty}^\infty \int_{\{ F>t\}}  (\nabla^* V) d \pi d t
    	\\&
    	\le \int_{-\infty}^\infty \int_{\U(\R^n)}  d \| \{ F>t\}\| d t \, ,	
    \end{align*}
    which easily yields 
    \begin{equation*}
    	\int_{-\infty}^\infty \int_{\U(\R^n)}  d \| \{F>t\}\| d t
    	\ge 
    	|\DDD_* F|(\U(\R^n)) = |\DDD_* F|[1] \, .
    \end{equation*}
   The sought conclusion follows now by recalling that $|\DDD_* F|[G_1 + G_2] \ge |\DDD_* F|[G_1] + |\DDD_* F|[G_2]$ and by the same argument in the paragraph after \eqref{ineq: OIER}.  Indeed, 
  % {\color{red}
  %the inequality  $|\DDD_* F|[G_1 + G_2] \ge |\DDD_* F|[G_1] + |\DDD_* F|[G_2]$  can be checked in the following argument}:
   \begin{align*}
   	|\DDD_* F|[G] + |\DDD_* F|[1-G]
   	 &\le
   	|\DDD_* F|[1] 
   	\le \int_{-\infty}^\infty \int_{\U(\R^n)}   d \| \{ F>t\}\| d t 
   	\\& =
   	\int_{-\infty}^\infty \int_{\U(\R^n)} G   d \| \{ F>t\}\| d t
   	+
   	\int_{-\infty}^\infty \int_{\U(\R^n)} (1-G)  d \| \{ F>t\}\| d t  
   	\\& \le |\DDD_* F|[G] + |\DDD_* F|[1 - G] \, ,
   \end{align*}
    for any $0\le G\le 1$, $G\in \CylF(\U(\R^n))$. 
    %		Let $f\in L^1(\U(\R^n),\p)$ such that $|Df|(\U(\R^n))< \infty$ and $A\subset \U(\R^n)$ be an open set. By definition there exists a sequence $(f_n)\subset \Cyl$ such that $f_n\to f$ in $L^1(A,\p)$ and $\int_A |\nabla f_n|_{T\U} d\pi \to |Df|(A)$. From Proposition \ref{prop:coarea} we know that
%	\begin{equation}\label{z10}
%		\int_{-\infty}^{\infty}\| {\{f_n >t\}} \|(A) d t
%		=
%		\int_{A}  |\nabla f_n |_{T\U} d \p \, .
%		%\quad \text{for any Borel set $A$} \, .
%	\end{equation}	
%	The sought conclusion comes	by passing to the limit \eqref{z10} as $n\to \infty$, we refer the reader to \cite[Theorem 5.9]{EvansGariepy} for more details.	\purple{We need the lower semi-continuity $\|E\| \le \liminf_{n \to \infty}\|E_n\|$ to apply the same argument as in \cite[Theorem 5.9]{EvansGariepy}.}
\end{proof}

\subsection{Gau\ss--Green formula} \label{sec: GG}
We prove the Gau\ss--Green formula. 
For $F \in L^2(\U(\R^n), \p)\cap{\rm BV}(\U(\R^n))$,  let $L^2(T\U, |\DDD F|)$ denote the completion of $\CylV(\U)$ with respect to $\|\cdot\|_{L^2(T\U,  |\DDD F|)}$ analogously in~\eqref{d:VFN}.
    	\begin{thm}[Gau\ss--Green formula] \label{thm: BV}
    		For $F \in L^2(\U(\R^n), \p)\cap{\rm BV}(\U(\R^n))$, there exists a unique element $\sigma_F \in L^2(T\U, |\DDD F|)$  such that $|\sigma_F|_{T\U}=1$ $|\DDD F|$-a.e., and   
    		\begin{align} \label{eq: IP1}
    			\int_{\U(\R^n)} (\nabla^* V) F d\p = \int_{\U(\R^n)} \langle V, \sigma_F \rangle_{T\U} d |\DDD F|, \quad \forall V \in \CylV(\U(\R^n)) \, .
    		\end{align} 
%    		If $f=\chi_{E}$, we denote $|D f| =\| E \|$. Then, $\| E \| = \rho^1|_{\partial_\ast A}.$
    	\end{thm}

%\subsection{Proof of Theorem \ref{thm: BV}}
\proof
	We assume without loss of generality that $|\DDD F|(\U(\R^n)) =1$. 	
	By Theorem \ref{thm: GGP} and Theorem \ref{thm:coarea}, it holds that 
	\begin{align} %\label{eq: CF}
		\int_{\U(\R^n)} (\nabla^* V) F d\p &=  
		\int_{-\infty}^\infty \int_{\{F>t\}} (\nabla^* V) d \pi d t \notag
		\\& 
		= \int_{-\infty}^\infty \int_{\U(\R^n)} \langle V, \sigma_{\{F>t\}}\rangle_{T\U} d \bigl\|\{F>t\} \bigr\| d t \notag
		\\&
		\le \int_{\U(\R^n)} |V|_{T\U} d | \DDD F | \notag
		\\&
		\le \| V \|_{L^2(T\U, |\DDD F|)}, \notag
	\end{align} 
    for every $V\in \CylV(\U)$. In particular, the map $L$ defined by 
		\begin{equation}
		L :  L^2(T\U, |\DDD F|) \to \R \, , \quad
		 L^2(T\U, |\DDD F|) \ni V \mapsto L(V):= \int_{\U(\R^n)} (\nabla^* V) F d\p \, ,
	\end{equation}
    is a well-defined continuous operator on the Hilbert space $L^2(T\U, |\DDD F|)$ and satisfies $\| L \| \le 1$.
    Therefore, the Riesz representation theorem on the Hilbert space $ L^2(T\U, |\DDD F|)$ gives the existence of $\sigma_F \in  L^2(T\U, |\DDD F|)$ so that 
    \begin{equation}%\label{z12}
    	\|\sigma_F\|_{L^2(T\U, |\DDD F|)} \le 1 \, ,
    	\quad \quad
    	\int_{\U(\R^n)} (\nabla^* V) F d\p = \int_{\U(\R^n)} \langle V, \sigma_F \rangle d|\DDD F| 
    	\quad  V\in \CylV(\U(\R^n)) \, . \notag
    \end{equation}
    From Theorem \ref{prop: IBV} and Theorem \ref{thm:coarea},  we deduce
    \begin{align*}
    1 &= |\DDD F|(\U(\R^n)) = |\DDD_* F|[1]=\V(F) = \sup_{V\in \CylV\, ,|V|_{T\U}\le 1} \int_{\U(\R^n)} (\nabla^* V) F d\p 
    \\
    &\le \int_{\U(\R^n)} |\sigma_F|_{T\U} d|\DDD F|
    \le \| \sigma_F \|_{L^2(T\U, |\DDD F|)} \le 1
     \, ,
    \end{align*}
    which yields $\| \sigma_F \|_{L^1(T\U, |\DDD F|)} = \| \sigma_F \|_{L^2(T\U, |\DDD F|)} = 1$, and therefore $|\sigma_F|_{T\U} =1$ $|\DDD F|$-a.e.\ as a consequence of the characterisation of the equality in Jensen's inequality.
    \qed

\subsection{${\rm BV}$ and Sobolev functions}
In this subsection, we discuss the consistency of the just developed theory of BV functions with the $(1,2)$-Sobolev space $H^{1,2}(\U(\R^n), \p)$.

\begin{prop} \label{prop: WB}
Let $F \in L^2(\U(\R^n), \p)\cap{\rm BV}(\U(\R^n))$. Suppose $|\DDD F| \ll \p$ with $|\DDD F| = H \cdot \p$ and $H\in L^2(\U(\R^n), \p)$. Then $F \in H^{1,2}(\U(\R^n), \p)$ and 
$$
H=|\nabla F| \, ,
\quad\quad
\sigma_F = \frac{\nabla F}{|\nabla F|} \cdot \chi_{\{|\nabla F| \neq 0\}} \, ,
$$
where $\sigma_F$ is the unique element in $L^2(T\U, |\DDD F|)$ in the Gau\ss--Green formula \eqref{eq: IP1}.
\end{prop}
 \proof
 %We first prove the case of $p \ge 2$.
 By Theorem \ref{thm: BV} and recalling $\mathbf T_tV \in \mathcal D(\mathcal E_H) \subset  \mathbf D^2(T\U(\R^n), \p)$ for $V \in \CylV(\U(\R^n))$ by \eqref{eq: important inclusion2},  the approximation of $\mathbf T_tV$ by $\CylV(\U(\R^n), \p)$ implies that 
 \begin{align} \label{eq: WB1}
 \int_{\U(\R^n)} (\nabla^* G) F d\p = \int_{\U(\R^n)} \langle G, \sigma_F \rangle_{T\U} F d\p \quad \forall G \in \mathbf T_t\CylV(\U(\R^n)) \quad \forall t>0 \, ,
\end{align}
where $\mathbf T_t\CylV(\U(\R^n)):=\{G = \mathbf T_t F: F \in \CylV(\U(\R^n))\}$ for $t>0$. 
%Let $|Dw| = p\cdot d\p$. 
%Then, 
%\begin{align} \label{eq: WB2}
% \int_{\U(\R^n)} (\nabla^* G) w d\p = \int_{\U(\R^n)} \langle G, p\sigma_w \rangle_{T\U} d\p \quad \forall G \in \mathbf T_t\CylV(\U). 
%\end{align}
By Lemma \ref{lem: divheat} and the $\p$-symmetry of $T_t$, %and \eqref{eq: WB2}, 
for any $U\in \CylV(\U(\R^n))$, setting $G= \mathbf T_t U$,  we obtain 
\begin{align*} %\label{eq: WB3}
\int_{\U(\R^n)} \langle U, \nabla T_t F \rangle  d\p 
&= \int_{\U(\R^n)} (\nabla^*U)  T_t F   d\p   
= \int_{\U(\R^n)} T_t(\nabla^*U) F   d\p 
= \int_{\U(\R^n)} (\nabla^* G) F   d\p  
%\qquad  ( \because \text{Lemma \ref{lem: divheat}})
\\
&= \int_{\U(\R^n)} \langle G, \sigma_F \rangle_{T\U} d|\DDD F| 
= \int_{\U(\R^n)} \langle G, \sigma_F \rangle_{T\U} H d\p      
%\qquad  ( \because \text{\eqref{eq: WB2}})
 = \int_{\U(\R^n)} \langle U, \mathbf T_t(H \sigma_F) \rangle_{T\U} d\p \, .
 %=  - \int_{\U} \langle \mathbf T_t G, \nabla  w \rangle  d\p = - \int_{\U} \langle G, \nabla T_t w \rangle  d\p. %=  \int_{\U} \langle G, \mathbf T_t \nabla  w \rangle  d\p.
\end{align*}
Thus,  $\mathbf T_t (H \sigma_F) =\nabla T_t  F$. Letting $t \to 0$, $\mathbf T_t (H\sigma_F)$ converges to $H \sigma_F$ in $L^2(T\U, \p)$, which implies that $\nabla T_t  F$ converges to $H \sigma_F$ in $L^2(T\U(\R^n), \p)$. 
Since $T_t F \to F$ in $L^2(\U(\R^n), \p)$, we conclude that $F \in H^{1,2}(\U(\R^n), \p)$, and $\nabla F = H\sigma_F$.  Therefore, $H\cdot \p = |\DDD F| =  |\nabla F| \cdot \p$, and 
$$
\sigma_F = \frac{\nabla F}{H}\chi_{\{H \neq 0\}} =  \frac{\nabla F}{|\nabla F|} \chi_{\{|\nabla F| \neq 0\}} \, . \qedhere
$$ 

\end{document}